\documentclass[a4paper,10pt]{article}
\usepackage[dvipsnames]{xcolor}
\usepackage[utf8x,utf8]{inputenc}
\definecolor{myblue}{RGB}{61, 61, 130}
\include{formatAndDefs}
\usepackage{pgfornament}
\usepackage{tcolorbox}
\tcbuselibrary{theorems}
\usepackage{mathrsfs}
\usepackage{bbm}
\usepackage{graphicx}
\usepackage{mathtools}
\usepackage{pdfpages}
\usepackage{natbib}
\usepackage{amssymb,amsmath,amsfonts,amsxtra,amsthm}
\usepackage{float}
\usepackage[colorinlistoftodos]{todonotes}
\usepackage{graphics,graphicx}
\usepackage{caption}
\usepackage{enumerate}   
\usepackage{float}
\usepackage[lined,boxed,commentsnumbered, ruled,vlined,linesnumbered,onelanguage]{algorithm2e}
\usepackage{bbm}

\newtheorem{Theorem}{Theorem}
\newtheorem{Conditions}{Conditions}
\newtheorem{Corollary}{Corollary}[Theorem]
\newtheorem{Lemma}[Theorem]{Lemma}

\newtheorem{Fact}[Theorem]{Fact}

\providecommand{\keywords}[1]{\textbf{\textit{Index terms---}} #1}

\title{Exponential bounds for inhomogeneous random graphs in a Gaussian case}
\author{Othmane SAFSAFI \footremember{LPSM}{Sorbonne Université}\\
       \\}

\begin{document}
\begin{titlepage}
    \centering
    \vfill
    {\bfseries\Large
       Exponential bounds for inhomogeneous random graphs in a Gaussian case\\
		\textit{}
 		
        \vskip2cm
		 Othmane SAFSAFI
		 \vskip1cm
		 \small{LPSM-Sorbonne Université}
	}

    \vskip2cm

\begin{abstract}

Rank-1 inhomogeneous random graphs are a natural generalization of Erd\H{o}s-Rényi random graphs. In this generalization each node is given a weight. Then the probability that an edge is present depends on the product of the weights of the nodes it is connecting. In this paper, we give precise and uniform exponential bounds on the size, weight and surplus of rank-1 inhomogeneous random graphs where the weight of the nodes behave like a random variable with finite third moments. We focus on the case where the mean degree of a random node is equal to $1$ (critical regime), or slightly larger than $1$ (barely supercritical regime). These bounds will be used in  follow up papers to study a general class of random minimum spanning trees. They are also of independent interest since they show that these inhomogeneous random graphs behave like Erd\H{o}s-Rényi random graphs even in a barely supercritical regime. The proof relies on novel concentration bounds for sampling without replacement and a careful study of the exploration process.
\end{abstract}

\keywords{Random, Graphs, Inhomogeneous, Networks}

\end{titlepage}
\section{Introduction}
\subsection{The model}
Consider $n \in \mathbb{N}$ vertices labeled $1,2,...,n$. For a vector of weights $\textbf{W}=(w_1,w_2,..w_n)$, where $0< w_n \leq w_{n-1} \leq ... \leq w_1$, we create the inhomogeneous random graph associated to $\textbf{W}$ and to $p \leq +\infty$ in the following way: \par 
Each potential edge $\{i,j\}$ is in the graph with probability $1-e^{-w_iw_jp}$ independently from everything else. This gives a random graph that we call the rank-1 inhomogeneous random graph associated to $\textbf{W}$ and $p \leq +\infty$. \par
One can couple the graphs for the different values of $p$ as follow: 
Let $K_n$ be the complete graph of size $n$. 
To every potential edge  $\{i,j\}$, associate independently the random capacity $E_{\{i,j\}}$ which is an exponential random variable of rate $w_iw_j$. The weights are then used to create a sequence of graphs. For each $p \in [0,+\infty]$ let $G(\textbf{W},p)$ be the graph on $\{1,2....,n\}$ containing the edges of weight at most $p$. So the edge set of $G(\textbf{W},p)$ is:
$$\left\{ \{i,j\} |  E_{\{i,j\}} \leq p \right\}.$$
Then  $(G(\textbf{W},p))_{p \in [0,+\infty]}$ is an increasing sequence of graphs for inclusion, and for each fixed value of $p$, this constructions matches the first one. We will use both construction 
interchangeably in this paper. 

\begin{figure}[!htbp]
\centering
\includegraphics[width=0.9\textwidth, height=0.4
\textheight]{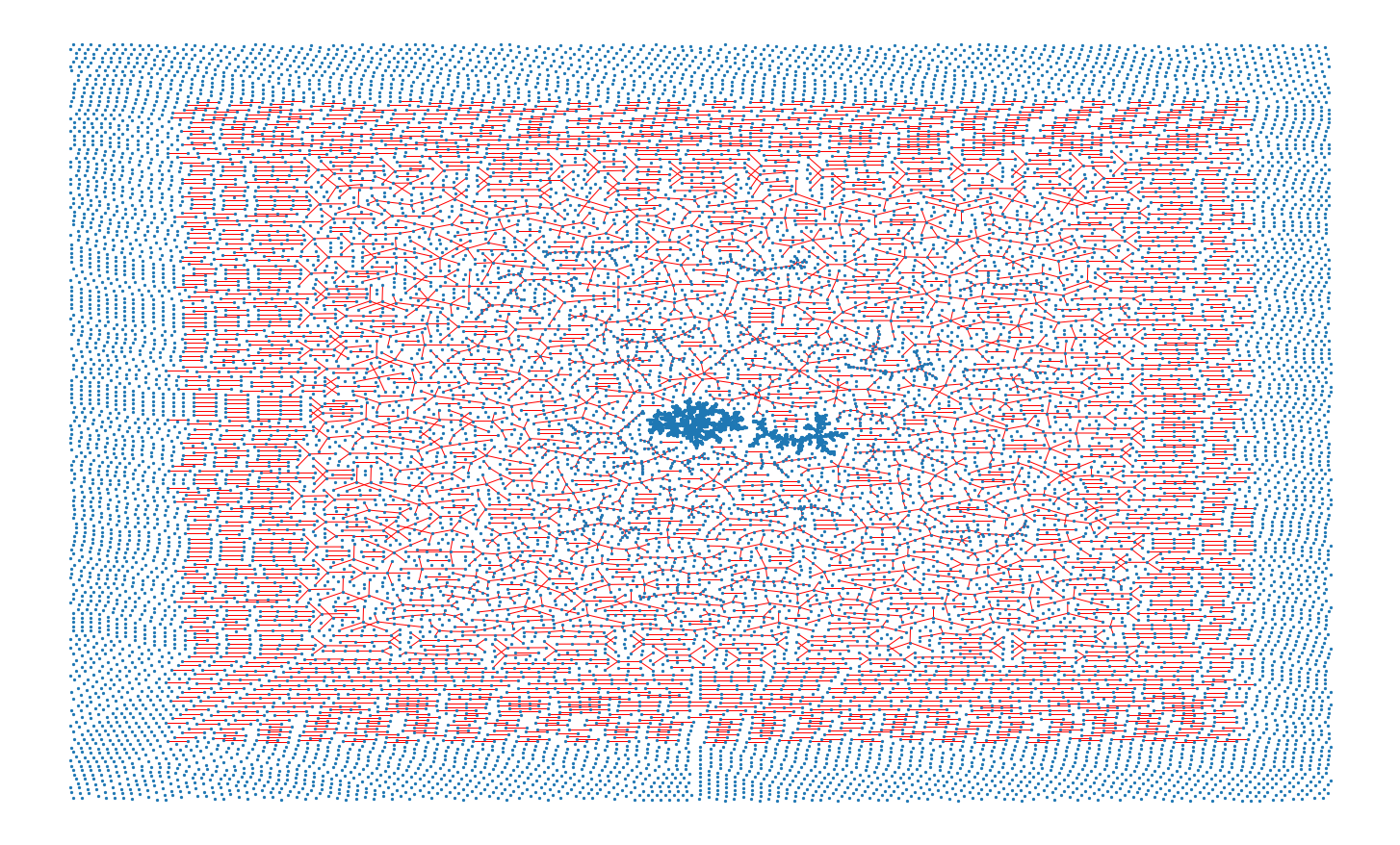}
\caption{An inhomogeneous random graph of size $n=20000$. The node weights are i.i.d with Pareto distribution of parameters $2/3,4$, and $p = \frac{5}{4n}$. These paremeters correspond to typical graphs that will be studied in this chapter.}
\label{figg1}
\end{figure}

\subsection{Definition of the exploration process} \label{2.1.2}
Before stating our main theorems, we define the exploration process of $G(\textbf{W},p)$ seen as a graph from the sequence $(G(\textbf{W},p))_{p \in [0,+\infty]}$ for a fixed $p$. All the results of this paper are proven by a careful study of this process. It is based on an "horizontal" exploration of the graph, called the breadth-first walk (BFW).  The BFW constructs the spanning forest of $G(\textbf{W},p)$, called the exploration forest. This is a forest consisting of  spanning trees of all the connected components of $G(\textbf{W},p)$, constructed in a particular way.\par 
For each potential edge $\{i,j\}$ recall the definition of  $E_{\{i,j\}}$ from the model presentation. The BFW operates by steps, define the following sets of vertices. A vertex is always in exactly one of those sets.
\begin{itemize}
    \item $(\mathcal{U}(i))_{n \geq i \geq 1}$ is the sequence of sets of undiscovered vertices at each step.
    \item $(\mathcal{D}(i))_{n \geq i \geq 1}$ is the sequence of sets of discovered but not yet explored vertices at each step.
    \item $(\mathcal{F}(i))_{n \geq i \geq 1}$ is the sequence of sets of explored vertices at each step.
\end{itemize}
First, choose a vertex  $i$  with probability: $$\mathbb{P}(v(1)=i) = \frac{w_i}{\ell_n},$$ 
and call it $v(1)$. Let $\mathcal{V}'$ be the set of all vertices labels, and $\mathcal{U}(1)=\mathcal{V}'\setminus \{v(1)\}$, $\mathcal{D}(1)=\{v(1)\}$. At step $2$, $v(1)$ is explored. It is thus not present  in $\mathcal{D}(2)$ and moved to $\mathcal{F}(2)$.  We call children of $v(1)$ the vertices $j$ that are unexplored at step $1$ and such that $E_{\{j,v(1)\}} \leq p $. Those children are moved to $\mathcal{D}(2)$ and become discovered but not yet explored.
Let $c(1)$ be the number of children of $v(1)$. Call them $(v(2),v(3),...,v(c(1)+1)$ in increasing order of their $E_{\{j,v(1)\}}$'s. For $i \geq 1$, denote the set $\{v(1),v(2),...,v(i)\}$ by $\mathcal{V}_i$. Hence, at step $2$ we have:
\begin{itemize}
    \item $\mathcal{U}(2) = \mathcal{V}\setminus \mathcal{V}_{c(1)+1}$.
    \item $\mathcal{D}(2) = \mathcal{V}_{c(1)+1} \setminus \mathcal{V}_{1}$.
    \item $\mathcal{F}(2) =  \mathcal{V}_{1}$.
\end{itemize}
Now, at step $3$, $v(2)$ becomes explored and  its children $\{v(c(1)+2),v(c(1)+3)...,v(c(1)+c(2)+3)\}$ become discovered but not yet explored. The BFW continues like this, node $v(i)$ becomes explored at step $i+1$, and its children are discovered at the same step. If the set of discovered nodes becomes empty at some step $i$, this means that the exploration of a connected component is finished. In that case, move on to the next step by choosing a vertex $j$ with probability proportional to its weight $w_j$ among the unexplored vertices and calling it $v(i)$ (like we did for $v(1)$) and exploring it. This construction ensures that a child has exactly one parent, since a child is always discovered while the process is exploring its parent. This ensures that we are constructing a forest. It is the exploration forest. We call the trees in that forest the exploration trees. By construction, exploration trees are spanning trees of the connected components of $G(\textbf{W},p)$. We say that a connected component is discovered at step $i$ if its first node discovered by the BFW is $v(i)$. Similarly, we say that a connected component is explored at step $i$ if its last node discovered by the BFW is $v(i-1)$.
Generally, let $c(i)$ be the number of children of the  node labeled $v(i)$. The exploration process associated to the BFW above is defined as follow for $n-1 \geq i \geq 0$:
\begin{equation*}
\begin{aligned}
    L'_0&=1,\\
    L'_{i+1}&= L'_i+c(i+1)-1.
\end{aligned}
\end{equation*}
The reflected exploration process is defined by
\begin{equation*}
\begin{aligned}
    L_0&=1,\\
    L_{i+1}&= \max(L_i+c(i+1)-1,1).
\end{aligned}
\end{equation*}
\par
\begin{figure}[!htbp]
\centering
\includegraphics[width=0.8\textwidth]{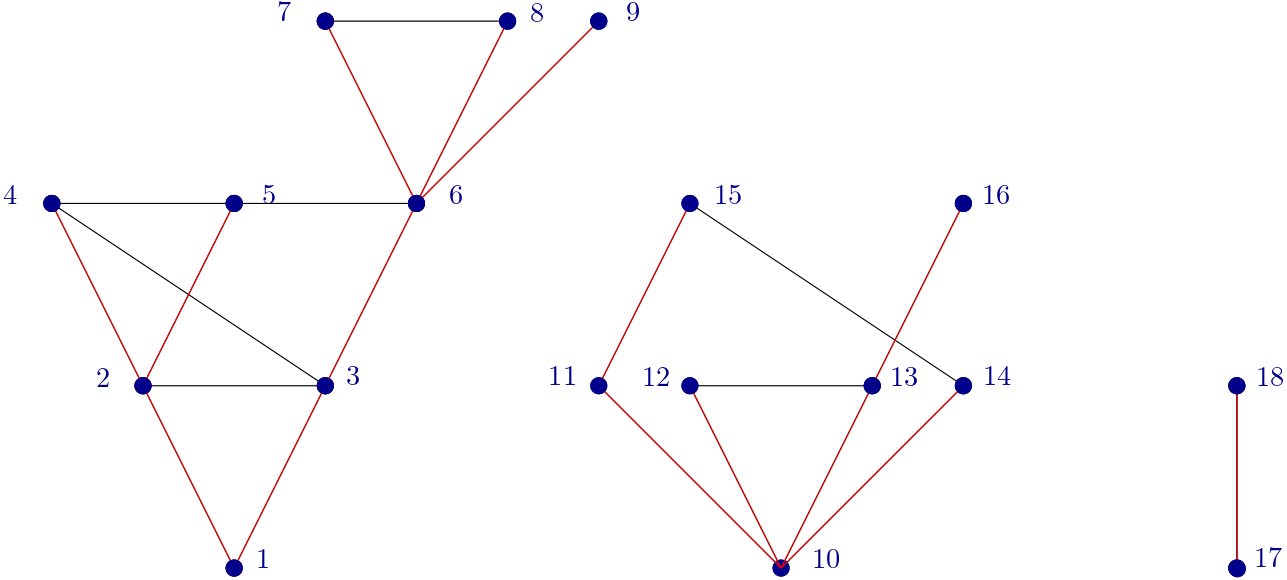}
\caption{Example of a graph with ordered nodes. The integers correspond to the order in the exploration process. The edges in red correspond to the edges of the exploration trees. The labels of the nodes are not represented.}
\label{fig1}
\end{figure}
\begin{figure}[!htbp]
\centering
\includegraphics[width=0.9\textwidth, height=0.3
\textheight]{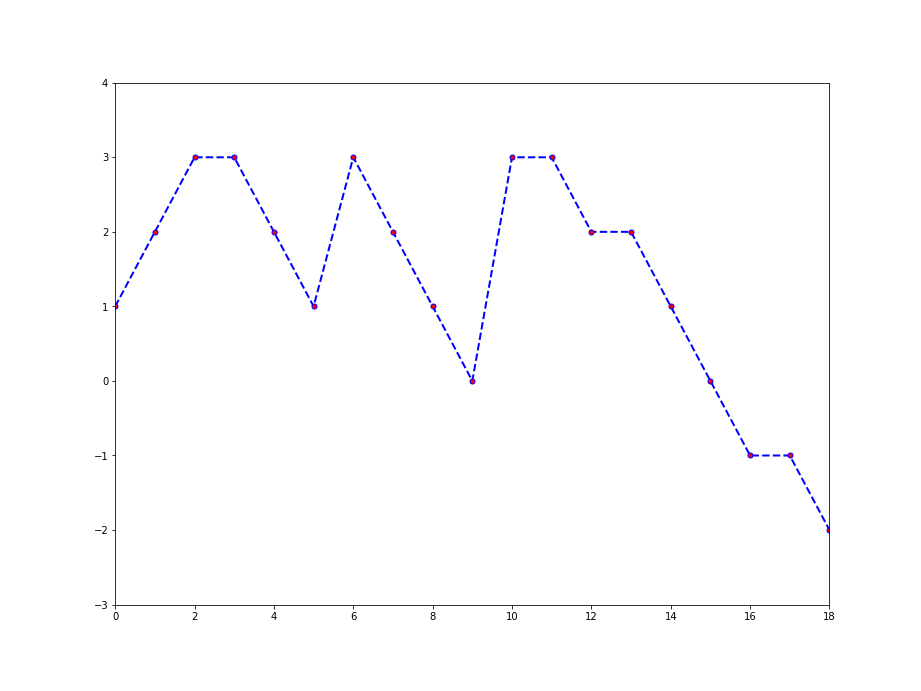}
\caption{The exploration process of the graph in Figure \ref{fig1}.}
\end{figure}
The increment of the process $L'$ at step $i$ is the number of nodes added to the set of discovered nodes in the BFW after exploring node $i$. This number is at least $-1$ if the node being explored has no children. The process $L'$ contains a lot of information  about $G(\textbf{W},p)$. For example, each time a connected component is explored $L'$ attains a new minimum. Using $L'$ transforms geometrical questions about the graph, such as "Is there a connected component of size proportional to $n$ ?" into questions regarding random walks such as "Is there an excursion of $L'$ above its past minimum of size proportional to $n$ ?". 
\par
Moreover, the order of appearance of the nodes in the exploration process corresponds to a size-biased sampling. Formally, we have for $i \in \{1,2,...,n-1\}$ and $j \in \{1,2,...,n\}$,
\begin{equation*}
\begin{aligned}
&\mathbb{P}\left(v(1)= j \right) = \frac{w_j}{\ell_n}.  \\
&\mathbb{P}\left(v(i+1)=j \quad | \mathcal{V}_i\right) = \frac{w_j\mathbbm{1}(j \not\in \mathcal{V}_i)}{\ell_n-\sum\limits_{k=1}^{i}w_{v(k)}}.
\end{aligned}
\end{equation*}
The proof of this fact uses only elementary results on exponential random variables. It is a widely known and used result (\cite{AL97}, \cite{SRJ10}, \cite{br20} ...). We sketch the proof here.
\begin{proof}
By construction:
$$
\mathbb{P}\left(v(1)= j \right) = \frac{w_j}{\ell_n}.
$$
Then for $v(2)$, if $c(1)$, the number of children of $v(1)$, is $0$ then, by definition, for any $j \geq 1$:
$$\mathbb{P}\left(v(2)=j | \mathcal{V}_1   , \, c(i) = 0\right) = \frac{w_j\mathbbm{1}(j \not\in \mathcal{V}_1)}{\ell_n-w_{v(1)}}.$$
Moreover if $c(1) \geq 1$, this means that there exists at least one $j \geq 1$ such that $j \neq v(1)$ and $E_{\{j,v(1)\}} \leq p $. By the absence of memory property of exponential random variables, for any $j \geq 1$:
\begin{equation*}
\begin{aligned}
&\mathbb{P}(v(2)=j  , \, c(1) \geq 1 |\mathcal{V}_{1}) \\
&= \mathbb{P}(v(2)=j |\mathcal{V}_{1}) -\mathbb{P}(v(2)=j  , \, c(1) = 0 |\mathcal{V}_{1}) \\
&= \mathbb{P}(\text{argmin}_{k \neq v(1)}(E_{\{k,v(1)\}})= j |\mathcal{V}_{1}) -\mathbb{P}(\text{argmin}_{k \neq v(1)}(E_{\{k,v(1)\}})= j|\mathcal{V}_{1})\mathbb{P}(c(1)=0|\mathcal{V}_{1}) \\
&=\mathbb{P}(\text{argmin}_{k \neq v(1)}(E_{\{k,v(1)\}})= j |\mathcal{V}_{1})\mathbb{P}(c(1)\geq 1|\mathcal{V}_{1}).
\end{aligned}
\end{equation*}
By well known properties of exponential random variables, since conditionally on $\mathcal{V}_1$ the $(E_{\{k,v(1)\}})_{k \neq v(1)}$'s are independent, we have:
$$\mathbb{P}(\text{argmin}_{k \neq v(1)}(E_{\{k,v(1)\}})= j |\mathcal{V}_{1})=\frac{w_j\mathbbm{1}(j \not\in \mathcal{V}_1)}{\ell_n-w_{v(1)}}.$$
This shows the statement for $v(2)$, and we can move to subsequent nodes by induction.
\end{proof}

\subsection{Conditions and main theorem}
The weights in  $\textbf{W}$ depend implicitly on $n$. We will assume the following conditions on $\textbf{W}$ in the entire paper.

\begin{Conditions} 
\label{cond_nodes}
There exists some positive random  variable $W$  such that: 
\begin{enumerate}[(i)]
\item The distribution of a  uniformly chosen weight $w_{X}$ converges weakly to $W$.
\item $\mathbb{E}[W^3] < \infty$.
\item $\mathbb{E}[W^2] = \mathbb{E}[W]$.
\item $\ell_n = \mathbb{E}[W]n + o(n^{2/3})$.
\item $\sum\limits_{k=1}^{n}w_k^2 = \mathbb{E}[W^2]n + o(n^{2/3})$.
\item $\sum\limits_{k=1}^{n}w_k^3 = \mathbb{E}[W^3]n + o(1)$. 
\item $\max_{i \leq n} w_i = o(n^{1/3})$. 
\end{enumerate}
\end{Conditions}

Conditions $i$,$ii$ and $iii$  ensure that the weak limit of $w_{v(1)}$ has a finite variance and mean $1$. Condition $iii$ can be ensured by changing the value of $p$. \par
Conditions $iv$,$v$ and $vi$ ensure that asymptotically the sum of the weights  behaves like the sum of independent identically distributed (i.i.d.) copies of $W$. Moreover, to further ease notations, as $n^{1/3} \geq w_1$, we will always use $n^{1/3}$ in our inequalities, even when $w_1$ would be sufficient.
An important case to keep in mind is when $(w_1,w_2,...,w_n)$ are realizations of random variables $(W_1,W_2,...,W_n)$ which are i.i.d. with distribution $W$. In that case Conditions $iv$,$v$ and $vi$ are consequences of Conditions $ii$ and $iii$ (see \cite{SRJ10} for a proof \footnote{\cite{SRJ10} shows that in that case the probability that the conditions hold tend to $1$ when $n$ tend to infinity. However, since we need concentration bounds, our weights need to verify these conditions deterministically.}).\par
We define the size of a connected component $\mathcal{C}$, with vertices set $V(\mathcal{C})$, of  $G(\textbf{W},p)$ as the number of vertices in $\mathcal{C}$. The distance between two vertices of $\mathcal{C}$ is the number of edges in the smallest (in number of edges) path between them. We also define the weight of $\mathcal{C}$ as:
$$
\sum_{i\in \mathcal{V}(C) } w_i.
$$
We call surplus (or excess) of $\mathcal{C}$ the number of edges that have to be removed from it in order to make it a tree. For instance, the surplus of a tree is $0$, and the surplus of a cycle is $1$. \par
Write $C = \frac{\mathbb{E}[W^3]}{\mathbb{E}[W]}$, and $p_{f_n} = \frac{\ell_n^{1/3}+f_n}{\ell_n^{4/3}}$. 
We can now state the main theorems of this paper. Of course, these theorems hold only under Conditions \ref{cond_nodes}.
\begin{Theorem}[\textbf{Size and weight of the giant component}]
\label{principal_1}
Let $1 \geq \epsilon'>0$. Then for $f_n = o(n^{1/3})$ large enough.  Consider the following event:  \\
 The largest connected component of $G(n,\textbf{W})$ has its size in the interval
$$
\left[\frac{2(1-\epsilon'/2)f_n\ell_n^{2/3}}{C}-\frac{\ell_n^{2/3}}{C},\frac{2(1+\epsilon'/2)f_n\ell_n^{2/3}}{C}\right],
$$
and its weight in the interval
$$
\left[\frac{2(1-\epsilon')f_n\ell_n^{2/3}}{C},\frac{2(1+\epsilon')f_n\ell_n^{2/3}}{C}\right],
$$
Then if Conditions \ref{cond_nodes} hold, there exists a positive constant $A > 0$ that only depend on the distribution of $W$, and such that the probability of this event not happening is at most:
$$A\exp\left(\frac{-f_n}{A}\right).$$
\end{Theorem}
\begin{Theorem}[\textbf{The excess of the giant component}]
\label{principal_2}
Let $Exc$ be the excess of the largest connected component of $G(n,\textbf{W})$. Then if Conditions \ref{cond_nodes} hold, there exists a positive constant $A >0$ that only depends on the distribution of $W$ such that:
\begin{equation*}
\mathbb{P}(Exc \geq Af_n^3) \leq A\exp\left(\frac{-f_n}{A}\right).
\end{equation*}
\end{Theorem}
\begin{Theorem}[\textbf{The sizes and weights of the small components}]
\label{principal_3}
Let  $1 >\epsilon'>0$ then for $f_n = o(n^{1/3})$ large enough, for any $1 \geq \epsilon > 0$ Consider the following events: \begin{itemize}
    \item All the connected components discovered before the largest connected component in the exploration process of $G(n,\textbf{W})$ have size smaller than
$$
\frac{\ell_n^{2/3}}{f_n^{1-\epsilon}},
$$
and weight smaller than
$$
\frac{(1+\epsilon')\ell_n^{2/3}}{f_n^{1-\epsilon}}.
$$
\item All the connected components discovered after the largest connected component in the exploration process of $G(n,\textbf{W})$ have size smaller than
$$
\frac{\ell_n^{2/3}}{f_n},
$$
and weight smaller than
$$
\frac{(1+\epsilon')\ell_n^{2/3}}{f_n}.
$$
\end{itemize}
Then if Conditions \ref{cond_nodes} hold, there exists a positive constant $A >0$ that only depends on the distribution of $W$ such that the probability of one of those events not happening is at most:
$$
A\left(\exp\left(\frac{-f_n^{\epsilon}}{A}\right)+\exp\left(\frac{-\sqrt{f_n}}{A}\right)+\exp\left(\frac{-n^{1/8}}{A}\right)\right).
$$
\end{Theorem}
\begin{Theorem}[\textbf{The excess of the small components}]
\label{principal_4}
Let $\textnormal{Exc}_0$ be the the sum of the excesses of the connected components discovered before the largest connected component in the exploration process of $G(n,\textbf{W})$. And let $\textnormal{Exc}_1$ be the maximal excess of the connected component discovered after the largest connected component. \par 
Then if Conditions \ref{cond_nodes} hold, there exists a positive constant $A >0$ that only depends on the distribution of $W$ such that, for any $1 \geq \epsilon > 0$:
\begin{equation*}
\mathbb{P}\left(\textnormal{Exc}_0 \geq Af_n^{\epsilon}\right) \leq A\exp\left(\frac{-f_n^{\epsilon/2}}{A}\right),
\end{equation*}
and 
$$
\mathbb{P}\left(\textnormal{Exc}_1 \geq Af_n^{\epsilon}\right) \leq A\left(\exp\left(\frac{-f_n^{\epsilon}\ln(\sqrt{f_n})}{A}\right)+\exp\left(\frac{-\sqrt{f_n}}{A}\right)+\exp\left(\frac{-n^{1/8}}{A}\right)\right).
$$
\end{Theorem}
These theorems give precise bounds on the size, weight and excess of not only the largest connected component but also the other small connected components of the graph $G(n,\textbf{W})$ in the barely supercritical regime, and in the critical regime when $f_n$ is a large enough constant. As a direct corollary of those theorems, we also obtain convergence results when $f_n \rightarrow +\infty$ (see Corollary \ref{conv_cor}). 
Statements concerning the largest connected component and the connected components discovered before it are proven in Section $4$. While statements concerning the connected components discovered after the largest one are proven in Section $5$. Moreover, at the cost of heavier notations, Theorem \ref{localize} provides a more precise statement than the one we presented in Theorem \ref{principal_3}. \par
\textbf{Notation}: In the remainder of the paper we drop the $n$ from $f_n$. $f$ will always be the critical parameter. Moreover we will always assume $f = o(n^{1/3})$ and $f \geq F$, where $F >0$ is a constant independent of $n$ which is large enough for all our theorems to hold. Similarly the variables $m=m_n$, $l=l_n$, $h=h_n$ and $y=y_n$ will always depend on $n$. The letters $A,A',A''...$ will be used for large positive constants that may only depend on the distribution of $W$.

\subsection{Motivation and previous work}
If $w_ i= 1$ for all $i$, then the edge capacities $(E_{\{i,j\}})$ are i.i.d.. In that case $G(\textbf{W},p)$ is an Erd\H{o}s-Rényi random graph. This is why the  rank-1 inhomogeneous random graph model is a natural generalization of Erd\H{o}s-Rényi random graphs. There are several variations of inhomogeneuous random graphs. The original inhomogeneous graph model was introduced by Aldous in his pioneer work on the multiplicative coalescent (\cite{AL97}), in this article he proved convergence of the component weights to a suitable limit. Then this model was further studied in \cite{AL98}.  The model we study here is closely related to the so called Norros-Reittu model (\cite{no06}). The difference between their model and ours being that their model allows for multi-edges. This, however, has no incidence on our proofs. And everything we show here still holds for their model.  Other models of inhomogeneuous random graphs include the Britton-Deijfen-Martin-L{\"o}f (Section $3$ in \cite{LNB06}) model, where edge $\{i,j\}$ is present with probability:
$$
\frac{w_iw_j}{n+w_iw_j}.
$$
And the Chung-Lu model (chapter $5$, Section $3$ in \cite{CL06} ) , where edge $\{i,j\}$ is present with probability:
$$
\frac{w_iw_j}{\ell_n}.
$$
This definition supposes that $\max_{i,j}(w_iw_j) \leq \ell_n$.
we could have chosen some other representation of the edge probabilities. However, under our conditions and regime, all the results that we will prove are also true for those models. Generally, it is easy to see that all the theorems we prove here under Conditions \ref{cond_nodes} will still hold for any of the models above. 
The choice of $p_{f} = \frac{\ell_n^{1/3}+f}{\ell_n^{4/3}}$, with $f = o(n^{1/3})$ is motivated by the phase transition that appears in the following theorem (proved in \cite{BSO07}). \par
\begin{Theorem}
Take $G(\textbf{W},\frac{c}{\ell_n})$ and  suppose that Conditions \ref{cond_nodes} are verified, then the following results hold with high probability \footnote{We say that a sequence of events $E_n$ holds with high probability if $\lim_{n \rightarrow \infty}\mathbb{P}(E_n) =1$}:
\begin{itemize}
\item \textbf{Subcritical regime}  If $c < 1$ then the largest connected component is of size $o(n)$.
\item \textbf{Supercritical regime}  If $c > 1$ then the largest connected component is of size $\Theta(n)$ and for any $i >1$ the $i$-th largest connected component is of size $o(n)$.
\item \textbf{Critical regime}  If $c = 1$ then for any $i \geq 1$ the $i$-th largest connected component is of size $\Theta(n^{2/3})$.
\end{itemize}
\end{Theorem}

From this theorem it appears that there is a phase transition at $c=1$. Just as in the Erd\H{o}s-Rényi model,  the right scale to look at the phase transition is for $c_n = 1 + \frac{\lambda}{\ell_n^{1/3}}$, with $\lambda >0$ a constant. Which explains our choice of $p_{f}$. This is the so called critical window. In Theorems \ref{principal_1}, \ref{principal_2}, \ref{principal_3}, and \ref{principal_4} we look at $c \sim 1$ and $f$ that is either a large constant, or that goes to infinity but stays $o(n^{1/3})$. The latter is what we call the barely supercritical regime.   \par

Plenty of work was done on $G(\textbf{W},\lambda)$ with $\lambda$ constant. The most recent and comprehensive one being in \citet*{br18} and \cite{br20}. Aldous was the first to study the closely related multiplicative coalescent in \cite{AL97}. In \cite{SRJ10} it is shown, under Conditions \ref{cond_nodes},  that the sequence of sizes of the connected components, properly rescaled, converges to a random vector. In \cite{SS17} this result is further extended, under stronger conditions than Conditions \ref{cond_nodes}, by showing that the sequence of connected components of the whole graph, seen as metric spaces, when properly rescaled, converge to a limit sequence of compact metric spaces. Moreover, under Conditions \ref{cond_nodes}, up to a multiplicative constant, this limit object has the distribution of the scaling limit of  Erd\H{o}s-Rényi random graphs (presented in \cite{ALBC12}). This shows that there is an invariance principle, although we have a generalization of Erd\H{o}s-Rényi random graphs the limit objects are just rescaled versions of one another. \par
However, unlike the Erd\H{o}s-Rényi case (see \citet*{LNB09}), there is no uniform study when $f$ moves through the critical window. For instance, there are no known concentration results  that depend on $f$ for the size of the largest component of rank-1 inhomogeneous random graphs. Moreover, there are no known concentration results for the barely supercritical regime. These are the cases that we treat in this paper. \par
This study has other implications for another object. For $n \in \mathbb{N}$, assign i.i.d., uniform random variables on $(0,1)$, that we call weights, to the edges of a complete graph of size $n$. Then the random minimum spanning tree (random MST) is the (almost surely unique) connected subgraph with $n$ vertices that minimizes the sum of the weights. It is a tree. In the Article by \cite{LNCG13} it is proven that when rescaling the distances by $n^{-1/3}$, the random MST converges to a compact tree-like metric space. The proof in \cite{LNCG13} relies heavily on a uniform study of the critical Erd\H{o}s-Rényi graph through the critical window and in the barely supercritical regime (done before in Article \cite{LNB09}). \par 
In order to do the same for the rank-1 inhomogeneous random graphs, instead of putting i.i.d. weights on a complete graph, put capacity $E_{\{i,j\}}$ on edge $\{i,j\}$ and construct the minimum spanning tree for those capacities. Call such a tree the inhomogeneous random MST. Clearly, this tree can be coupled with rank-1 inhomogeneous random graphs in the same fashion as in \cite{LNCG13}. One can ask  whether that tree, when properly rescaled, also converges to a continuous random tree-like metric space. And if the answer is yes,  will this metric space be a rescaled version of the scaling limit of the random MST in \cite{LNCG13}? A positive answer would show that there is still an invariance principle for those trees. \par
We intend on answering these questions in  follow up papers, and the bounds we prove in this paper will be crucial in our future proofs. \par
The biggest difficulty in proving our theorems is that the weight discovered at step $i$ of the exploration process depend on the weights discovered before it. Those weights appear in a size-biased fashion. This is why we show new concentration inequalities for size-biased sampling without replacement. We also make use of the note \cite{AYS16} in order to estimate the deviations of the sum of variables sampled without replacement.  Another difficulty is that we cannot rely on known results (for example results in Article \cite{LT90}) that were proved for Erd\H{o}s-Rényi graphs.  Everything has to be done separately for inhomogeneuous random graphs.\par

6There are other interesting problems that require more work. For instance there is the case of power law distributions for the node weights. Conditions \ref{cond_nodes} ensure that a uniform node weight behaves like a random variable with finite third moment. One can change those conditions, and allow the variable to follow a power law distribution of parameter $\tau > 3$. If $\tau > 4$, then we are in the case of finite third moments treated here. However, when $\tau \leq 4$, we expect the results to be vastly different. Informal arguments show that  in that case the scaling limit of the minimum spanning tree should be mutually singular with the scaling limit of random MST. This intuition is due to the appearance of Levy trees when studying those graphs (see \citet*{RSJ18} for further discussion of this model). \par
Finally another totally different set of questions regard biased sampling without replacement. Let $n \geq 1$ be an integer and $(a_1,a_2,...a_n)$ be decreasing real number. Moreover let $(p_1,p_2,...,p_n)$ be positive real numbers such that:
$$
\sum_{k=1}^n p_i = 1.
$$
Let $(V(1),V(2),...,V(n))$ be a vector random variables that correspond to indices sampled without replacement in the following way, for any $ i \in \{1,2,...,n-1\}$ and $j \in \{1,2,...,n-1\}$:
\begin{equation*}
\begin{aligned}
&\mathbb{P}\left(V(1)= j \right) = p_j,  \\
&\mathbb{P}\left(V(i+1)=j \quad | (V(1),V(2)...,V(i))\right) = \frac{p_j\mathbbm{1}(V(j) \not\in (V(1),...,V(i)))}{\sum\limits_{k=1}^{n}p_{k}-\sum\limits_{k=1}^{i}p_{V(k)}}.
\end{aligned}
\end{equation*}
Consider also $(J(1),J(2),...,J(n))$ that is a vector of independent random variables with the same distribution as $V(1)$. The $J(i)$'s correspond to indices sampled with replacement. Remark that size-biased sampling is a special case of biased sampling. While working on this paper two questions arose regarding these two samplings. First, under which set of conditions do we have the following inequality for any $n \geq m \geq l$ and real number $x \geq 0$:
$$
\mathbb{P}\left(\middle|\sum_{k=l}^ma_{V(i)}-\mathbb{E}[a_{V(i)}]\middle| \geq x\right) \leq \mathbb{P}\left(\middle|\sum_{k=l}^ma_{J(i)}-\mathbb{E}[a_{J(i)}]\middle| \geq x\right).
$$
This inequality means that biased sampling without replacement is more concentrated around its mean than biased sampling with replacement. The main idea behind this conjecture is that sampling without replacement tends to auto-concentrate itself around its mean. For instance, if for some $i \geq 1$, $V(i) = j$ and $a_{j}$ is very large, then we will not draw the same index $j$ in subsequent rounds. But in biased sampling with replacement, the same "bad" event can keep happening. \par
We were not able to find any trivial counter example to this inequality, so it could be true that it holds without any further assumptions. If not, then under which set of assumptions does it hold ? With such an inequality it would be easy to answer the question regarding inhomogeneous random graphs with power law distribution presented in the paragraph above. \par
Another question is for the ordered case. Suppose now that $p_1 \geq p_2 \geq ... \geq p_n$. This means that larger $a_i$'s  have larger probabilities of being drawn first. This is again a general case of size-biased sampling. Is it true then that for any $n-1 \geq m \geq 1$, and real numbers $(x_1,x_2,x_3,...x_n)$
$$
\mathbb{P}\left(a_{V(1)}\geq x_1  , \, a_{V(2)}\geq x_2, ...   , \, a_{V(m)}\geq x_m\right) \geq \mathbb{P}\left(a_{V(2)}\geq x_1  , \, a_{V(3)}\geq x_2, ...   , \, a_{V(m+1)}\geq x_m\right),
$$
and also
$$
\mathbb{P}\left(a_{J(1)}\geq x_1  , \, a_{J(2)}\geq x_2, ...   , \, a_{J(m)}\geq x_m\right) \geq \mathbb{P}\left(a_{V(1)}\geq x_1  , \, a_{V(2)}\geq x_2, ...   , \, a_{V(m)}\geq x_m\right).
$$
In Lemma \ref{decreasing}, we prove those inequalities for $m=1$. With some more work, we can prove them for $m=2$ also. We conjecture that they are in fact true for all $1 \leq m \leq n-1$. 
\section{Bounding the weights}
A well known fact is that the sum of weights sampled uniformly without replacement verifies slightly better Chernoff concentration inequalities as  the sum of weights sampled uniformly with replacement (See \cite{S74}). No such general result is available for size-biased sampling. \par
In this section we will always assume that Conditions \ref{cond_nodes} are verified. We will
prove concentration bounds for the weights sampled in size-biased order and without replacement under some conditions. 
\subsection{First concentration result and the mean}
The following theorem, from Article \cite{AYS16}, is a first important step in comparing the sum of the $(w_{v(i)})_i$'s with the sum of i.i.d. copies of a random variable.
\begin{Theorem}
\label{replacement}
Let $0 < l \leq m \leq n$ be two integers, and $J(1),J(2)...,J(n)$ be i.i.d. random variables with the  distribution of $v(l)$, then for any convex function $g$:
$$
\mathbb{E}\left[g\left(\sum_{i=l}^mw_{v(i)}\right)  \right] \leq \mathbb{E}\left[ g\left(\sum_{i=l}^mw_{J(i)} \right)   \right].
$$
\end{Theorem}

Generally, concentration bounds that use Chernoff's inequality are based on the fact that: 
$$
\mathbb{E}\left[ \exp{\left(\sum_{i=l}^mw_{J(i)} \right)} \right] = \mathbb{E}\left[ \exp{\left(w_{J(1)} \right)}  \right]^m.
$$
Hence, taking $g$ to be the exponential function in Theorem \ref{replacement} shows a Chernoff type inequality. This means that upper bounds that use Chernoff's inequality (first used in \cite{B24}) and which  hold for size-biased sampling with replacement are still true for size-biased sampling without replacement. This fact will be used later in the proofs. This is true in particular for Bernstein's inequality (\cite{B24}) which stems from Chernoff's bound. \par
The following lemmas give an estimation of the mean of $w_{v(i)}$. This first Lemma is already shown in one of the proofs that appear in \cite{SRJ10}, we prove it here again for clarity.
\begin{Lemma}
\label{sum_weights}
Suppose that Conditions \ref{cond_nodes} hold. Then for any $0 < l = o(n)$, and $i \in \{1,2,3\}$ we have:
$$\sum_{k=1}^lw_k^i=o(n).$$
\end{Lemma}
\begin{proof}
We do the proof for $i=3$, the other cases can be proved similarly or deduced easily from this case. Recall that the weights $(w_1,w_2,...w_n)$ are taken in decreasing order.
For any $K > 0$:
\begin{equation}
\begin{aligned}
\label{partial_sum}
    \sum_{k =1}^{l} \frac{w_k^3}{\ell_n} &\leq  \sum_{k =1}^{l} \frac{w_k^3\mathbbm{1}(w_k\leq K)}{\ell_n} + \sum_{k =1}^{n} \frac{w_k^3\mathbbm{1}(w_k > K)}{\ell_n} \\
    &\leq \frac{lK^3}{\ell_n} + \sum_{k =1}^{n} \frac{w_k^3\mathbbm{1}(w_k > K)}{\ell_n}.
\end{aligned}
\end{equation}
By the weak convergence in Conditions \ref{cond_nodes}:
\begin{equation*}
\lim_{n \rightarrow \infty}\left(\sum_{k =1}^{n} \frac{w_k^3\mathbbm{1}(w_k \leq K)}{n}\right) = \mathbb{E}[W^3\mathbbm{1}(W \leq K)],
\end{equation*}
and by the fact that: 
$$\sum_{k=1}^nw_k^3 = \mathbb{E}[W^3]n + o(n),$$ 
it follows that:
\begin{equation*}
\begin{aligned}
\lim_{n \rightarrow \infty}\left(\sum_{k =1}^{n} \frac{w_k^3\mathbbm{1}(w_k > K)}{\ell_n}\right) &= \frac{1}{\mathbb{E}[W]}\left(\mathbb{E}[W^3]- \mathbb{E}[W^3\mathbbm{1}(W \leq K)]\right) \\
&= \frac{\mathbb{E}[W^3\mathbbm{1}(W > K)]}{\mathbb{E}[W]}.
\end{aligned}
\end{equation*}
Since $ \mathbb{E}[W^3] < \infty$:
$$\lim_{K \rightarrow \infty}\left(\lim_{n \rightarrow \infty}\left(\sum_{k =1}^{n} \frac{w_k^3\mathbbm{1}(w_k > K)}{\ell_n}\right)\right) = 0.$$
 Together with the fact that and $l = o(n)$, letting $n$ go to infinity then $K$ go to infinity in Equation \eqref{partial_sum} yields:
\begin{equation}
    \label{partial_sum_2}
    \sum_{k =1}^{l} \frac{w_k^3}{\ell_n} = o(1).
\end{equation}
\end{proof}

\begin{Lemma}
\label{mean_pow_weights_2}
Suppose that Conditions \ref{cond_nodes} hold. Recall that $C = \frac{\mathbb{E}[W^3]}{\mathbb{E}[W]}$.
For any $l=o(n)$:
$$\mathbb{E}[w^{2}_{v(l)}] = C+o(1).$$
\end{Lemma}
\begin{proof}
We have using Lemma \ref{sum_weights}:
\begin{equation*}
\sum_{k \in \mathcal{V}_{l-1}}\frac{w_k}{\ell_n}  \leq \sum_{k =1}^{l-1} \frac{w_k}{\ell_n} = o(1).
\end{equation*}
Hence:
\begin{equation}
\begin{aligned}
\label{eq_3}
\mathbb{E}[w^2_{v(l)}] &= \mathbb{E}\left[\sum_{k \not\in \mathcal{V}_{l-1}}\frac{w_k^3}{\ell_n-\sum_{k' \in \mathcal{V}_{l-1}}w_k'} \right] \\
&= \mathbb{E}\left[\sum_{k \not\in \mathcal{V}_{l-1}}\frac{w_k^3}{\ell_n} \right](1+o(1)) \\
&= C(1+o(1))-\mathbb{E}\left[\sum_{k \in \mathcal{V}_{i-1}}\frac{w_k^3}{\ell_n} \right](1+o(1))+o(1).
\end{aligned}
\end{equation}
In order to finish the proof we use Lemma \ref{sum_weights} again:
\begin{equation}
\label{eq_4}
\mathbb{E}\left[\sum_{k \in \mathcal{V}_{l-1}}\frac{w_k^3}{\ell_n} \right] \leq \sum_{k =1}^{l-1} \frac{w_k^3}{\ell_n}=o(1).
\end{equation}
From Equations \eqref{eq_3} and \eqref{eq_4} we obtain:
\begin{equation}
    \label{grd}
    \mathbb{E}(w^2_{v(l)})  = C+o(1),
\end{equation}
which finishes the proof.
\end{proof}

\begin{Lemma}
\label{mean_weights}
Suppose that Conditions \ref{cond_nodes} hold. Let $l=o(n)$, we have:
$$\mathbb{E}[w_{v(l)}] = 1+o(1).$$
\end{Lemma}
\begin{proof}
As in the proof of Lemma \ref{mean_pow_weights_2} we have:
\begin{equation*}
\begin{aligned}
\label{mean_2}
\mathbb{E}(w_{v(l)}) &=  \frac{\mathbb{E}[W^2]}{\mathbb{E}[W]}(1+o(1)) -\mathbb{E}\left[\sum_{k \in \mathcal{V}_{l-1}}\frac{w_k^2}{\ell_n}  \right](1+o(1))\\
&= \frac{\mathbb{E}[W^2]}{\mathbb{E}[W]}(1 + o(1)).
\end{aligned}
\end{equation*} Recalling that $\frac{\mathbb{E}[W^2]}{\mathbb{E}[W]} = 1$ ends the proof.
\end{proof}
By the same argument we also have:
\begin{Lemma}
\label{mean_corr}
Suppose that Conditions \ref{cond_nodes} hold. Let $l=o(n)$. For any $0 < i < l$ we have:
$$\mathbb{E}(w_{v(i)}w_{v(l)}) = 1+o(1).$$
\end{Lemma}
\begin{proof}
We have using Lemma \ref{sum_weights}:
\begin{equation*}
\begin{aligned}
\mathbb{E}(w_{v(i)}w_{v(l)}) &=  \mathbb{E}\left[w_{v(i)}\sum_{k \not\in \mathcal{V}_{l-1}}\frac{w_k^2}{\ell_n-\sum_{k' \in \mathcal{V}_{l-1}}w_k'} \right] \\
&= \mathbb{E}\left[w_{v(i)}\sum_{k \not\in \mathcal{V}_{l-1}}\frac{w_k^2}{\ell_n} \right](1+o(1)) \\
&= 1+o(1),
\end{aligned}
\end{equation*}
which ends the proof.
\end{proof}
Thanks to these lemmas, we obtain a more precise estimation of the mean of $w_{v(l)}$.
\begin{Lemma}
\label{esperance}
Suppose that Conditions \ref{cond_nodes} hold. For any $l = o(n)$, we have :
\begin{equation*}
\begin{aligned}
\label{Esperance}
\mathbb{E}[w_{v(l)}] = 1+\frac{l}{\ell_n}\left(1-C\right) + o\left(\frac{l+n^{2/3}}{n}\right).
\end{aligned}
\end{equation*}
\end{Lemma}

\begin{proof}
By definition:
$$
\mathbb{E}[w_{v(l)}] = \mathbb{E}\left[\sum_{i \not\in \mathcal{V}_{l-1}} \frac{w_i^2}{\ell_n- \sum \limits_{i' \in \mathcal{V}_{l-1}}w_{i'}}\right].
$$

Moreover, by Lemma \ref{sum_weights}:
\begin{equation*}
\begin{aligned}
\label{first_approx}
\mathbb{E}\left[w_{v(l)}\right] &= \mathbb{E}\left[\sum_{i \not\in \mathcal{V}_{l-1}} \frac{w_i^2}{\ell_n\left(1-\frac{\sum_{i' \in \mathcal{V}_{l-1}}w_{i'}}{\ell_n}\right)}  \right]  \\
&=  \mathbb{E}\left[\sum_{i \not\in \mathcal{V}_{l-1}} \frac{w_i^2}{\ell_n}\left(1+\frac{\sum_{i' \in \mathcal{V}_{l-1}}w_{i'}}{\ell_n}\right) \right] + o\left(\frac{l}{n}\right).
\end{aligned}
\end{equation*}
By Lemmas \ref{sum_weights}, \ref{mean_pow_weights_2} and \ref{mean_weights}  it follows that:
\begin{equation*}
\begin{aligned}
\mathbb{E}\left[w_{v(l)}\right] &= \mathbb{E}\left[\sum_{i \not\in \mathcal{V}_{l-1}} \frac{w_i^2}{\ell_n}\left(1+\frac{\sum_{i' \in \mathcal{V}_{l-1}}w_{i'}}{\ell_n}\right) \right] + o\left(\frac{l}{n}\right) \\
&= \frac{\sum_{i = 1}^nw_i^2}{\ell_n} + \mathbb{E}\left[\frac{(\sum_{i'\in \mathcal{V}_{l-1}}w_{i'})(\sum_{i = 1}^nw_i^2)}{\ell_n^2}  \right] - \mathbb{E}\left[\frac{\sum_{i\in \mathcal{V}_{l-1}}w_i^2}{\ell_n} \right] \\
&- \mathbb{E}\left[\frac{\left(\sum_{i\in \mathcal{V}_{l-1}}w_i^2\right)\left(\sum_{i'\in \mathcal{V}_{l-1}}w_{i'}\right)}{\ell_n^2} \right] +  o\left(\frac{l}{n}\right) \\
&= \frac{\sum_{i = 1}^nw_i^2}{\ell_n} + \mathbb{E}\left[\frac{(\sum_{i'\in \mathcal{V}_{l-1}}w_{i'})(\sum_{i = 1}^nw_i^2)}{\ell_n^2}  \right] - \mathbb{E}\left[\frac{\sum_{i\in \mathcal{V}_{l-1}}w_i^2}{\ell_n} \right] \\
&- o\left(\mathbb{E}\left[\frac{\sum_{i\in \mathcal{V}_{l-1}}w_i^2}{\ell_n} \right]\right) +o\left(\frac{l}{n}\right) \\
&= 1 + \mathbb{E}\left[\frac{(\sum_{i'\in \mathcal{V}_{l-1}}w_{i'})(\sum_{i = 1}^nw_i^2)}{\ell_n^2}  \right] - \mathbb{E}\left[\frac{\sum_{i\in \mathcal{V}_{l-1}}w_i^2}{\ell_n} \right] + o\left(\frac{l+n^{2/3}}{n}\right). \\
&= 1 +\frac{l}{\ell_n}\left(1-C\right)+o\left(\frac{l+n^{2/3}}{n}\right).
\end{aligned}
\end{equation*}
\end{proof}
 Observe that with the assumption that $\mathbb{E}[W^2]=\mathbb{E}[W]$, the Cauchy-Schwarz inequality implies that:
 $$1-C = \left(1-\frac{\mathbb{E}[W^3]}{\mathbb{E}[W]}\right) \leq 0,$$
 so asymptotically $\mathbb{E}(w_{v(i)})$ decreases with $i$. Lemma \ref{decreasing} shows that in fact, it decreases all the time.
\begin{figure}[!htbp]
\centering
\includegraphics[width=1\textwidth]{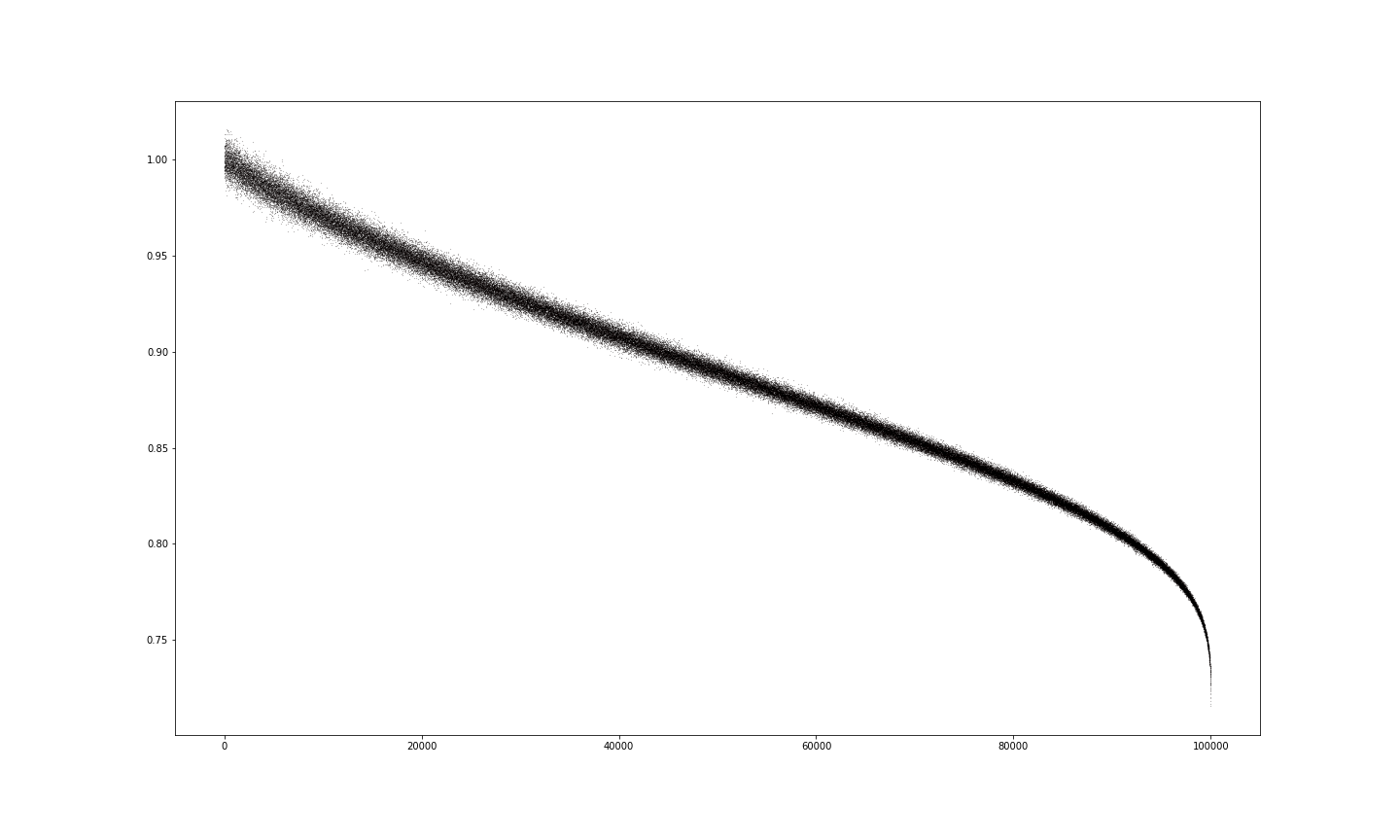}
\caption{A simulation of the values of the $\mathbb{E}[w_{v(i)}]$'s for $n \geq i \geq 1$. This simulation is done on $n = 100000$ weights verifying Conditions \ref{cond_nodes} by doing $m = 10000$ rounds of biased sampling without replacement and averaging the result. }
\end{figure}
\subsection{A more precise concentration inequality}
In order to obtain concentration inequalities for size-biased sampling without replacement, we will use a randomization trick. The main idea here is that taking weights without replacement is the same as putting exponential "clocks" on each weight and taking a weight when its clock rings. \par
More precisely let $(T_i)_{i \leq n}$ be a sequence of independent exponential random variables with respective rates $(w_i/\ell_n)_{i\leq n}$. Define the following quantities for $x \geq 0$: 
$$
N(x) = \sum_{k=1}^n\mathbbm{1}(T_k \leq x),
$$
$$
X(x) = \sum_{k=1}^nw_k\mathbbm{1}(T_k \leq x).
$$
By basic properties of exponential random variables, $(v'(1),v'(2),...,v'(n))$, the distinct random indices of the $T_i$'s taken in increasing order, i.e:
$$T_{v'(1)} \leq T_{v'(2)} \leq ... \leq T_{v'(n)},$$
are distributed as a size-biased sample taken without replacement. \par
Moreover the following equality holds : 
$$
X(x) = \sum_{k=1}^{n}w_{v'(k)}\mathbbm{1}(N(x) \geq k).
$$

Since $N(x)$ and $X(x)$ are sums of independent random variables, we can apply Bernstein's inequality (\cite{B24}) in order to obtain the following lemma. We let $w_{v(0)} = 0$.
\begin{Lemma}
\label{bernstinou}
Suppose that Conditions \ref{cond_nodes} hold. For any $x \geq 0$ and $t \geq 0$, the following holds: 
$$
\mathbb{P}(|X(x) - \mathbb{E}[X(x)]| \geq t) \leq 2\exp\left({\frac{-t^2}{2(tn^{1/3}+x)}}\right),
$$
and
$$
\mathbb{P}(|N(x) - \mathbb{E}[N(x)]| \geq t) \leq 2\exp\left({\frac{-t^2}{2(t+x)}}\right).
$$
\end{Lemma}

The following conditions will always be verified in this section. They give a regime where our concentration bounds hold. 
\begin{Conditions}
\label{cond_2}
We say that $(a(n),b(n))$  verifies Conditions \ref{cond_2} if for all $n$ large enough:
$$
 \exp\left(\frac{-b(n)^2}{\bar{A}(b(n)n^{1/3}+a(n))}\right) < 1/4,
$$
$$
\lim_{n \rightarrow \infty} a(n) = \lim_{n \rightarrow \infty} b(n) = +\infty
$$
$$
a(n)=o(n),
$$
$$
b(n) = O(a(n)),
$$
$$
a(n) =  O\left(b(n)\ell_n^{1/3}\right),
$$
and:
$$
(a(n))^2 = O\left(b(n)\ell_n\right),
$$
where $\bar{A} > 0$ is independent of $n$ and larger than all the other constants $A,A',A'' ...$ that appear in this paper.
\end{Conditions}
The condition $b(n)=O(a(n))$ is not necessary, but it makes some computations easier and will be true in all the practical cases in this paper. Moreover, notice that if $(a(n),b(n))$ verifies Conditions \ref{cond_2} then for any $A > 0$ the couple $(a(n),Ab(n))$ will also verify those conditions. 
We want to prove that there exists an $A >0$ such that if $(m,y)$ verify Conditions \ref{cond_2} then: 
$$
\mathbb{P}\left[\sup_{i \leq m}  \left | \sum_{k=1}^i w_{v(k)} - \mathbb{E}\left[\sum_{k=1}^i w_{v(k)}\right] \right | \geq y\right] \leq  A\exp\left(\frac{-y^2}{A(yn^{1/3}+m)}\right).
$$
In order to do so, we will use the fact that if $N(u_n) \geq m$ for some $u_n > 0$ then: 
\begin{equation*}
\sup_{i \leq m} \left|\sum_{k=1}^i w_{v'(i)} - \mathbb{E}\left[\sum_{k=1}^i w_{v'(i)}\right]\right| \leq \sup_{x \leq u_n} \left|X(x) - \sum_{k=1}^{N(x)}\mathbb{E}\left[     w_{v'(i)}\right]\right|.
\end{equation*}
Then we will show concentration of the right-hand side of the above inequality. The following fact will be used through this whole section.
For any $x \geq 0$:
\begin{equation}
\label{expx}
x \geq 1-e^{-x} \geq x-\frac{x^2}{2}.
\end{equation}

We start by showing the following lemma:
\begin{Lemma}
\label{eqfff}
Suppose that Conditions \ref{cond_nodes} hold. Let $(a(n),b(n))$ verify Conditions \ref{cond_2}. Then there exists a constant $A>0$ such that for all $n$ large enough:
\begin{equation*}
\mathbb{P}\left[  \sup_{x \leq a(n)} \mathbb{E}\left[X(x)\right] - \sum_{k=1}^{N(x)}\mathbb{E}\left[     w_{v(i)}\right] \geq b(n) \right] \leq \mathbb{P}\left[\inf_{x \leq a(n)}N(x)-\mathbb{E}[N(x)] \leq \frac{-b(n)}{A}+1\right],
\end{equation*}
and: 
\begin{equation*}
\mathbb{P}\left[ \inf_{x \leq a(n)} \mathbb{E}\left[X(x)\right] - \sum_{k=1}^{N(x)}\mathbb{E}\left[     w_{v(i)}\right] \leq -b(n) \right] \leq \mathbb{P}\left[\sup_{x \leq a(n)} N(x)-\mathbb{E}[N(x)] \geq \frac{b(n)}{A}-1\right],
\end{equation*}
and the same inequalities hold without the $\sup$ and $\inf$.
\end{Lemma}
\begin{proof}
Let $x \leq a(n)$. By Equation \eqref{expx} and Conditions \ref{cond_nodes}:
\begin{equation}
\begin{aligned}
\label{eq126}
\mathbb{E}[X(x)] &= \sum_{k=1}^{n}w_k\mathbb{P}(T_k \leq x) \\
&= \sum_{k=1}^{n}w_k\left(1-\exp\left(\frac{-w_kx}{\ell_n}\right)\right) \\
&\leq \sum_{k=1}^{n}\frac{w_k^2x}{\ell_n} \\
&= x(1+o(n^{-1/3})).
\end{aligned}
\end{equation}
For any $b'(n)$ such that $(a(n),b'(n))$ verify Conditions \ref{cond_2}, there exists $A' >0$ such that:
$$
x^2 \leq a(n)^2 \leq A'b'(n)\ell_n.
$$
Denote $\lceil\mathbb{E}[N(x)]-b'(n)\rceil$ by $u$. By Conditions \ref{cond_nodes} and Equation \eqref{expx} we obtain:
\begin{equation}
\begin{aligned}
\label{eq123}
u &\geq x-b'(n) - \sum_{k=1}^n\frac{w_k^2x^2}{2\ell_n^2} \\
&\geq x - b'(n) - \frac{x^2}{2\ell_n}+o\left(\frac{x^2}{\ell_n}\right) \\
&\geq x - b'(n) - \frac{A'b'(n)}{2}+o\left(\frac{x^2}{\ell_n}\right).
\end{aligned}
\end{equation}
Moreover by Condition \ref{cond_2}:
\begin{equation}
\begin{aligned}
\label{eq124}
u^2 &\leq (x+b'(n))^2 \\
&\leq 2x^2+2b'(n)^2 \\
&\leq 2A'\ell_nb'(n)+2b'(n)^2 \\
&\leq A''\ell_nb'(n),
\end{aligned}
\end{equation}
where $A'' > 0$ is some large constant. By
Equations \eqref{eq123}, \eqref{eq124}, Conditions \ref{cond_2} and Lemma \ref{esperance} we have:
\begin{equation}
\begin{aligned}
\label{eq125}
\sum_{k=1}^{u}\mathbb{E}\left[w_{v(i)}\right] &=  \sum_{k=1}^{u}\left(1 +\frac{k}{\ell_n}\left(1-C\right)\right)+o\left(\frac{u^2+un^{1/3}}{n}\right) \\
&= u + \frac{u^2}{2\ell_n}\left(1-C\right)+o\left(\frac{u^2+un^{1/3}}{n}\right) \\
&\geq x-A'''b'(n),
\end{aligned}
\end{equation}
where $A''' > 0$ is a large constant. 
 Inequalities \eqref{eq126} and \eqref{eq125} and Conditions \ref{cond_2} yield:
\begin{equation*}
\begin{aligned}
\label{eq142}
\mathbb{E}[X(x)]-\sum_{k=1}^{u}\mathbb{E}\left[w_{v(i)}\right] \leq A'''b'(n)+o(xn^{-1/3}).
\end{aligned}
\end{equation*}
And of course, since $\mathbb{E}\left[w_{v(i)}\right]$ is positive for all $i \leq n$,  the same inequality holds if we replace $u$ by $u' \geq u$. This show that:
\begin{equation*}
\begin{aligned}
\label{eqf}
\left( \mathbb{E}\left[X(x)\right] - \sum_{k=1}^{N(x)}\mathbb{E}\left[     w_{v(i)}\right] \geq 2A'''b'(n) \right) \Rightarrow \left(N(x) \leq \mathbb{E}[N(x)]-b'(n)+1\right)
\end{aligned}
\end{equation*}
Taking $b(n)=2A'''b'(n)$ proves the first inequality of the lemma, the second inequality is proved similarly.
\end{proof}
Similarly we have the following lemma for which we omit the proof
\begin{Lemma}
\label{eqfff1}
Suppose that Conditions \ref{cond_nodes} hold. Let $(a(n),b(n))$ verify Conditions \ref{cond_2}. Then there exists a constant $A>0$ such that for all $n$ large enough:
\begin{equation*}
\begin{aligned}
&\mathbb{P}\left[  \sup_{x \leq a(n)} \mathbb{E}\left[X(a(n))-X(x)\right] - \sum_{k=N(x)}^{N(a(n))}\mathbb{E}\left[     w_{v(i)}\right] \geq b(n) \right]  \\\leq &\mathbb{P}\left[\inf_{x \leq a(n)}N(a(n))-N(x)-\mathbb{E}[N(a(n))-N(x)] \leq \frac{-b(n)}{A}+1\right],
\end{aligned}
\end{equation*}
and: 
\begin{equation*}
\begin{aligned}
&\mathbb{P}\left[ \inf_{x \leq a(n)} \mathbb{E}\left[X(a(n))-X(x)\right] - \sum_{k=N(x)}^{N(a(n))}\mathbb{E}\left[     w_{v(i)}\right] \leq -b(n) \right] \\ \leq &\mathbb{P}\left[\sup_{x \leq a(n)} N(a(n))-N(x)-\mathbb{E}[N(a(n))-N(x)]\geq \frac{b(n)}{A}-1\right],
\end{aligned}
\end{equation*}
and the same inequalities hold without the $\sup$ and $\inf$.
\end{Lemma}
These lemmas will allow us to prove the following concentration inequality. Recall that $m=m(n)$ and $y = y(n)$ depend implicitly on $n$.
\begin{Lemma}
\label{wut}
Suppose that Conditions \ref{cond_nodes} hold. There exist a constant $A >0$ such that 
if $(x_n,y)$  verifies Conditions \ref{cond_2}, then:
$$
\mathbb{P}\left(\left|X(x_n)-\sum_{k=1}^{N(x_n)}\mathbb{E}[w_{v(i)}]\right| \geq y \right) \leq A\exp\left(\frac{-y^2}{A(yn^{1/3}+x_n)}\right),
$$

\end{Lemma}
\begin{proof}
By the union bound: 
\begin{equation}
\begin{aligned}
\label{ters}
&\mathbb{P}\left[ \middle|X(x_n) - \sum_{k=1}^{N(x_n)}\mathbb{E}\left[     w'_{v(i)}\right]\middle| \geq y \right] \\ &\leq \left(\mathbb{P}\left[ \middle|X(x_n) - \mathbb{E}\left[X(x_n)\right]\middle| \geq \frac{y}{2}\right] 
+ \mathbb{P}\left[  \middle|\mathbb{E}\left[X(x_n)\right] - \sum_{k=1}^{N(x_n)}\mathbb{E}\left[     w'_{v(i)}\right]\middle| \geq \frac{y}{2}\right]\right).
\end{aligned}
\end{equation}
We bound separately each term of the right-hand side of Equation \eqref{ters}. 
Lemma \ref{bernstinou} states that:
\begin{equation}
\begin{aligned}
\label{f_part}
\mathbb{P}\left[\middle| X(x_n) - \mathbb{E}\left[X(x_n)\right]\middle| \geq \frac{y}{2}\right] &\leq 2\exp\left(\frac{-y^2}{8(yn^{1/3}+x_n)}\right).
\end{aligned}
\end{equation}

Using Equations  \eqref{f_part}, Lemma \ref{eqfff} on $\left(x_n,y/2\right)$ and Lemma \ref{bernstinou} to bound the second expression in the right-hand side of Equation \eqref{ters} expression in Equation \eqref{ters}  shows that:
$$
\mathbb{P}\left(\left|X(x_n)-\sum_{k=1}^{N(x_n)}\mathbb{E}[w_{v(i)}]\right| \geq y \right) \leq A'\exp\left(\frac{-y^2}{A'(yn^{1/3}+x_n)}\right),
$$
where $A' >0$ is a large constant.
\end{proof}
In order to prove concentration inequalities on the $N(t)$ and $X(t)$ for all $t$ in some interval, we use the chaining method. This method consists of crafty discretizations of the "time" parameter space in order to derive general bounds for all "times". The method is explained in chapter $13$ of \cite{BML13}. It is attributed to Kolmogorov, and it has been vastly used and improved by Dudley (\cite{D73}) and Talagrand (for instance \cite{T05}). 
\begin{Lemma}
\label{lema102}
Suppose that Conditions \ref{cond_nodes} hold. There exist a constant $A >0$ such that, for any $(m,y)$ that verify Conditions \ref{cond_2}:
$$
\mathbb{P}\left(\sup_{0 \leq t \leq m} \left(X(t) -  \mathbb{E}\left[X(t)\right]\right) \geq y\right) \leq A\exp\left(\frac{-y^2}{A(yn^{1/3}+m)}\right).
$$
\end{Lemma}
\begin{proof}
Recall that $w_1 \geq w_2 \geq w_3 \geq  ...$ and for $i \leq n$ write:
$$
X_i(t) =  \sum_{k=i+1}^nw_k\mathbbm{1}(T_k \leq t).
$$
By Bernstein's inequality and basic computations, for any $u > 0$ and $s < t$:
\begin{equation}
\label{dead1}
\mathbb{P}\left(|X_i(t)-X_i(s)-\mathbb{E}[X_i(t)-X_i(s)]| \geq \sqrt{2(t-s)\sum_{k=i+1}^n\frac{w_k^3}{\ell_n}u}+uw_{i+1}\right) \leq 2\exp(-u).
\end{equation}
For $ i \geq 0$ let:
$$
\Gamma_i=\left\{m\frac{k}{2^i}, 0 \leq k < 2^i\right\}\cup\left\{T_k, 1 \leq k < 2^i\right\}.
$$
Let $f_i : t \in [0,m] \mapsto \max\{z \in \Gamma_i, t > z\}$. We have, by definition of $f_i$ and $\Gamma_i$, for any $t \leq m$:
\begin{equation*}
\begin{aligned}
X(t)-X(f_i(t)) &=  \sum_{k=1}^nw_k\mathbbm{1}(f_i(t)< T_k \leq t) \\
&= \sum_{k=2^i}^nw_k\mathbbm{1}(f_i(t)< T_k \leq t) \\
&= X_{2^i-1}(t)-X_{2^i-1}(f_i(t)).
\end{aligned}
\end{equation*}
Since $f_i(t)$ is measurable with respect to the $(T_k)_{k <2^i}$'s. And conditionally on $f_i(t)$, $X(t)-X(f_i(t))$ is a sum of independent random variables. We can apply Bernstein's inequality to obtain similarily to Equation \eqref{dead1}:
\begin{equation}
\begin{aligned}
\label{dead2}
&\mathbb{P}\left(|X(t)-X(f_i(t))-\mathbb{E}[X(t)]-X(f_i(t))]| \geq \sqrt{2(t-f_i(t))\sum_{k=2^i}^n\frac{w_k^3}{\ell_n}u}+uw_{2^i}\right)\\ &\leq 2\exp(-u).
\end{aligned}
\end{equation}
 Let:
$$
\rho_i=\sqrt{3\frac{m}{2^i}C(u(i+1))}+u(i+1)w_{2^i}.
$$
Since $(t-f_i(t)) \leq \frac{m}{2^i}$ and $\sum_{k=1}^n\frac{w_k^3}{\ell_n}=C(1+o(1))$.
Inequality \eqref{dead2} with $u'=u(i+1)$ yields:
\begin{equation*}
\begin{aligned}
\label{dead3}
&\mathbb{P}\left(|X_i(f_i(t))-X_i(t)-\mathbb{E}[X_i(f_i(t))-X_i(s)]| \geq \rho_i\right) \leq 2\exp(-u(i+1)).
\end{aligned}
\end{equation*}
The classical chaining argument is that any $0 \leq t \leq m$ can be written as:
$$
t = \sum_{i=0}^{\infty}(f_{i+1}(t)-f_i(t)),
$$
This gives us by union bound, if we suppose that $u > \ell_n(4)$:
\begin{equation}
\begin{aligned}
\label{dead4}
&\mathbb{P}\left(\sup_{0 \leq t \leq m}|X(t)-\mathbb{E}[X(t)]| \geq \sum_{i=0}^{\infty}\rho_i\right) \\
&\leq \sum_{i=0}^{\infty}\sum_{t\in \Gamma_{i+1}}\mathbb{P}\left(|X_i(t)-X_i(f_i(t))-\mathbb{E}[X_i(t)-X_i(f_i(t))]| \geq \rho_i\right)\\ 
&\leq \sum_{i=0}^{\infty}\sum_{t\in \Gamma_{i+1}}2\exp(-u(i+1)) \\
&\leq \sum_{i=0}^{\infty}2^{i+3}\exp(-u(i+1))\\
&\leq \frac{8e^{-u}}{1-e^{-(u-\ell_n(2))}} \\
&\leq Ae^{-u},
\end{aligned}
\end{equation}
where $A > 0$ is some large constant and  with the convention that $w_k = 0$ if $k \geq n$. Now notice that as $\sum_{k=1}^nw_k^3 \leq An$ for some constant $A$, we have for any $i \geq 0$, $w_{2^i} \leq \frac{A^{1/3}n^{1/3}}{2^{i/3}}$. Hence:
\begin{equation}
\begin{aligned}
\label{dead5}
\sum_{i=1}^{\log(n)}(i+1)w_{2^i} &\leq \sum_{i=1}^{+\infty}\frac{A^{1/3}(i+1)n^{1/3}}{2^{i/3}} \\
&\leq A'n^{1/3},
\end{aligned}
\end{equation}
where $A' > 0$ is some large constant.
 With  Equation \eqref{dead5}, a simple computation shows that there exists $A  > 0$ such that:
$$
\sum_{i=0}^{\infty} \rho_i = A'\left(\sqrt{mu}+un^{1/3}\right),
$$
Replacing in Equation \eqref{dead4} give just another way of writing Bernstein's inequality, we finish by taking for instance:
$$u = \frac{y^2}{2A'^2(n^{1/3}y+m)},$$
which also ensures that $u > \ln(4)$ by Conditions \ref{cond_2}. 
\end{proof}
The following three lemmas have similar proofs, and their proofs are thus omitted.
\begin{Lemma}
\label{lema101}
Suppose that Conditions \ref{cond_nodes} hold. There exists $A > 0$ such that, for any $(m,y)$ that verifies Conditions \ref{cond_2}:
$$
\mathbb{P}\left(\sup_{0 \leq t \leq m} \left|N(m) -N(m-t) -  \mathbb{E}\left[N(m)-N(m-t)\right]\right| \geq y\right) \leq A\exp\left(\frac{-y^2}{A(y+m)}\right).
$$
\end{Lemma}

\begin{Lemma}
\label{lema100}
Suppose that Conditions \ref{cond_nodes} hold. There exists $A >0$ such that, for any $(m,y)$ that verify Conditions \ref{cond_2}:
$$
\mathbb{P}\left(\sup_{0 \leq t \leq m} \left|N(t) -  \mathbb{E}\left[N(t)\right]\right| \geq y\right| \leq A\exp\left(\frac{-y^2}{A(y+m)}\right)
$$
\end{Lemma}
\begin{Lemma}
\label{lema103}
Suppose that Conditions \ref{cond_nodes} hold. There exists $A > 0$ such that, for any $(m,y)$ that verifies Conditions \ref{cond_2}:
$$
\mathbb{P}\left(\sup_{0 \leq t \leq m} \left|X(m) -X(m-t) -  \mathbb{E}\left[X(m)-X(m-t)\right]\right| \geq y\right) \leq A\exp\left(\frac{-y^2}{A(yn^{1/3}+m)}\right).
$$
\end{Lemma}
Now we can prove the concentration of the size-biased sum of weights sampled without replacement.
\begin{Theorem}
\label{concentration}
Suppose that Conditions \ref{cond_nodes} hold. There exists a constant $A >0$ that satisfies the following, for  $(m,y)$ that verifies Conditions \ref{cond_2}, we have:
$$
\mathbb{P}\left[ \sup_{0 \leq i \leq j \leq m}  \left | \sum_{k=i}^j w_{v(k)} - \mathbb{E}\left[\sum_{k=i}^j w_{v(k)}\right] \right | \geq y\right] \leq  A\exp\left(\frac{-y^2}{A(yn^{1/3}+m)}\right).
$$
\end{Theorem}
\begin{proof}
Let $l(m)$ be such that $\mathbb{E}[N(l(m))]=m$. If $E = \{N(3(l(m)+y)) \geq m\}$ holds, then: 
\begin{equation*}
\sup_{0 \leq i \leq j \leq m} \left|\sum_{k=i}^{j} w_{v'(k)} - \mathbb{E}\left[\sum_{k=i}^{j} w_{v'(k)}\right]\right| \leq \sup_{0 \leq x \leq z \leq 3(l(m)+y))} \left|X(z) - X(x) - \sum_{k=N(x)}^{N(z)}\mathbb{E}\left[     w_{v'(k)}\right]\right|.
\end{equation*}
We only bound:
$$
\mathbb{P}\left[\inf_{i \leq j \leq m} \; \sum_{k=i}^{j} w_{v(k)} - \mathbb{E}\left[\sum_{k=i}^{j} w_{v(k)}\right] \leq -y\right],
$$
as the argument for bounding the other part is the same.
By union bound with the event $E$:
\begin{equation}
\begin{aligned}
\label{toztoz}
&\mathbb{P}\left(\inf_{i \leq j \leq m} \; \sum_{k=i}^{j} w_{v'(k)} - \mathbb{E}\left[\sum_{k=i}^{j} w_{v'(k)}\right] \leq -y\right) \\
\leq &\mathbb{P}\left(E,\inf_{i \leq j \leq m} \; \sum_{k=i}^{j} w_{v'(k)} - \mathbb{E}\left[\sum_{k=i}^{j} w_{v'(k)}\right] \leq -y\right) + P(\bar{E})\\
\leq &\mathbb{P}\left(\inf_{0 \leq x \leq z \leq 3(l(m)+y))} X(z)-X(x)- \sum_{k=N(x)}^{N(z)}\mathbb{E}\left[     w_{v'(i)}\right] \leq -y\right)+P(\bar{E}).
\end{aligned}
\end{equation}
Note that by Conditions \ref{cond_nodes}, for $n$ large enough:
\begin{equation}
\begin{aligned}
\label{eq009}
\mathbb{E}\left[N\left(\frac{\ell_n}{9}\right)\right] &\geq \sum_{k=1}^n\left(\frac{w_i}{9}-\frac{w_i^2}{162}\right) \\
&\geq \frac{\ell_n}{11}(1+o(1)) \\
&\geq \frac{\ell_n}{12}.
\end{aligned}
\end{equation}
Since $(\mathbb{E}[N(x)])_{x\geq0}$ is an increasing function, by Equation \eqref{eq009}, $l(m) \leq \ell_n/9$. Hence, by Equation \eqref{expx}:
\begin{equation}
\begin{aligned}
\label{eq1456}
\mathbb{E}[N(l(m)] &= m \\
&\geq l(m)-\frac{\sum_{k=1}^{n}w_k^2l(m)^2}{2\ell_n^2} \\
&\geq l(m)-\frac{l(m)}{18}(1+o(1)) \\
&\geq \frac{8l(m)}{9}.
\end{aligned}
\end{equation}
By Lemma \ref{bernstinou} and Equation \eqref{eq1456}:
\begin{equation}
\label{first_half}
\mathbb{P}(\bar{E}) \leq A\exp\left(\frac{-y^2}{A(y+m)}\right),
\end{equation}
for some large constant $A >0$. 
Now we need to prove that: 
\begin{equation}
\label{eq357}
\mathbb{P}\left(\inf_{0 \leq x \leq z \leq 3(l(m)+y))} \; X(z) - X(x) - \sum_{k=N(x)}^{N(z)}\mathbb{E}\left[     w_{v(i)}\right] \leq -y\right) \leq A\exp\left(\frac{-y^2}{A(yn^{1/3}+x)}\right).\end{equation}
By equation \eqref{eq1456}:
\begin{equation}
\label{eq80}
3(l(m)+y) \leq 4m+3y.
\end{equation}
Let:
$$\mathcal{C}= \left\{ \inf_{0 \leq x \leq z \leq  4m+3y} \; X(z) - X(x) - \sum_{k=N(x)}^{N(z)}\mathbb{E}\left[     w_{v(i)}\right] \leq -y\right\},$$
and:
$$\mathcal{B} = \left\{X(4m+y) - \sum_{k=0}^{N(4m+3y)}\mathbb{E}\left[     w_{v(i)}\right] \leq -y/2\right\}.$$
Also, write 
$$(x^*,z^*) = \inf\left\{0 \leq x \leq z \leq 4m+3y : \; X(z) - X(x) - \sum_{k=N(x)}^{N(z)}\mathbb{E}\left[     w_{v(i)}\right] \leq -y \right\},$$
where the infimum is taken in lexicographical order. And, by convention, $\inf({\emptyset}) = (0,4m+3y)$. 
Let:
$$\mathcal{D} := \left\{X(x^*) - \sum_{k=1}^{N(x^*)}\mathbb{E}\left[     w_{v(k)}\right] \geq y/4\right\} \, \textnormal{or} \, \left\{X(4m+y) - X(z^*) - \sum_{k=N(z^*)}^{N(4m+3y)}\mathbb{E}\left[     w_{v(k)}\right] \geq y/4\right\}.$$ 
If $\mathcal{C}$ happens then  one of the events $\mathcal{B}$ or $\mathcal{D}$ happens.
By Lemma \ref{wut}:
\begin{equation}
\label{eq33}
\mathbb{P}(\mathcal{B}) \leq A\exp\left(\frac{-y^2}{A(yn^{1/3}+m)}\right).
\end{equation}
By Lemma \ref{eqfff} and union bound:
\begin{equation}
\begin{aligned}
\label{eq3554}
\mathbb{P}\left(X(x^*) - \sum_{k=1}^{N(x^*)}\mathbb{E}\left[     w_{v(k)}\right] \geq y/4\right)
\leq &\mathbb{P}\left(\sup_{t \leq 4m+3y}X(t) - \sum_{k=1}^{N(t)}\mathbb{E}\left[     w_{v(k)}\right] \geq y/4\right)\\
\leq &\mathbb{P}\left(\sup_{t \leq 4m+3y}X(t) - \mathbb{E}\left[     X(t)\right] \geq y/8\right) \\ +&\mathbb{P}\left(\sup_{t \leq 4m+3y}N(t) - \mathbb{E}\left[   N(t)\right] \geq \frac{y}{A}+1\right),
\end{aligned}
\end{equation}
where $ A >0$ is the positive constant that appears in Lemma \ref{eqfff}.
And by the same arguments, using Lemma \ref{eqfff1} gives:
\begin{equation}
\begin{aligned}
\label{eq34}
&\mathbb{P}\left(X(4m+3y) - X(z^*) - \sum_{k=N(z^*)}^{N(4m+3y)}\mathbb{E}\left[     w_{v(k)}\right] \geq y/4\right) \\
\leq &\mathbb{P}\left(\sup_{t \leq 4m+3y}X(4m+3y) - X(t) - \sum_{k=N(t)}^{N(4m+3y)}\mathbb{E}\left[     w_{v(k)}\right] \geq y/4\right)\\
\leq &\mathbb{P}\left(\sup_{t \leq 4m+3y}X(4m+3y) - X(t) - \mathbb{E}\left[     X(4m+3y) - X(t)\right] \geq y/8\right)\\ +&\mathbb{P}\left(\sup_{t \leq 4m+3y}N(4m+3y) - N(t) - \mathbb{E}\left[   N(4m+3y) - N(t)\right] \geq \frac{y}{A'}+1\right).
\end{aligned}
\end{equation}
The union bound using Inequality \eqref{eq3554} and \eqref{eq34} alongside Lemmas \ref{lema102},\ref{lema101}, \ref{lema100} and \ref{lema103} yields:
\begin{equation}
\begin{aligned}
\label{eq355}
\mathbb{P}(\mathcal{D})
\leq A''\exp\left(\frac{-y^2}{A''(yn^{1/3}+m)}\right)
\end{aligned}
\end{equation}
Hence, from Equations \eqref{eq33} and \eqref{eq355} we obtain:
\begin{equation}
\begin{aligned}
\label{second_half'}
\mathbb{P}(\mathcal{C}) &\leq \mathbb{P}(\mathcal{B}) + \mathbb{P}(\mathcal{D}) \\
&\leq  A'''\exp\left(\frac{-y^2}{A'''(yn^{1/3}+m)}\right).
\end{aligned}
\end{equation}
This proves Equation \eqref{eq357}. We can then bound Equation \eqref{toztoz} by using Equation \eqref{first_half} and Equation \eqref{second_half'} which finishes the proof.
\end{proof}

In the above theorems we started the sums from one for the sake of clarity. The following general theorem has a similar proof.

\begin{Theorem}
\label{concentration_final}
Suppose that Conditions \ref{cond_nodes} hold. There exists a constant $A >0$ such that, if $1 \leq l \leq m$ is such that $\left(\sqrt{m(m-l)},y\right)$ verify Conditions \ref{cond_2}, $m-l \rightarrow \infty$ and $y = O(m-l)$ then:
\begin{equation*}
\mathbb{P}\left[\sup_{l \leq i \leq j \leq m} \; \left | \sum_{k=i}^j w_{v(k)} - \mathbb{E}\left[\sum_{k=i}^j w_{v(k)}\right] \right | \geq y\right] \leq  A\exp\left(\frac{-y^2}{A(yn^{1/3}+(m-l))}\right).
\end{equation*}
\end{Theorem}

\section{Bounds on the exploration process}
In this section we prove concentration inequalities for the exploration process and related processes. These various inequalities will be used in the following sections. Recall that $f = o(n)$ is the critical parameter and  $p_f = \frac{1}{\ell_n}+\frac{f}{\ell_n^{4/3}}$. In the rest of this section we consider the BFW of $G(\textbf{W},p_f)$.\par

For $0\leq i \leq n$ and $0\leq j\leq n$ define: 
$$Y(i,j) = \mathbbm{1}(\textnormal{There is an edge between nodes} \; i \;  \textnormal{and} \; j).$$
Then by definition of the BFW we have:
\begin{equation}
\begin{aligned}
&L_0=1,\\
&X_{i+1}= \sum_{j \notin \mathcal{V}(i+L_i)} Y(v(i+1),j) -1, \\
&L_{i+1}= \max(L_i+X_{i+1},1).
\end{aligned}
\end{equation}
Recall also that:
\begin{equation}
\begin{aligned}
&L'_0=1,\\
&L'_{i+1}= L'_i+X_{i+1}.
\end{aligned}
\end{equation}
When seen as processes of $i$, $L'$ is equal to $L$ until we finish discovering the first connected component. After that $L'=L-1$ until the second connected component is discovered, then $L'=L-2$ and so on. Generally $L'$ is equal to $L$ minus the number of connected components fully discovered. We say that the process $L$ visits $0$ in 
$i$ if $L'_i=\min_{j \leq i}L'_j$. \par
One of the difficulties in studying this process lies in the fact that $X_{i+1}$ depends on $L_i$. In the case of simple 
Erdős–Rényi random graphs, \cite{LNB09} use a different exploration process where the children of a node being explored are taken uniformly. This allows them to use a simpler and close enough process in order to circumvent this problem. If we want to do like them, in our case the naive way to define such a process would be as follows, for $h \geq 0$:
\begin{equation*}
\begin{aligned}
&L^h_0=1,\\
&X^h_{i+1}= \sum_{j \notin \mathcal{V}(i+1+h)} Y(v(i+1),j)-1, \\
&L^h_{i+1}= L^h_i+X^h_{i+1}.
\end{aligned}
\end{equation*}

\begin{figure}[!htbp]
\centering
\includegraphics[width=0.9\textwidth]{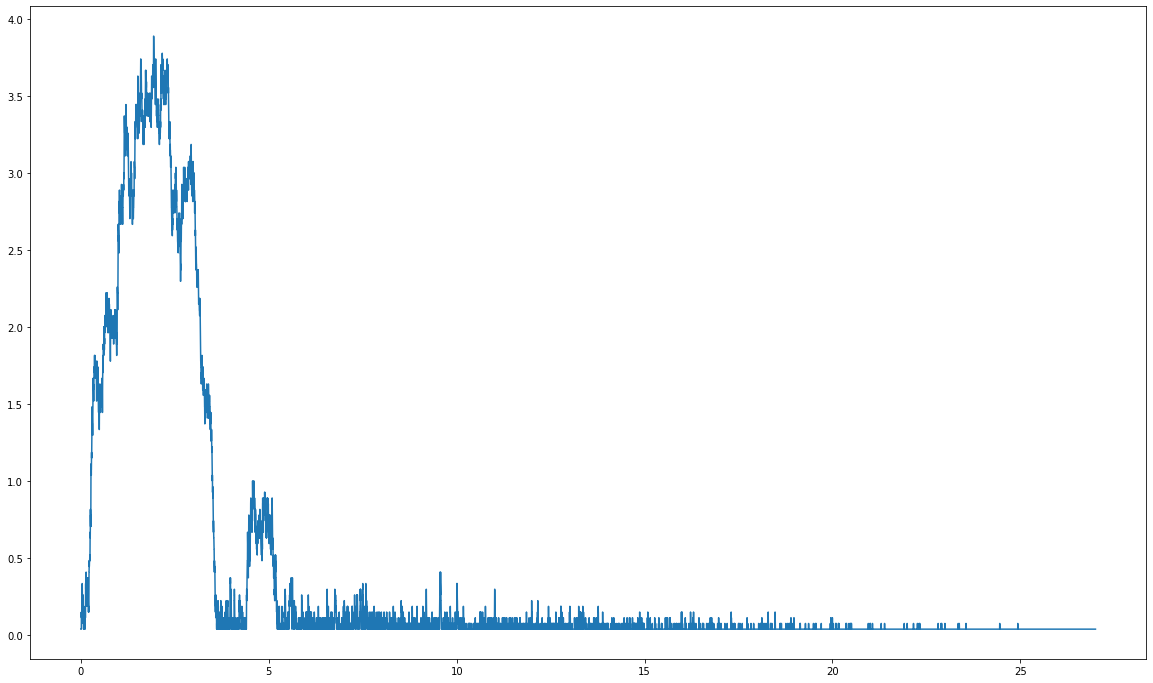}
\caption{The reflected exploration process of the graph in Figure \ref{figg1} with time rescaled by $20000^{2/3}$ and space is rescaled by $20000^{1/3}$.}
\end{figure}
In that case $L^0$ is always above $L'$ and in general  $L^h_i \leq L'_i$ as long as $L_i \leq h+1$. $L^0$ is used to bound $L'$ (and thus $L$) from above while $L^h$ for $h$ large enough would be used to bound it from below. 
However, in our case we sort the discovered children of a node by the weights of their edges. Hence, it is very likely that the indicator functions present in $L'_i$ but not in $L^h_i$ for $h > L_i$ will be equal to $1$ and hence $L^h_i$ would be too far away from $L'_i$. This is why we will use a martingale technique that we present now. \par
Note that for $i \geq 1$, $L_i$ is $\sigma(X_1,X_2,...,X_i)$ measurable. Let $(\mathcal{F}_i)_{i \geq 1}$ be the increasing sequence of $\sigma$-fields such that $\mathcal{F}_{i}$ is the $\sigma$-field generated by $\mathcal{V}(i+L_i)$ and the $(X_k)_{k \leq i}$'s, with the convention that $\mathcal{V}(k)=\mathcal{V}$ when $k \geq n$. Then for any $i\geq 1$, $X_i$ is measurable with respect to $\mathcal{F}_i$ and moreover we have:
$$
\mathbb{E}[X_i | \mathcal{F}_{i-1}] + 1 = \sum_{k > i + L_{i-1}-1}(1-e^{-w_{v(i)}w_{v(k)}p_f}).
$$
And we have the following fact
\begin{figure}[!htbp]
\centering
\includegraphics[width=0.9\textwidth, height=0.4
\textheight]{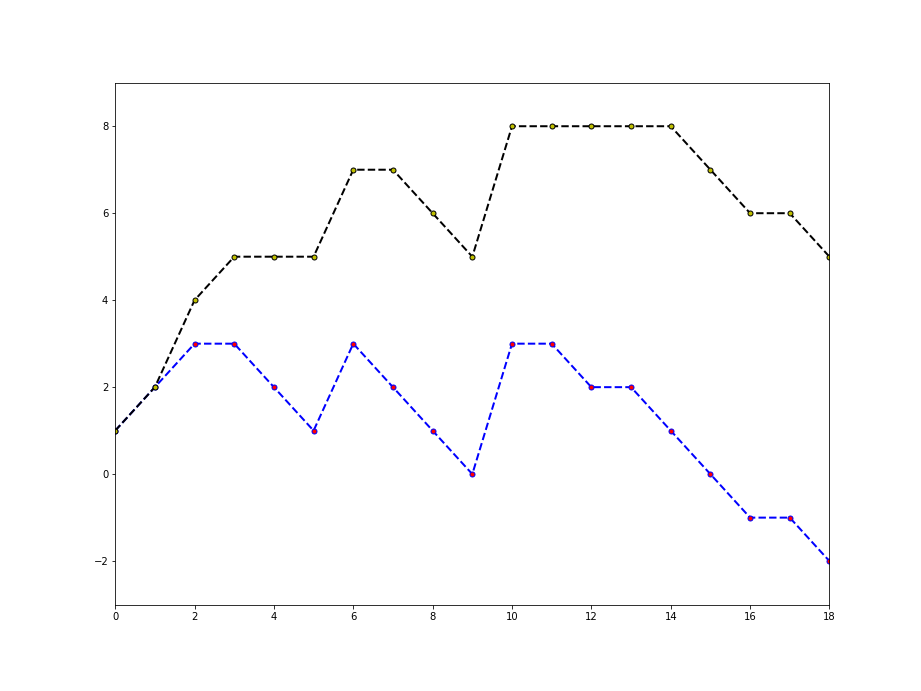}
\caption{In red with blue dashes, the exploration process of the graph in Figure \ref{fig1}. In yellow with black dashes, the process $L^0$ for the same graph. $L^0$ is always above $L'$.}
\end{figure}
\begin{Fact}
\label{fact1}
  Let 
  $$
  \tilde{L}_i= \sum_{k=0}^i\mathbb{E}[X_k | \mathcal{F}_{k-1}],
  $$
  with the convention that $X_0=1$. Then for any $l \geq 0$, the process $(L'(i)-L'(l)-(\tilde{L}_i-\tilde{L}_l))_{i \geq l}$ is a martingale with respect to $(\mathcal{F}_i)_{i \geq l}$. 
\end{Fact}
This fact allows us to use Bernstein's inequality for martingales (\cite{F75}). Then in order to bound $L'_i$ from below, we will use the fact that $(L'_i-\tilde{L}_i)_{i \geq 1}$ is a martingale, and for $i \geq 1$ as long as $L_i \leq h$ we have:
$$
\mathbb{E}[X_i | \mathcal{F}_{i-1}]+1 \geq \sum_{k > i + h}(1-e^{-w_{v(i)}w_{v(k)}p_f}).
$$
This is why we define the following process, for $i \geq 1$ and $h \geq 0$:
$$
\tilde{L}^h_m+m-1=\sum_{i=1}^m\sum_{k >i + h}(1-e^{-w_{v(i)}w_{v(k)}p_f}),
$$
then $\tilde{L}^h_m$ will be close to, and greater than $\tilde{L}_m$ as long as $h \geq L_i$ and $h$ is not too large.
A second important fact is that while constructing the exploration process, we never inspect the potential surplus edges, namely the $Y(v(i),v(j))$'s where $i \geq 1$ and $i+1 \leq j \leq i + L_i-1$. This means that:
\begin{Fact}
\label{fact2}
  Conditionally on $\mathcal{F}_n$, the $\sigma$-field generated by $\mathcal{V}$ and the $(X_k)_{k \leq n}$'s, the random variables 
  $$Y(v(i),v(j))_{1 \leq i, \, \leq j \leq i + L_i-1},$$ 
  are independent Bernoulli random variables of parameters 
  $$
  \left(1-e^{-w_{v(i)}w_{v(j)}p_f}\right)_{1 \leq i, \, i+1 \leq j \leq i + L_i-1}.
  $$
\end{Fact}
Moreover, for $h \geq 0$ and $ i \geq 0$ define 
$$\bar{L}'(k) = \sum_{i=0}^k X_{i}\mathbbm{1}(X_{i} \leq 2n^{1/3}),$$
 and if we write $d(i)$ for the degree of node $i$, then $d(i)$ is a sum of independent Bernoulli variables. Hence, when Conditions \ref{cond_nodes} hold, by the classical Bernstein inequality (\cite{B24}) we have: 
$$
\mathbb{P}(d(i) \geq w_i + n^{1/3}) \leq \exp\left(\frac{-(n^{1/3})^2}{2(n^{1/3}+w_i)}\right). 
$$
By using Conditions \ref{cond_nodes} we have for $n$ large enough:
\begin{equation}
\begin{aligned}
\mathbb{P}(\exists (k,h), \, \bar{L}'(k) \neq L'(k)) &\leq \sum_{i=1}^{n} \mathbb{P}(d(i) \geq w_i + n^{1/3}) \\
&\leq  \sum_{i=1}^{n} \exp\left(\frac{-(n^{1/3})^2}{2(n^{1/3}+w_i)}\right) \\
&\leq A\exp\left(\frac{-n^{1/3}}{A}\right),
\end{aligned}
\end{equation}
where $A$ is some large constant, this probability is smaller than the ones we will get in this section and the one following it. It is also clear that Fact \ref{fact1} also holds if we replace  $ L'(k)$ by $\bar{L}'(k)$ and $X_{i}$ by $ X_{i}\mathbbm{1}(X_{i} \leq 2n^{1/3})$. Hence, we will assume that the increments of $L'$ are smaller than $2n^{1/3}$. And we will assume the same of $L^0$. This will make computations lighter, as Bernstein's inequality requires a bound on the maximal increment of the process. We will not have to do a union bound in each calculation and consider the case where $\bar{L}'(k) \neq L'(k)$. This convention will be used up to Section $5$, after that we will use the fact that the increments of $L'(k)$ are even smaller when $k$ is large enough.

A direct corollary of Lemma \ref{esperance} is the following:
\begin{Corollary}
\label{S_h}
For all $m \geq l \geq 1$ such that $m = o(n)$, and $h=o(n)$:
$$
\mathbb{E}(\tilde{L}^h_m-\tilde{L}^h_{l-1}) = (m-l)\left(f\ell_n^{-1/3}-\frac{C(m+l)+2h}{2\ell_n}\right)+1+o\left(\frac{m^2-l^2+(m-l)(h+n^{2/3})}{n}\right).
$$
\end{Corollary}
\begin{proof}
For any $ l-1 \leq i \leq m$, let:
\begin{equation*}
\tilde{X}^h_{i+1}= \sum_{j \notin \mathcal{V}(i+1+h)}(1-e^{-w_{v(i+1)}w_{j}p_f})-1,
\end{equation*}
then:
\begin{equation*}
\tilde{L}^h_{i+1}= \tilde{L}^h_i+\tilde{X}^h_{i+1},
\end{equation*}
and:
\begin{equation}
\label{eq110}
\mathbb{E}[\tilde{X}^h_i]+1=\mathbb{E}\left[\sum_{j\geq i+1+h}1-\exp\left(-w_{v(i)}w_{v(j)}p_f\right)\right].
\end{equation}
By Conditions \ref{cond_nodes}, $w_{v(i)}w_{v(j)}p_f = o(1)$ deterministically for any $(i,j)$. The bounds giving $O$ and $o$ in the following expectations can thus be chosen to be deterministic. By Equation \eqref{expx} we have:
\begin{equation}
\begin{aligned}
\label{eq111}
\mathbb{E}[\tilde{X}^h_i]+1&= \mathbb{E}\left[\sum_{j\geq i+1+h}w_{v(i)}w_{v(j)}p_f(1+O(w_{v(i)}w_{v(j)}p_f))\right] \\
&= \mathbb{E}\left[w_{v(i)}\left(1+f\ell_n^{-1/3}+O\left(\sum_{j=1}^{n}w_{v(i)}w_{v(j)}^2p_f^2\right)\right)-\sum_{j< i+1+h}w_{v(i)}w_{v(j)}p_f(1+o(1))\right] \\
&= \mathbb{E}\left[w_{v(i)}\left(1+f\ell_n^{-1/3}+o\left(n^{-2/3}\right)\right)-\sum_{j< i+1+h}w_{v(i)}w_{v(j)}p_f(1+o(1))\right].
\end{aligned}
\end{equation}
We use Lemmas \ref{mean_pow_weights_2} and \ref{esperance} to do the proper replacements in Equation \eqref{eq111}:
\begin{equation*}
\begin{aligned}
\mathbb{E}[\tilde{X}^h_i]= &-1+\left(1+\frac{i(1-C)}{\ell_n}+o\left(\frac{i+n^{2/3}}{n}\right)\right)\left(1+\frac{f}{\ell_n^{1/3}}\right)(1+o(1))\\
&-\mathbb{E}\left[\sum_{j< i+1+h}w_{v(i)}w_{v(j)}p_f(1+o(1))\right].
\end{aligned}
\end{equation*}
Finally, Lemma \ref{mean_corr} yields:
\begin{equation*}
\begin{aligned}
\mathbb{E}[\tilde{X}^h_i]&= -1+\left(1+\frac{i(1-C)}{\ell_n}+o\left(\frac{i+n^{2/3}}{n}\right)\right)\left(1+\frac{f}{\ell_n^{1/3}}\right)(1+o(1))-\frac{i+h}{\ell_n}(1+o(1)).
\end{aligned}
\end{equation*}
Summing over $i$ ends the proof.
\end{proof}
We will first show concentration results for $\tilde{L}^h$ before moving to $L$. We start by stating a set of conditions that will ensure the theorems holds.
\begin{Conditions}
\label{cond_3}
We say that $(a(n),b(n),c(n),d(n))$ verifies Conditions \ref{cond_3} if:
$$a(n)+c(n)=o(n),$$ 
and:
$$\lim_{n}(a(n)-b(n)) = +\infty,$$
and
$$d(n) = O(a(n)-b(n)),
$$
and $$\left(\sqrt{(a(n)-b(n))(a(n)+c(n))},d(n)\right)$$ verify Conditions \ref{cond_2}.
\end{Conditions}
We start with the following technical lemma. Concentration follows here from the concentration of the ordered weights proved in the previous section.
\begin{Lemma}
\label{inf_s_0}
Suppose that Conditions \ref{cond_nodes} hold. There exists a constant $A >0$ such that, if $(m,l,h,y)$ verifies Conditions \ref{cond_3}, then the following is true:
$$\mathbb{P}\left(\sup_{l \leq i \leq j \leq m}\left(\middle|\tilde{L}^h_j-\tilde{L}^h_i -\mathbb{E}\left[\tilde{L}^h_j-\tilde{L}^h_i\right]\middle|\right) \geq y\right) \leq A\exp\left(\frac{-y^2}{A(yn^{1/3}+m-l)}\right),$$
\end{Lemma}

\begin{proof}
Let:
$$
D = \mathbb{P}\left(\sup_{l \leq i \leq j \leq m}\left(\middle|\tilde{L}^h_j-\tilde{L}^h_i -\mathbb{E}\left[\tilde{L}^h_j-\tilde{L}^h_i\right]\middle|\right) \geq y\right)
$$
Since $p_f \geq 1/n$ and $m-l = o(n)$. Conditions \ref{cond_nodes} and Equation \eqref{expx} yield:
\begin{equation*}
\begin{aligned}
&\sum_{k=i+1}^j\sum_{k' > k+h} \left(1-e^{-w_{v(k)}w_{v(k')}p_f} -\mathbb{E}\left[ 1-e^{-w_{v(k)}w_{v(k')}p_f}\right]\right) \\
&=\sum_{k=i+1}^j\left(\sum_{k' > k+h} w_{v(k)}w_{v(k')}p_f-\mathbb{E}\left[w_{v(k)}w_{v(k')}p_f\right]\right) + O(1).
\end{aligned}
\end{equation*}
Moreover, recall, by our conditions, that $y = y(n)$ and $\lim_{n \rightarrow \infty}y(n) = +\infty$. Since 
$$
\sum_{k' > k+h}w_{v(k)} w_{v(k')}p_f = \sum_{k'=1}^nw_{v(k)} w_{v(k')}p_f - \sum_{k' \leq k+h}w_{v(k)} w_{v(k')}p_f,
$$ 
we obtain by the union bound for $n$ large enough:
\begin{equation}
\begin{aligned}
\label{zhlh}
D &\leq \mathbb{P}\left(\sup_{l \leq i \leq j \leq m}\left|\sum_{k=i+1}^j\left(\sum_{k' > k+h}w_{v(k)} w_{v(k')}p_f\right)-\mathbb{E}\left[\sum_{k=i+1}^j\left(\sum_{k' > k+h} w_{v(k)}w_{v(k')}p_f\right)\right]\right| \geq y/2\right) \\
&\leq \mathbb{P}\left(\sup_{l \leq i \leq j \leq m}\left|\sum_{k=i+1}^j\left(\sum_{k' \leq k+h}w_{v(k)} w_{v(k')}p_f\right)-\mathbb{E}\left[\sum_{k=i+1}^j\left(\sum_{k' \leq k+h} w_{v(k)}w_{v(k')}p_f\right)\right]\right| \geq y/4\right)\\
&+\mathbb{P}\left(\sup_{l \leq i \leq j \leq m}\left|\sum_{k=i+1}^jw_{v(k)}-\mathbb{E}\left[\sum_{k=i+1}^jw_{v(k)}\right]\right| \geq \frac{y}{4\ell_np_f}\right).
\end{aligned}
\end{equation}
Since $\ell_np_f \leq 2$, by Conditions \ref{cond_3} we can apply Theorem \ref{concentration_final} with $(m,l,y)$ to obtain:
\begin{equation}
\label{zhlzhl}
\mathbb{P}\left(\sup_{l \leq i \leq j \leq m}\left|\sum_{k=i+1}^jw_{v(k)}-\mathbb{E}\left[\sum_{k=i+1}^jw_{v(k)}\right]\right| \geq \frac{y}{4\ell_np_f}\right) \leq A\exp\left(\frac{-y^2}{A(yn^{1/3}+m-l)}\right).
\end{equation}
By injecting Inequality \eqref{zhlzhl} in Inequality \eqref{zhlh}, bounding $D$ amounts to bounding:
$$
\mathbb{P}\left(\sup_{l \leq i \leq j \leq m}\left|\sum_{k=i+1}^j\left(\sum_{k' \leq k+h}w_{v(k)} w_{v(k')}p_f\right)-\mathbb{E}\left[\sum_{k=i+1}^j\left(\sum_{k' \leq k+h} w_{v(k)}w_{v(k')}p_f\right)\right]\right| \geq y/4\right).
$$
We focus on proving a one-sided version of this inequality, the other half of the inequality is proven similarly:
$$
\mathbb{P}\left(\sup_{l \leq i \leq j \leq m}\left(\sum_{k=i+1}^j\left(\sum_{k' \leq k+h}w_{v(k)} w_{v(k')}p_f\right)-\mathbb{E}\left[\sum_{k=i+1}^j\left(\sum_{k' \leq k+h} w_{v(k)}w_{v(k')}p_f\right)\right]\right) \geq y/4\right).
$$
By Lemmas \ref{mean_pow_weights_2} and \ref{mean_corr}, for any $l \leq i \leq j \leq m$: 
\begin{equation}
\label{eq1111}
\mathbb{E}\left[\sum_{k=i+1}^j\left(\sum_{k' \leq k+h}w_{v(k)} w_{v(k')}p_f\right)\right] = \frac{j^2-i^2+2(j-i)h}{2\ell_n}(1+o(1)).
\end{equation}
By a simple computation,  Conditions \ref{cond_3} imply that $\left(m+h,\frac{y(m+h)}{16(m-l)}\right)$ verify Conditions \ref{cond_2}. Using this with Theorem \ref{concentration_final} yields for $n$ large enough:
\begin{equation*}
\begin{aligned}
&\mathbb{P}\left(\sup_{l \leq k \leq m}  \left|\sum_{k' \leq k+h} w_{v(j)}-\mathbb{E}\left[\sum_{k' \leq k+h} w_{v(j)}\right]\right| \geq \frac{y}{16p_f(m-l)}\right)\\
&\leq \mathbb{P}\left(\sup_{1 \leq k \leq m+h}  \left|\sum_{k' = 1}^{k} w_{v(j)}-\mathbb{E}\left[\sum_{k' = 1}^k w_{v(j)}\right]\right| \geq \frac{y(m+h)}{16(m-l)}\right)\\
&\leq A\exp\left(\frac{-y^2(m+h)^2}{A(y(m+h)(m-l)n^{1/3}+(m+h)(m-l)^2}\right) \\
&\leq A\exp\left(\frac{-y^2}{A(yn^{1/3}+(m-l)}\right). \\
\end{aligned}
\end{equation*}
Hence, by the above inequality and Equation \eqref{eq1111} the union bound yields:
\begin{equation}
\begin{aligned}
\label{eq1112}
&\mathbb{P}\left(\sup_{l \leq i \leq j \leq m}\left(\sum_{k=i+1}^jw_{v(k)}\left(\sum_{k' \leq k+h} w_{v(k')}p_f\right)-\mathbb{E}\left[\sum_{k=i+1}^j\left(\sum_{k' \leq k+h} w_{v(k)}w_{v(k')}p_f\right)\right]\right) \geq y/4\right) \\
\leq &\mathbb{P}\left(\sup_{l \leq k \leq m}  \left|\sum_{k' \leq k+h} w_{v(k')}-\mathbb{E}\left[\sum_{k' \leq k+h} w_{v(k')}\right]\right| \geq \frac{y}{16p_f(m-l)}\right)\\
&+ \mathbb{P}\left(\sup_{l \leq i \leq j \leq m}\left(\sum_{k=i+1}^jw_{v(k)}\left(\frac{y}{16(m-l)}\right)\right) \geq y/8\right) \\
&+ \mathbb{P}\left(\sup_{l \leq i \leq j \leq m}\left(\sum_{k=i+1}^jw_{v(k)}\mathbb{E}\left[\sum_{k' \leq k+h} w_{v(k')}p_f\right]-\frac{j^2-i^2+2(j-i)h}{2\ell_n}(1+o(1))\right) \geq y/8\right) \\
\leq  &A\exp\left(\frac{-y^2}{A(yn^{1/3}+(m-l)}\right)+\mathbb{P}\left(\sup_{l \leq i \leq j \leq m}\left(\sum_{k=i+1}^jw_{v(k)}\left(\frac{y}{16(m-l)}\right)\right) \geq y/8\right) \\
&+ \mathbb{P}\left(\sup_{l \leq i \leq j \leq m}\left(\sum_{k=i+1}^jw_{v(k)}\mathbb{E}\left[\sum_{k' \leq k+h} w_{v(k')}p_f\right]-\frac{j^2-i^2+2(j-i)h}{2\ell_n}(1+o(1))\right) \geq y/8\right).
\end{aligned}
\end{equation}
By Lemma \ref{mean_weights}, for any $k \leq m$:
\begin{equation}
\mathbb{E}\left[\sum_{k' \leq k+h} w_{v(k')}p_f\right] = \frac{(k+h)(1+o(1))}{\ell_n}.
\end{equation}
Moreover, notice that for any $l \leq i \leq j \leq m$:
$$
\sum_{k=i+1}^jw_{v(k)}\frac{k+h}{\ell_n}=\frac{i}{\ell_n}\sum_{k=i+1}^jw_{v(k)}+\frac{h}{\ell_n}\sum_{k=i+1}^jw_{v(k)}+\left(\frac{1}{\ell_n}\right)\sum_{k=i+1}^j\sum_{k'=k}^jw_{v(k')}.
$$
By Conditions \ref{cond_3}, we have for any for any $l \leq i \leq j \leq m$: 
\begin{equation}
\begin{aligned}
\label{fact4}
 \frac{j^2-i^2+2(j-i)h}{2\ell_n} = O(y).
 \end{aligned}
\end{equation}
Moreover, by Conditions \ref{cond_3} $y \leq A(m-l)$, for some large constant $A > 0$. Hence, by the union bound, Equation \eqref{eq1112} becomes:
\begin{equation}
\begin{aligned}
\label{eq1122}
&\mathbb{P}\left(\sup_{l \leq i \leq j \leq m}\left(\sum_{k=i+1}^jw_{v(k)}\left(\sum_{k' \leq k+h} w_{v(k')}p_f\right)-\mathbb{E}\left[\sum_{k=i+1}^j\left(\sum_{k' \leq k+h} w_{v(k)}w_{v(k')}p_f\right)\right]\right) \geq y/4\right) \\
\leq &\mathbb{P}\left(\sup_{l \leq i \leq j \leq m}\left(\frac{i}{\ell_n}\sum_{k=i+1}^j(w_{v(k)}-\mathbb{E}[w_{v(k)}])(1+o(1))\right) \geq \frac{y}{48}\right) \\
&+ \mathbb{P}\left(\sup_{l \leq i \leq j \leq m}\left(\left(\frac{h}{\ell_n}\right)\sum_{k=i+1}^j(w_{v(k)}-\mathbb{E}[w_{v(k)}])(1+o(1))\right) \geq \frac{y}{48}\right) \\
&+ \mathbb{P}\left(\sup_{l \leq i \leq j \leq m}\left(\left(\frac{1}{\ell_n}\right)\sum_{k=i+1}^j\sum_{k'=k}^j(w_{v(k')}-\mathbb{E}[w_{v(k')}])(1+o(1))\right) \geq \frac{y}{48}\right) \\
&+ A\exp\left(\frac{-y^2}{A(yn^{1/3}+(m-l)}\right)+\mathbb{P}\left(\sup_{l \leq i \leq j \leq m}\left(\sum_{k=i+1}^jw_{v(k)}\right) \geq \frac{2y}{A}\right).
\end{aligned}
\end{equation}
Notice that we implicitly use Equation \eqref{fact4} in the above Inequality in order to make the $o$ factors match and at the cost of taking $y/48$. This is why we are able to write:
$$\mathbb{P}\left(\sup_{l \leq i \leq j \leq m}\left(\frac{i}{\ell_n}\sum_{k=i+1}^j(w_{v(k)}-\mathbb{E}[w_{v(k)}])(1+o(1))\right) \geq \frac{y}{48}\right),$$
instead of:
$$
\mathbb{P}\left(\sup_{l \leq i \leq j \leq m}\left(\frac{i}{\ell_n}\sum_{k=i+1}^j\left(w_{v(k)}(1+o(1))-\mathbb{E}[w_{v(k)}](1+o(1))\right)\right) \geq \frac{y}{24}\right).$$
By Conditions \ref{cond_3} we can apply Theorem \ref{concentration_final} with $(m-l,y/48)$ to obtain:
\begin{equation}
\label{zhlzhl1}
\mathbb{P}\left(\sup_{l \leq i \leq j \leq m}\left|\sum_{k=i+1}^jw_{v(k)}-\mathbb{E}\left[\sum_{k=i+1}^jw_{v(k)}\right]\right| \geq \frac{y}{48}\right) \leq A\exp\left(\frac{-y^2}{A(yn^{1/3}+m-l)}\right).
\end{equation}
We finish by noticing that the first three probabilities in  the right-hand side of Inequality \ref{eq1122} are all smaller than the left hand-side of \ref{zhlzhl1}.
\end{proof}
\begin{Theorem}
\label{inf_s_0_3}
Suppose that Conditions \ref{cond_nodes} hold. There exists a constant $A >0$ such that, if $(m,l,0,y)$ verifies Conditions \ref{cond_3}, then the following holds:
\begin{equation*}
\mathbb{P}\left(\sup_{l \leq u \leq w\leq m}\left|L^0_w-L^0_u-\mathbb{E}[L^0_w-L^0_u]\right| \geq y \right)\leq  A\exp\left(\frac{-y^2}{A(yn^{1/3}+m-l)}\right).
\end{equation*}
\end{Theorem}
\begin{proof} 
For $i \geq 0$ let $\mathcal{F}^0_{i}$ be the sigma-field generated by $\mathcal{V}_{i+1}$ and the random variables $(X^0_k)_{k\leq i}$. Write:
$$D_1 = \mathbb{P}\left(\sup_{l \leq u \leq w\leq m}\left|L^0_w-L^0_u-\sum_{i=u+1}^w\mathbb{E}\left[X^0_i | \mathcal{F}^0_{i-1}\right]\right| \geq y/2\right),$$
and
$$D_2 = \mathbb{P}\left(\sup_{l \leq u \leq w\leq m}\left|\sum_{i=u+1}^w\mathbb{E}\left[X^0_i | \mathcal{F}^0_{i-1}\right]-\mathbb{E}[L^0_w-L^0_u]\right| \geq y/2\right).$$
Then, by the union bound:
\begin{equation*}
\begin{aligned}
\label{rororo11}
\mathbb{P}\left(\sup_{l \leq u \leq w\leq m}\left|L^0_w-L^0_u-\mathbb{E}[L^0_w-L^0_u]\right| \geq y \right)&\leq  D_1+D_2.
\end{aligned}
\end{equation*}
We start by bounding $D_1$. We have by the union bound:
\begin{equation*}
\begin{aligned}
D_1 \leq &\mathbb{P}\left(\sup_{l \leq u \leq m}\left|L^0_u-L^0_l-\sum_{i=l+1}^u\mathbb{E}\left[X^0_i | \mathcal{F}^0_{i-1}\right]\right| \geq y/4\right)\\
&+ \mathbb{P}\left(\sup_{l \leq w \leq m}\left|L^0_w-L^0_l-\sum_{i=l+1}^w\mathbb{E}\left[X^0_i | \mathcal{F}^0_{i-1}\right]\right| \geq y/4\right).
\end{aligned}
\end{equation*} 
Notice that 
$$
\left(L^0_w-L^0_l-\sum_{i=l+1}^w\mathbb{E}\left[X^0_i | \mathcal{F}^0_{i-1}\right]\right)_{w\geq l},
$$
is a martingale with respect to the $(\mathcal{F}^0_{i})_{i\geq l}$'s.
Moreover:
$$
 \mathbb{E}[(X^0_i)^2| \mathcal{F}^0_{i-1}]  = \sum_{k \not\in \mathcal{V}_i }\sum_{k' \not\in \mathcal{V}_i}\mathbb{E}[Y(v(i),k)Y(v(i),k')|\mathcal{F}^0_{i-1}] \leq w_{v(i)}+w_{v(i)}^2.
$$
Applying Theorem \ref{replacement} to the $(w_{v(i)}^2)$'s, let $J(1),J(2),...$ be i.i.d copies of $v(l+1)$.  We have by Lemma \ref{mean_pow_weights_2} the following Bernstein inequality:
\begin{equation}
\begin{aligned}
\label{roko}
&\mathbb{P}\left(\sum_{i=l+1}^m w_{v(i)}^2 \geq \sum_{i=l+1}^m 2\mathbb{E}[w_{v(i)}^2]+2yn^{1/3}\right) \\
&\leq \mathbb{E}\left[\mathbb{E}\left[\frac{\exp\left(\sum_{i=l+1}^m w_{v(i)}^2\right)}{\exp\left(\sum_{i=l+1}^m 2C\mathbb{E}[w_{v(i)}^2]+2yn^{1/3}\right)}\right]\right] \\
&\leq \mathbb{E}\left[\mathbb{E}\left[\frac{\exp\left(\sum_{i=l+1}^m w_{J(i)}^2\right)}{\exp\left(2C(m-l)(1+o(1))+2yn^{1/3}\right)}\right]\right] \\
&\leq \mathbb{E}\left[\exp\left(\frac{-\left|2C(m-l)(1+o(1))+2yn^{1/3}-\sum_{i=l+1}^m \mathbb{E}[w_{J(i)}^2]\right|_{+}^2}{\left(Ayn+A(m-l)n^{2/3}+\sum_{i=l+1}^m \mathbb{E}[w_{J(i)}^4]\right)}\right)\right] \\
&\leq \mathbb{E}\left[\exp\left(\frac{-\left|2C(m-l)(1+o(1))+2yn^{1/3}-\sum_{i=l+1}^m \mathbb{E}[w_{J(i)}^2]\right|_{+}^2}{An^{2/3}\left(yn^{1/3}+(m-l)+\sum_{i=l+1}^m \mathbb{E}[w_{J(i)}^2]\right)}\right)\right],
\end{aligned}
\end{equation}
where line $3$ of the equation is a Chernoff bound which yields Bernstein's inequality in line $4$ (as in the original proof of \cite{B24}), and we used the fact that, by Conditions \ref{cond_nodes}, we have $\mathbb{E}[w_{J(i)}^4] \leq n^{2/3}\mathbb{E}[w_{J(i)}^2]$.
Now notice that by definition, and by Lemma \ref{mean_pow_weights_2}, for any $i \geq l+1$:
\begin{equation}
\label{roko22}
\mathbb{E}[w_{J(i)}^2] =  C(1+o(1)).
\end{equation}
Since $(m,l,0,y)$ verifies Conditions \ref{cond_3}, we can apply Theorem \ref{concentration_final} on $(m,l,2y)$ to obtain:
\begin{equation}
\label{roko23}
\mathbb{P}\left(\sum_{i=l+1}^m w_{v(i)} \geq \sum_{i=l+1}^m \mathbb{E}[w_{v(i)}]+2y\right) \leq A\exp\left(\frac{-y^2}{A\left(yn^{1/3}+(m-l))\right)}\right),
\end{equation}
where $A >0$ is a large enough constant.
By Equations \eqref{roko},  \eqref{roko22} and \eqref{roko23} and by Bernstein's inequality for martingales (Theorem $2.1$ in \cite{F75}) we obtain:
\begin{equation}
\begin{aligned}
\label{rororo2}
D_1
&\leq A'\exp\left(\frac{-y^2}{A'(yn^{1/3}+m-l)}\right).
\end{aligned}
\end{equation} \par
In order to bound $D_2$ notice that the sum inside $D_2$ is equal to the one in Lemma \ref{inf_s_0} when $h=0$ by definition. This finishes the proof.
\end{proof}

Since $L^0$ is always greater than $L'$ deterministically, Theorem \ref{inf_s_0_3} gives us the following theorem.
\begin{Theorem}
\label{sup_l}
Suppose that Conditions \ref{cond_nodes} hold. Let $\frac{4f\ell_n^{2/3}}{C} \geq m  \geq \frac{f\ell_n^{2/3}}{C}$, then 
there exists $A >0$ and $A' > 0$ such that for any $\epsilon >0$:
\begin{equation*}
\begin{aligned}
\mathbb{P}\left(\sup_{1 \leq i \leq m}( L_i) \geq \frac{10f^2\ell_n^{1/3}}{C}\right) &\leq A\exp\left(\frac{-f}{A}\right).
\end{aligned}
\end{equation*}
\end{Theorem}
\begin{proof}
By definition: 
$$
\sup_{1 \leq i \leq m}( L_i)  \leq \sup_{1 \leq u \leq v \leq m}(L^0_v-L^0_u).
$$
Hence
\begin{equation}
\begin{aligned}
\label{Z_m}
&\mathbb{P}\left(\sup_{1 \leq i \leq m}( L_i) \geq \frac{10f^2\ell_n^{1/3}}{C} \right) \leq \mathbb{P}\left(\sup_{1 \leq u \leq v \leq m}(L^0_v-L^0_u) \geq \frac{10f^2\ell_n^{1/3}}{C} \right).
\end{aligned}
\end{equation}
From Corollary \ref{S_h}: 
\begin{equation}
\begin{aligned}
\label{wasali}
\min_{ 1 \leq u \leq v \leq m}(\mathbb{E}[L^0_v-L^0_h])  &= \min_{ 1 \leq u \leq v \leq m}\left((v-u)\left(f\ell_n^{-1/3}-\frac{C(v+u)}{2\ell_n}\right)(1+o(1))+1\right) \\
&\geq \min_{ 1 \leq u \leq v \leq m}\left(-\frac{C(v+u)(v-u)}{2\ell_n}(1+o(1))+1\right) \\
&\geq \frac{-9f^2\ell_n^{1/3}}{C}.
\end{aligned}
\end{equation} 
We finish by injecting Equation \eqref{wasali} in \eqref{Z_m} and using Theorem \ref{inf_s_0_3} with $\left(m,1,0,\frac{f^2\ell_n^{1/3}}{C}\right)$.
\end{proof}
The same method that we used to bound the term $D_1$ in the proof of Theorem \ref{inf_s_0_3} directly yields
\begin{Theorem}
\label{inf_s_1}
Suppose that Conditions \ref{cond_nodes} hold. There exists a constant $A >0$ such that, if $(m,l,0,y)$ verifies Conditions \ref{cond_3}, then the following holds:
\begin{equation*}
\mathbb{P}\left(\sup_{l \leq u \leq v \leq m}\left|L'_v-L'_u-\mathbb{E}[\tilde{L}_v-\tilde{L}_u]\right| \geq y \right)\leq  A\exp\left(\frac{-y^2}{A(yn^{1/3}+m-l)}\right).
\end{equation*}
\end{Theorem}

\section{The structure of the giant component}
The bounds in the previous section will allow us to determine the structure of the giant component of $G(\textbf{W},p_f)$. We write  $H^*_f$ for the component of $G(\textbf{W},p_f)$ being explored at time $\frac{f\ell_n^{2/3}}{C}$. We will prove that this component is the largest one with high enough probability. Informally, the BFW has a random unbiased part plus a drift (its expectation). Corollary \ref{S_h} shows that the drift of $L^0$ is a parabola that has its maximum at $\frac{f\ell_n^{2/3}}{C}$. Given the concentration of  $L^0$, and if we  also assume that it behaves like $L$,  it follows that $L$ also has its maximum around $\frac{f\ell_n^{2/3}}{C}$. Now recall that $L$ corresponds to the number of nodes discovered but not yet explored. It is then naturally  maximal when the exploration process is in a large connected component. Hence $H^*_f$ should be the largest component. In this section we will prove this rigorously. Then we will prove in the following section that the other connected components are small enough. \footnote{In the rest of the proof, and in order to ease notations we do not use integer part notations for the indices and instead abuse notation by using real indices in our sums sometimes.}

\begin{figure}[!htbp]
\centering
\includegraphics[width=0.9\textwidth]{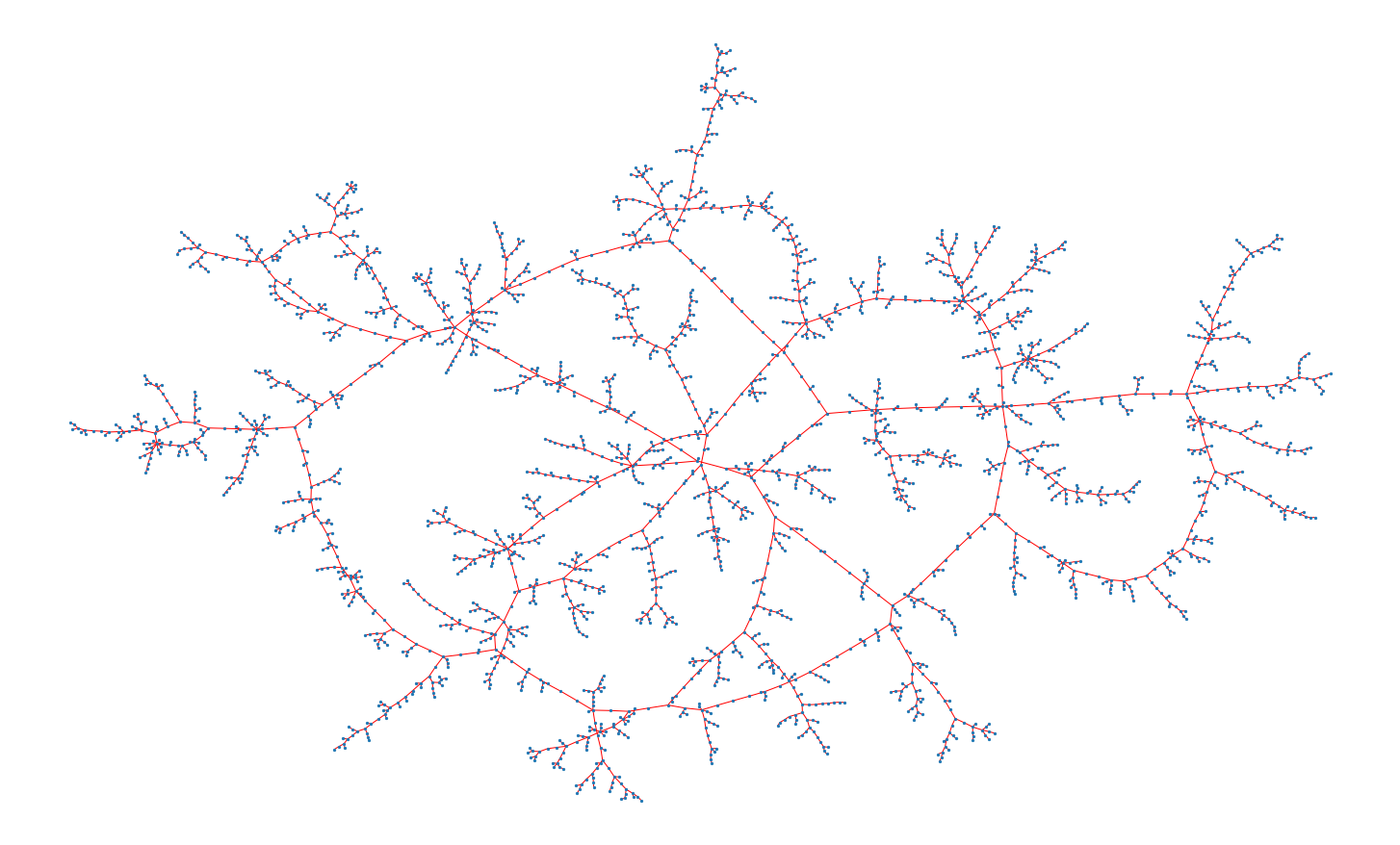}
\caption{The largest connected component of the graph in Figure \ref{figg1}. Its size is $2654$.}
\end{figure}

\subsection{The size of the giant component}
\begin{Theorem}
\label{lower_bound_1}
Suppose that Conditions \ref{cond_nodes} hold. Let $1 > \epsilon' > 0$. For $f$ large enough and for any $1 \geq \epsilon > 0$ consider the following event:  \\
The exploration of $H_f^*$ starts before time $\frac{\ell_n^{2/3}}{f^{1-\epsilon}C}$ and ends between times $\frac{2(1-\epsilon')f\ell_n^{2/3}}{C}$ and $\frac{2(1+\epsilon')f\ell_n^{2/3}}{C}$. \\
Then there exists a positive constant $A > 0$ such that  the probability of this event not happening is at most
$$A\exp\left(\frac{-f^{\epsilon}}{A}\right).$$
\end{Theorem}
\begin{proof}
Let  $t_1 = \frac{\ell_n^{2/3}}{f^{1-\epsilon}C}$,  $t_2 = \frac{2(1-\epsilon')f\ell_n^{2/3}}{C}$ and $t_3 = \frac{2(1+\epsilon')f\ell_n^{2/3}}{C}$. \\
In order to prove this theorem we need to bound the probability that $L$ visits zero between times $t_1$ and $t_2$ and also the probability that $L$ does not visit $0$ between times $t_2$ and $t_3$. Recall that for any $i$:
$$
\tilde{L}_i = \sum_{k=1}^i\mathbb{E}[X_k|\mathcal{F}_{k-1}].
$$
We start by the probability of the first event. Recall that by definition $L \geq L'$. We will thus focus on $L'$. For any $h > 0$, $\tilde{L}$ is at least $\tilde{L}^h$ until the first time $i$ when $L_i \geq h$. \par
Let $h=\frac{10f^2\ell_n^{1/3}}{C}$.
Then by Theorem \ref{sup_l} and Conditions \ref{cond_nodes}: \begin{equation}
\label{eq611}
\mathbb{P}\left(\sup_{1 \leq j \leq t_2} L_j \geq h\right) \leq A\exp\left(\frac{-f}{A}\right).
\end{equation}
Now divide the interval $[t_1,t_2]$ by introducing intervals of the form $[t'_i,t'_{i+1}]$ with
$$
t'_i=t_1+\frac{2^{i+1}\ell_n^{2/3}}{f^{1-\epsilon}C}.
$$
This subdivision is necessary in order to respect Conditions \ref{cond_3} when we apply our concentration theorems. We stop at $t'_{\bar{i}}=t_2$ by truncating the last interval. By Corollary \ref{S_h} and a straightforward calculation, for $i < \bar{i}-1$:
\begin{equation}
\begin{aligned}
\label{L^hh}
\min_{t'_i \leq j \leq t'_{i+1}} \mathbb{E}(\tilde{L}_j^{h})
&\geq \frac{2^i\epsilon'f^{\epsilon}\ell_n^{1/3}}{2C},
\end{aligned}
\end{equation}
and:
\begin{equation}
\begin{aligned}
\label{L^hhh}
\min_{t'_{\bar{i}-1} \leq j \leq t'_{\bar{i}}} \mathbb{E}(\tilde{L}_j^{h})
&\geq \frac{\epsilon'f^{2}\ell_n^{1/3}}{2C}.
\end{aligned}
\end{equation}
A simple computation shows that we can apply Theorem \ref{inf_s_0} to $\tilde{L}^h$ between $1$ and $t_{i+1}$ in order to obtain the following inequalities for $i < \bar{i}-1$ and for $\bar{i}$:  
\begin{equation}
\begin{aligned}
\label{eq35}
&\mathbb{P}\left( \inf_{t'_i \leq j \leq t'_{i+1}} (\tilde{L}_j^{h}-\mathbb{E}(\tilde{L}_j^{h}))\leq  - \frac{2^{i-1}\epsilon'f^{\epsilon}\ell_n^{1/3}}{2C}\right) \leq A\exp\left(\frac{-2^{i-1}f^{\epsilon}}{A}\right),\\
&\mathbb{P}\left( \inf_{t'_{\bar{i}-1} \leq j \leq t'_{\bar{i}}} (\tilde{L}_j^{h}-\mathbb{E}(\tilde{L}_j^{h}))\leq  - \frac{\epsilon'f^{2}\ell_n^{1/3}}{4C}\right) \leq A\exp\left(\frac{-f}{A}\right).
\end{aligned}
\end{equation}
By the union bound using Equations  \eqref{L^hh}, \eqref{L^hhh} and \eqref{eq35}, we get:
\begin{equation}
\begin{aligned}
\label{eq61}
&\mathbb{P}\left( \inf_{t_1 \leq j \leq t_2} L_j \leq  0\right) \\
&\leq \sum_{i=0}^{\bar{i}-1}\mathbb{P}\left( \inf_{t'_i \leq j \leq t'_{i+1}} \tilde{L}_j \leq  \frac{2^{i-1}\epsilon'f^{\epsilon}\ell_n^{1/3}}{2C}\right)+ \sum_{i=0}^{\bar{i}-1}\mathbb{P}\left( \inf_{t'_i \leq j \leq t'_{i+1}} (L'_j-\tilde{L}_j ) \leq  -\frac{2^{i-1}\epsilon'f^{\epsilon}\ell_n^{1/3}}{2C}\right)\\
&+ \mathbb{P}\left( \inf_{t'_{\bar{i}-1} \leq j \leq t'_{\bar{i}}} \tilde{L}_j \leq   \frac{\epsilon'f^{2}\ell_n^{1/3}}{4C}\right)+\mathbb{P}\left( \inf_{t'_{\bar{i}-1} \leq j \leq t'_{\bar{i}}} (L'_j-\tilde{L}_j)\leq  - \frac{\epsilon'f^{2}\ell_n^{1/3}}{4C}\right) \\
&\leq \sum_{i=0}^{\bar{i}-1}\mathbb{P}\left( \inf_{t'_i \leq j \leq t'_{i+1}} \tilde{L}^h_j \leq  \frac{2^{i-1}\epsilon'f^{\epsilon}\ell_n^{1/3}}{2C}\right)+ \sum_{i=0}^{\bar{i}-1}\mathbb{P}\left( \inf_{t'_i \leq j \leq t'_{i+1}} (L'_j-\tilde{L}_j ) \leq  -\frac{2^{i-1}\epsilon'f^{\epsilon}\ell_n^{1/3}}{2C}\right)\\
&+ \mathbb{P}\left( \inf_{t'_{\bar{i}-1} \leq j \leq t'_{\bar{i}}} \tilde{L}^h_j \leq   \frac{\epsilon'f^{2}\ell_n^{1/3}}{4C}\right)+\mathbb{P}\left( \inf_{t'_{\bar{i}-1} \leq j \leq t'_{\bar{i}}} (L'_j-\tilde{L}_j)\leq  - \frac{\epsilon'f^{2}\ell_n^{1/3}}{4C}\right)\\
&+ \mathbb{P}\left(\sup_{1 \leq j \leq t_2} L_j \geq h\right) \\
&\leq \sum_{i=0}^{\bar{i}-1}\mathbb{P}\left( \inf_{t'_i \leq j \leq t'_{i+1}} (\tilde{L}^h_j-\mathbb{E}[\tilde{L}^h_j]) \leq  -\frac{2^{i-1}\epsilon'f^{\epsilon}\ell_n^{1/3}}{2C}\right)+ \sum_{i=0}^{\bar{i}-1}\mathbb{P}\left( \inf_{t'_i \leq j \leq t'_{i+1}} (L'_j-\tilde{L}_j ) \leq  -\frac{2^{i-1}\epsilon'f^{\epsilon}\ell_n^{1/3}}{2C}\right)\\
&+ \mathbb{P}\left( \inf_{t'_{\bar{i}-1} \leq j \leq t'_{\bar{i}}} (\tilde{L}^h_j-\mathbb{E}[\tilde{L}^h_j])  \leq  - \frac{\epsilon'f^{2}\ell_n^{1/3}}{4C}\right)+\mathbb{P}\left( \inf_{t'_{\bar{i}-1} \leq j \leq t'_{\bar{i}}} (L'_j-\tilde{L}_j)\leq  - \frac{\epsilon'f^{2}\ell_n^{1/3}}{4C}\right)\\
&+ \mathbb{P}\left(\sup_{1 \leq j \leq t_2} L_j \geq h\right) \\
&\leq \sum_{i=0}^{\infty}A\exp\left(\frac{-2^{i-1}f^{\epsilon}}{A}\right) + A\exp\left(\frac{-f}{A}\right) \\
&\leq A'\exp\left(\frac{-f^{\epsilon}}{A'}\right),
\end{aligned}
\end{equation}
here the constant $A'>0$ is large enough and of course these inequalities only hold for $n$ large enough. \\
We now show that $L$ visits $0$ between times $t_2$ and $t_3$. 
Recall that $(Z(i))_{i \leq n}$ is defined by $Z(i) = L_i-L'_i$. Then if $L_{t_3}' \leq -Z(t_2)$, it means that $L'$ attained a new minimum between $t_2$ and $t_3$ i.e $L$ visited $0$ between $t_2$ and $t_3$. Also, by construction, $Z(i) = -\min\limits_{j \leq i}(L'_j)+1$. Since $L'$ is deterministically  smaller than $L^0$, if $L_{t_3}' \geq -Z(t_2)$ then $L_{t_3}^{0} \geq -Z(t_2)$. Therefore, it is sufficient to bound  $\mathbb{P}(L_{t_3}^{0} \geq -Z(t_2))$. We do so by introducing an intermediate term:
\begin{equation}
\begin{aligned}
\label{couv}
\mathbb{P}(L_{t_3}^{0} \geq -Z_{t_2}) &\leq \mathbb{P}\left(L_{t_3}^{0} \geq -\frac{\epsilon'f^{2}\ell_n^{1/3}}{C}\right) + \mathbb{P}\left(Z(t_2) \geq \frac{\epsilon'f^{2}\ell_n^{1/3}}{C}\right) \\
&\leq \mathbb{P}\left(L_{t_3}^{0} \geq -\frac{\epsilon'f^{2}\ell_n^{1/3}}{C}\right) + \mathbb{P}\left(Z(t_2) \geq \frac{\epsilon'f^{\epsilon}\ell_n^{1/3}}{C}\right),
\end{aligned}
\end{equation}
we bound each one of the two terms of the right-hand side of \eqref{couv} separately. First:
$$
\mathbb{P}\left(Z(t_2) \geq \frac{\epsilon'f^{\epsilon}\ell_n^{1/3}}{C}\right) \leq \mathbb{P}\left(Z(t_1) \geq \frac{\epsilon'f^{\epsilon}\ell_n^{1/3}}{C}\right)+ \mathbb{P}\left(Z(t_2) > Z_{t_1}\right).
$$
Since $Z(t_2) > Z(t_1)$ occurs precisely if $L$ visits $0$ between $t_1$ and $t_2$ we already know by Equation \eqref{eq61} that:
\begin{equation}
\begin{aligned}
\label{toutou}
\mathbb{P}(Z(t_2) > Z(t_1)) \leq  A'\exp\left(\frac{-f^{\epsilon}}{A'}\right).
\end{aligned}
\end{equation}
By definition $Z(t_1) \geq r$ precisely if $L'_i < 1-r$  for some $i \leq t_1$. By Corollary \ref{S_h}, for any $i \leq t_1$:
$$
\mathbb{E}(L_i^{h}) \geq 0.
$$
Using this inequality alongside Inequality \eqref{eq611} and Theorems  \ref{inf_s_0} and \ref{inf_s_1} yields:
\begin{equation}
\begin{aligned}
\label{toutou_2}
&\mathbb{P}\left(Z(t_1) \geq \frac{\epsilon'f^{\epsilon}\ell_n^{1/3}}{C}\right) \\
&= \mathbb{P}\left(\inf_{i \leq t_1}(  L_i') \leq 1 - \frac{\epsilon'f^{\epsilon}\ell_n^{1/3}}{C}\right) \\
&\leq \mathbb{P}\left( \inf_{1\leq j \leq t_1} (\tilde{L}^h_j) \leq  - \frac{\epsilon'f^{\epsilon}\ell_n^{1/3}}{4C}\right)+\mathbb{P}\left( \inf_{1 \leq j \leq t_1} (L'_j-\tilde{L}_j)\leq  - \frac{\epsilon'f^{\epsilon}\ell_n^{1/3}}{4C}\right)
+ \mathbb{P}\left(\sup_{1 \leq j \leq t_1} L_j \geq h\right) \\
&\leq A\exp\left(\frac{-f^{\epsilon}}{A}\right).
\end{aligned}
\end{equation}
By the union bound between Equations \eqref{toutou} and \eqref{toutou_2} we get:
\begin{equation}
\label{eq46}
\mathbb{P}\left(Z(t_2) \geq \frac{\epsilon'f^{\epsilon}\ell_n^{1/3}}{C}\right)\leq A\exp\left(\frac{-f^{\epsilon}}{A}\right).
\end{equation}
Furthermore, by Corollary \ref{S_h}:
$$
\mathbb{E}[L^0_{t_3}] \leq -\frac{2\epsilon'f^{2}\ell_n^{1/3}}{C}.
$$
By this fact and Theorem \ref{inf_s_0_3} we obtain:  
\begin{equation}
\begin{aligned}
\label{eq45}
\mathbb{P}\left(L^0(t_3) \geq -\frac{\epsilon'f^{2}\ell_n^{1/3}}{C}\right) &\leq 
\mathbb{P}\left(L^0(t_3) -\mathbb{E}[L^0_{t_3}]\geq \frac{\epsilon'f^{2}\ell_n^{1/3}}{C}\right) \\
&\leq A'\exp\left(\frac{-f}{A'}\right).
\end{aligned}
\end{equation}
Injecting Inequalities \eqref{eq46} and \eqref{eq45} in Inequality \eqref{couv} yields:
\begin{equation*}
\mathbb{P}(L_{t_3}^{0} \geq -Z(t_2)) \leq A\exp\left(\frac{-f^{\epsilon}}{A}\right),
\end{equation*}
and this finishes the proof.
\end{proof}

The following  theorem gives a lower and upper bound on the total weight of $H_f^*$.
\begin{Theorem}
\label{Big_weight}
Suppose that Conditions \ref{cond_nodes} hold.  Let $1 > \epsilon' > 0$. For $f$ large enough and for any $1 \geq \epsilon > 0$,
let $t_1 = \frac{\ell_n^{2/3}}{f^{1-\epsilon}C}$, $t_2 = \frac{2(1-\epsilon')f\ell_n^{2/3}}{C}$ and $t_3 = \frac{2(1+\epsilon')f\ell_n^{2/3}}{C}$. \\
There exists a constant $A > 0$ such that the probability that the total weight of $H_f^*$ is less than $t_2-t_1-\epsilon'(t_2-t_1)$ or more than $t_3+\epsilon' t_3$ is at most 
$$ 
A\exp\left(\frac{-f^{\epsilon}}{A}\right).
$$
\end{Theorem}

\begin{proof}
Let $E$ be the event that $L_i$ visists $0$ for an  $t_1 \leq i \leq t_2$ or $L_i$ does not visit $0$ for any $t_2 \leq i \leq t_3$. For $n$ large enough, Theorem \ref{lower_bound_1} states that there exists $A > 0$ such that:
$$\mathbb{P}(E) \leq A\exp\left(\frac{-f^{\epsilon}}{A}\right).$$ \\
If $E$ does not hold, the total weight of $H_f^*$ is larger than: 
$$T = \sum_{i=t_1}^{t_2}w_{v(i)}.$$
By Lemma \ref{esperance}
$$\mathbb{E}[T] = (t_2-t_1) + o(t_2-t_1).$$
By  Theorem \ref{concentration_final}, there exist positive constants $A'',A'''$ such that: 
\begin{equation*}
\begin{aligned}
\label{brr}
\mathbb{P}\left[T \leq \mathbb{E}(T) - \epsilon'(t_2-t_1) \right] &\leq A''\exp\left(\frac{-\epsilon'f\ell_n^{1/3}}{A''}\right),
\end{aligned}
\end{equation*}
hence by the union bound the total weight of  $H_f^*$ is less than $t_2-t_1-\epsilon'(t_2-t_1)$ with probability at most:
\begin{equation*}
\begin{aligned}
\mathbb{P}\left[T \leq  (t_2-t_1) -\epsilon'(t_2-t_1)\right]  + \mathbb{P}(E) 
\leq A'\exp\left(\frac{-f^{\epsilon}}{A'}\right),
\end{aligned}
\end{equation*}
where $A > 0$ is a large constant.  
Moreover when $E$  does not hold the total weight of $H_f^*$ is less than:
$$T' = \sum_{i=0}^{t_3}w_{v(i)}.$$
By the same arguments  $H_f^*$ is  more than $t_3+\epsilon' t_3$ with probability at most:
\begin{equation*}
\begin{aligned}
\mathbb{P}\left[T' \geq t_3 + \epsilon' t_3 \right] + \mathbb{P}(E)\leq  A'\exp\left(\frac{-f^{\epsilon}}{A'}\right).
\end{aligned}
\end{equation*}
\end{proof}

\subsection{The excess of the giant component.}
The previous theorems give us information about the size of $H^*_f$. We now turn to its surplus. Recall that the surplus (or excess) is the number of edges we need to remove from a connected graph in order to make it a tree. The excess of a general graph is the sum of excesses of its connected components. \par

\begin{Theorem}
\label{exc_big}
Suppose that Conditions \ref{cond_nodes} hold. Let $Exc$ be the excess of $H^*_f$, there exists a positive constant $A >0$ such that:
\begin{equation*}
\mathbb{P}(Exc \geq Af^3) \leq A\exp\left(\frac{-f}{A}\right).
\end{equation*}
\end{Theorem}

\begin{proof}
By construction, if a component is discovered between times $t_1$ and $t_2$ of the process,  then its excess is precisely
$$
\sum\limits_{i=t_1}^{t_2}\sum\limits_{j=i+1}^{L_i+i-1}Y(v(i),v(j)).
$$
Let $m = \frac{3f\ell_n^{2/3}}{2C}$. By  Theorem \ref{sup_l}:  
\begin{equation}
\begin{aligned}
\label{eq43}
\mathbb{P}\left(\sup_{ 1 \leq i \leq m}(L_i) \geq \frac{10f^2\ell_n^{1/3}}{C}\right)
&\leq A''\exp\left(\frac{-f}{A''}\right).
\end{aligned}
\end{equation}
By  Theorem \ref{lower_bound_1},  there exists  a constant $A' > 0$ such that the probability that $H^*_f$ has size more than $m$ is at most:  
\begin{equation}
\begin{aligned}
\label{eq44}
A'\exp\left(\frac{-f}{A'}\right).
\end{aligned}
\end{equation}
Let $E$ be the event that $H^*_f$ has size less than $m$ and  $L_i \leq \frac{10f^2\ell_n^{1/3}}{C} \text{ for all } 1 \leq i \leq m$. By the union bound between Inequalities \eqref{eq43} and \eqref{eq44} we get:
\begin{equation}
\label{fin_2_bzzz}
\mathbb{P}(\bar{E}) \leq A''\exp\left(\frac{-f}{A''}\right),
\end{equation}
for some large constants $A'' >0$. Let $R = \frac{10f^2\ell_n^{1/3}}{C}$ and:
$$
U(R,i) = \sum_{j=i+1}^{L_{i-1}+i-1}Y(v(i),v(j))+\sum_{j=L_{i-1}+i}^{R+i}Y'(v(i),v(j)),
$$
with $Y'(i,j)$ being a Bernoulli random variable independent of everything else and having the same distribution as $Y(i,j)$ for $i \neq j$.
 We have thus by the union bound for any $l \geq 0$:
\begin{equation}
\begin{aligned}
\label{baz_1}
\mathbb{P}\left(\textnormal{Exc} \geq l  \right) &\leq \mathbb{P}\left(\left(\sum\limits_{i=1}^{|H^*_f|}\sum_{j=i+1}^{L_{i-1}+i-1}Y(v(i),v(j)) \geq l\right), E\right) + \mathbb{P}\left(\bar{E}\right) \\
&\leq \mathbb{P}\left(\sum\limits_{i=1}^{m}U(R,i)\geq l\right) + \mathbb{P}\left(\bar{E}\right)
\end{aligned}
\end{equation}
Conditionally on $\mathcal{F}_n$ the $U(R,i)$'s are sums of independent Bernoulli random variables. This is true because the first sum in the definition of $U(R,i)$ consists on independent Bernoulli random variables as stated in Fact \ref{fact2}. Moreover, for any $(i,j)_{1 \leq i, \, 1+i \leq j \leq L_i+i-1}$ by Equation \ref{expx}:
$$
\mathbb{E}[Y(v(i),v(j)) | \mathcal{F}_n ] \leq w_{v(i)}w_{v(j)}p_f,
$$
and 
$$
\mathbb{E}[Y(v(i),v(j))^2 | \mathcal{F}_n] \leq w_{v(i)}w_{v(j)}p_f.
$$
The first inequality yields:
$$
\mathbb{E}\left[\sum\limits_{i=1}^{m}U(R,i) \middle| \mathcal{F}_n \right] \leq  \sum\limits_{i=1}^{m} \sum\limits_{j=i+1}^{(R+i)}w_{v(i)}w_{v(j)}p_f.
$$
Hence, by Bernstein's inequality: 
\begin{equation}
\label{lahlah}
\mathbb{P}\left(\sum\limits_{i=1}^{m}U(R,i)\geq l+\sum\limits_{i=1}^{m} \sum\limits_{j=i+1}^{(R+i)}w_{v(i)}w_{v(j)}p_f \middle| \mathcal{F}_n  \right) \leq \exp\left(\frac{-l^2}{2l+2\sum\limits_{i=1}^{m} \sum\limits_{j=i+1}^{(R+i)}w_{v(i)}w_{v(j)}p_f}\right).
\end{equation}
Denote by $J_1,J_2, ..,J_n$ i.i.d. copies of $v(1)$. From Lemma \ref{mean_pow_weights_2}, there exists a constant $A' >0$ such that:
\begin{equation}
\begin{aligned}
\label{suu_1}
\mathbb{E}\left[p_f\sum\limits_{k=0}^{\lceil \frac{m}{R}\rceil} 2R \left( \sum\limits_{j=kR+1}^{(k+2)R}{w_{J_i}}^2\right)\right] \leq A'mRp_f.
\end{aligned}
\end{equation}
Moreover, by Cauchy-Schwarz's inequality:
\begin{equation*}
\begin{aligned}
 \mathbb{P}\left(\sum\limits_{i=1}^{m} \sum\limits_{j=i+1}^{(R+i)}w_{v(i)}w_{v(j)}p_f  \geq (A'+1)mRp_f   \right) 
&\leq 
\mathbb{P} \left( p_f\sum\limits_{k=0}^{\lceil \frac{m}{R}\rceil} \left(\sum\limits_{i=kR+1}^{(k+2)R}w_{v(i)}\right)^2 \geq (A'+1)mRp_f  \right) \\
&\leq \mathbb{P} \left(p_f\sum\limits_{k=0}^{\lceil \frac{m}{R}\rceil} 2R \left( \sum\limits_{i=kR+1}^{(k+2)R}w_{v(i)}^2 \right)\geq (A'+1)mRp_f  \right).
\end{aligned}
\end{equation*}
Hence, by Theorem \ref{replacement} applied on the $(w_{v(i)}^2)_{1 \leq i \leq m}$'s and Inequality \eqref{suu_1} we have the following Chernoff bound which yields  a Bernstein's inequality (\cite{B24}, \cite{BML13}):
\begin{equation}
\begin{aligned}
\label{suu_3}
 \mathbb{P}\left(\sum\limits_{i=1}^{m} \sum\limits_{j=i+1}^{(R+i)}w_{v(i)}w_{v(j)}p_f  \geq (A'+1)mRp_f   \right) 
&\leq \mathbb{P} \left(\sum\limits_{k=0}^{\lceil \frac{m}{R}\rceil} 2 \left( \sum\limits_{i=kR+1}^{(k+2)R}w_{v(i)}^2 \right)\geq (A'+1)m  \right) \\
&\leq \mathbb{E} \left[\exp\left(\sum\limits_{k=0}^{\lceil \frac{m}{R}\rceil} 2 \left( \sum\limits_{i=kR+1}^{(k+2)R}w_{v(i)}^2 \right)\right)\exp\left(- (A'+1)m  \right)\right] \\
&\leq  \mathbb{E} \left[\exp\left(\sum\limits_{k=0}^{\lceil \frac{m}{R}\rceil} 2 \left( \sum\limits_{i=kR+1}^{(k+2)R}w_{J_i}^2 \right)\right)\exp\left(- (A'+1)m  \right)\right] \\
&\leq A\exp\left(\frac{-m^2}{An^{2/3}m}\right) \\
&\leq A\exp\left(\frac{-f}{A}\right).
\end{aligned}
\end{equation}
Here the penultimate inequality uses the fact that $\mathbb{E}[w_{v(1)}^4] \leq n^{2/3}\mathbb{E}[w_{v(1)}^2]$ and Lemma \ref{mean_pow_weights_2}. 
We have that $mRp_f =  Af^3$ for some $A > 0$. By Equations \ref{fin_2_bzzz}, \ref{baz_1}, \ref{lahlah} and \ref{suu_3}, the union bound yields:
\begin{equation*}
\begin{aligned}
\mathbb{P}\left(Exc \geq (A'+2)mRp_f \right) &\leq \mathbb{P}\left(\sum\limits_{i=1}^{m}U(R,i)\geq (A'+2)mRp_f\right) + \mathbb{P}[\bar{E}] \\
&\leq \exp\left(\frac{-(mRp_f)^2}{2mRp_f+2 (A'+1)mRp_f}\right)  + A''\exp\left(\frac{-f}{A''}\right)\\
&+\mathbb{P}\left(\sum\limits_{i=1}^{m} \sum\limits_{j=i+1}^{(R+i)}w_{v(i)}w_{v(j)}p_f  \geq (A'+1)mRp_f   \right)\\
&\leq A'''\exp\left(\frac{-(mRp_f)^2}{A'''(mRp_f)}\right)+A'''\exp\left(\frac{-f}{A'''}\right) \\
&\leq A\exp\left(\frac{-f}{A}\right),
\end{aligned}
\end{equation*}
where $A>0$ is a large enough constant.
\end{proof}

\subsection{The excess of the components discovered before the largest connected component.}
\begin{Theorem}
\label{small_exc_1}
Suppose that Conditions \ref{cond_nodes} hold. Let $\textnormal{Exc}_0$ be the total excess of the components discovered before the largest component.  There exists $A > 0$ such that for any $0 < \epsilon \leq 1$:
$$
\mathbb{P}\left(\textnormal{Exc}_0 \geq Af^{\epsilon}\right) \leq A\exp\left(\frac{-f^{\epsilon/2}}{A}\right).
$$
\end{Theorem}
\begin{proof}
We know from Theorem \ref{lower_bound_1} that for any $0 < \bar{\epsilon} \leq 1$ the exploration of the largest component starts before time $m=\frac{\ell_n^{2/3}}{f^{1-\bar{\epsilon}}C}$ with probability at least:
\begin{equation}
\label{sousou'}
   1-A\exp\left(\frac{-f^{\bar{\epsilon}}}{A}\right).
\end{equation}
In that case the total excess of components discovered before the largest one is at most:
$$
\sum\limits_{i=0}^{m}\sum\limits_{j=i+1}^{L_{i-1}+i-1}Y(v(i),v(j)).
$$
By Corollary \ref{S_h} and Conditions \ref{cond_nodes}, for any $0 \leq i \leq j \leq m$:
$$
\mathbb{E}(L^{0}(j)-L^{0}(i)) \leq \frac{f^{\bar{\epsilon}}\ell_n^{1/3}}{C}.
$$
By this fact and Theorem \ref{inf_s_0_3} applied on $(m,0,0,y)$, there exists an $A>0$ such that:
\begin{equation*}
\mathbb{P}\left(\sup_{ 0 \leq i \leq j \leq m} (L^{0}(j)-L^{0}(i)) \geq \frac{2f^{\bar{\epsilon}}\ell_n^{1/3}}{C} \right) \leq A\exp\left(\frac{-f^{\bar{\epsilon}}}{A}\right),
\end{equation*}
Remark that, deterministically, $$\sup\limits_{0\leq k\leq m} L(k) \leq \sup\limits_{0\leq i\leq j \leq m} (L'(j)-L'(i)) \leq \sup\limits_{0\leq i\leq j \leq m} (L^{0}(j)-L^{0}(i)),$$ hence:
\begin{equation}
\begin{aligned}
\label{sousou}
\mathbb{P}\left(\sup_{0 \leq i\leq m}L_i \geq \frac{2f^{\bar{\epsilon}}\ell_n^{1/3}}{C}\right) &\leq \mathbb{P}\left(\sup_{ 0 \leq i \leq j \leq m} (L^{0}(j)-L^{0}(i)) \geq \frac{2f^{\bar{\epsilon}}\ell_n^{1/3}}{C}\right), \\
&\leq  A\exp\left(\frac{-f^{\bar{\epsilon}}}{A}\right).
\end{aligned}
\end{equation}
Let $R = \frac{2f^{\bar{\epsilon}}\ell_n^{1/3}}{C}$. Let $E$ be the event $\{\max_{0 \leq i\leq m}L_i \leq R\}$ and the exploration of the largest component starts before time $m$.
Recall the definition of $U(R,i)$ from Theorem \ref{exc_big}. We have for any $l \geq 0$ by the union bound:
\begin{equation}
\begin{aligned}
\label{bez_1}
\mathbb{P}\left(\textnormal{Exc}_0 \geq l \right) \leq \mathbb{P}\left(\sum\limits_{i=0}^{m}U(R,i) \geq l \right) + \mathbb{P}[\bar{E}].
\end{aligned}
\end{equation}
We use the same idea as in Theorem \ref{exc_big}. 
By Bernstein's inequality (\cite{B24}):
\begin{equation}
\label{lahlahlaah}
\mathbb{P}\left(\sum\limits_{i=1}^{m}U(R,i)\geq l + \sum\limits_{i=1}^{m} \sum\limits_{j=i+1}^{(R+i)}w_{v(i)}w_{v(j)}p_f\middle| \mathcal{F}_n  \right) \leq \exp\left(\frac{-l^2}{2l+2\sum\limits_{i=1}^{m} \sum\limits_{j=i+1}^{(R+i)}w_{v(i)}w_{v(j)}p_f}\right).
\end{equation}

Denote by $J_1,J_2, ..,J_n$ i.i.d. copies of $v(1)$. Similarly to Equation \eqref{suu_1}, there exists a constant $A' >0$ such that:
\begin{equation}
\begin{aligned}
\label{suuu_2}
\mathbb{E}\left[p_f\sum\limits_{k=0}^{\lceil \frac{m}{R}\rceil} 2R \left( \sum\limits_{j=kR+1}^{(k+2)R}{w_{J_i}}^2\right)\right] \leq A'mRp_f.
\end{aligned}
\end{equation}
And similarly to Equation \eqref{suu_3} we have for any $\lambda \geq 0$:
\begin{equation}
\begin{aligned}
\label{suuu_3}
 &\mathbb{P}\left(\sum\limits_{i=1}^{m} \sum\limits_{j=i+1}^{(R+i)}w_{v(i)}w_{v(j)}p_f  \geq (A'+1)mRf^{\lambda \bar\epsilon}p_f   \right) \\
\leq 
 &\mathbb{E} \left[\exp\left(\sum\limits_{k=0}^{\lceil \frac{m}{R}\rceil} 2 \left( \sum\limits_{i=kR+1}^{(k+2)R}w_{J_i}^2 \right)\right)\exp\left(- (A'+1)mf^{\lambda \bar\epsilon}  \right)\right] \\
 \leq  &A\exp\left(\frac{-m^2f^{2\lambda\bar{\epsilon}}}{A(m\ell_n^{2/3}f^{\lambda\bar{\epsilon}}+m\ell_n^{2/3})}\right) \\
\leq  &A''\exp\left(\frac{-f^{(\lambda+1)\bar{\epsilon}-1}}{A''}\right).
\end{aligned}
\end{equation}
And also, Equations \eqref{sousou'} and \eqref{sousou} yield:
\begin{equation}
\label{eq79}
    \mathbb{P}(\bar{E}) \leq A'\exp\left(\frac{-f^{\bar{\epsilon}}}{A'}\right).
\end{equation}
By Equations \ref{bez_1}, \ref{lahlahlaah}, \ref{suuu_3}, and \ref{eq79} the union bound yields for $A'' >0$ large enough:
\begin{equation*}
\begin{aligned}
\mathbb{P}\left(Exc_0 \geq (A'+2)mRf^{\lambda\bar{\epsilon}}p_f \right) &\leq \mathbb{P}\left(\sum\limits_{i=1}^{m}U(R,i)\geq (A'+2)mRf^{\lambda \bar\epsilon}p_f \right) + \mathbb{P}[\bar{E}] \\
&\leq \exp\left(\frac{-(mRf^{\lambda\bar{\epsilon}}p_f)^2}{ A'''(mRf^{\lambda\bar{\epsilon}}p_f}\right)  + A''\exp\left(\frac{-f^{\bar{\epsilon}}}{A''}\right)\\
&+\mathbb{P}\left(\sum\limits_{i=1}^{m} \sum\limits_{j=i+1}^{(R+i)}w_{v(i)}w_{v(j)}p_f  \geq (A'+1)mRf^{\lambda\bar{\epsilon}}p_f  \right)\\
&\leq A'''\exp\left(\frac{-(mRf^{\lambda\bar{\epsilon}}p_f)^2}{A'''(mRf^{\lambda\bar{\epsilon}}p_f)}\right)+A''\exp\left(\frac{-f^{\bar{\epsilon}}}{A''}\right) \\ &+A'''\exp\left(\frac{-f^{(\lambda+1)\bar{\epsilon}-1}}{A'''}\right),
\end{aligned}
\end{equation*}
where $A>0$ is a large enough constant. Moreover, we have for $n$ large enough:
$$
mRf^{\lambda\bar{\epsilon}}p_f \geq  \frac{1}{C^2}f^{(2+\lambda)\bar{\epsilon}-1}
$$
for some large constant $A > 0$.
Hence, if we take:
$$
\lambda = \frac{2}{\epsilon},
$$
and:
$$
\epsilon = (2+\lambda)\bar{\epsilon}-1.
$$
We obtain $\bar{\epsilon} = \epsilon/2$ and:
$$
(1+\lambda)\bar{\epsilon}-1 = \frac{\epsilon}{2}.
$$
This proves the inequality of the theorem.
\end{proof}
\section{The structure of the tail's components}

\subsection{Preliminaries}
We call tail of the exploration process the part of it that starts after $H^*_f$ is fully explored and ends at $n$.
In order to get bounds on the size, weight and excess of the tail, we will use two main ideas. Firstly we use an appropriate division of the interval that start after the exploration of $H^*_f$, and ends in $n$. Secondly we make use of the fact that the further we go in the exploration the smaller the weights we discover. These two ideas are formalized below. The rest of the proofs uses similar techniques to the ones presented in Section $4$, but with the added complexity of incorporating these two ideas. \par
For $i \geq 1$, write: 
$$\bar{k}_i = i^2f((i+1)^2-i^2).$$
For $\bar{k}_i> k \geq 0$, and as long as $t^i_k < \ell_n^{11/12}$, write:
$$t_k^i = t + \frac{(i^2-1)f\ell_n^{2/3}}{C} + \frac{k\ell_n^{2/3}}{Ci^2f},$$
with $t = \frac{2(1-\epsilon')f\ell_n^{2/3}}{C}$ and where $1/2 >\epsilon' >0$ is fixed from here on.
Moreover, let $(\tilde{i},\tilde{k})$ be the first time when $t_{\tilde{k}}^{\tilde{i}} \geq \ell_n^{11/12}$. For any $k > \tilde{k}$ let:
$$
t_k^{\tilde{i}} = t +  \frac{(\tilde{i}^2-1)f\ell_n^{2/3}}{C} + \frac{k\ell_n^{2/3}}{C\tilde{i}^2f}.
$$
$(\tilde{i},\tilde{k})$ depends implicitly on $\epsilon'$. Moreover, by construction $\tilde{i}^2f = o(n^{1/3})$.
We are only interested in $t_k^i \leq n$, and for simplicity, since there is no real difficulty in dealing with the boundaries, we assume everything is well truncated. \par
This construction gives a division of the interval between $t$ and $n$ in the following way: Take intervals of the form $\left[t^{i}_0,t^{i+1}_0\right)$. Such intervals get larger and larger. Divide each one of them into small intervals of the form $\left[t^{i}_k,t^{i}_{k+1}\right)$ that get smaller with $i$.  The main idea here is that the large intervals, those where $i$ changes, represent phases of the exploration where we will find connected components that are of size at most the size of small intervals $\left[t^{i}_k,t^{i}_{k+1}\right)$. Moreover Conditions \ref{cond_3} will be verified inside the small intervals for good enough deviation values, which will allow us to use all our concentration theorems. 
We start by showing that the maximum weight gets smaller the further we explore the tail.
\begin{Lemma}
\label{weight_tail}
Suppose that Conditions \ref{cond_nodes} hold. There exists a constant $A >0$ such that:\par
For any $1 \leq i \leq \tilde{i}$, the probability of discovering a weight larger than $\frac{\ell_n^{1/3}}{i\sqrt{f}}$ in the BFW after time $t^i_0$ is less than:
$$
A\exp\left(\frac{-i\sqrt{f}}{A}\right).
$$
\end{Lemma}

\begin{proof}
Recall that $(T_i)_{i \leq n}$ is a sequence of independent exponential variables with rates $(w_i/\ell_n)_{i\leq n}$. And that for any $x>0$: 
$$
N(x) = \sum_{k=1}^n\mathbbm{1}(T_k \leq x),
$$
Moreover, recall that by the properties of exponential random variables, the order statistic indices $(\tilde{v}(1),\tilde{v}(2),...\tilde{v}(n))$ of the $(T_k)_{k \leq n}$ have the same distribution as $(v(1),v(2),...v(n))$. \par
Let $x = t^i_0/2$, then by Lemma \ref{bernstinou}, Conditions \ref{cond_nodes} and obvious bounds:
\begin{equation}
\label{Nx}
\mathbb{P}(N(x) \geq t^i_0) \leq A\exp\left(\frac{-t^i_0}{A}\right).
\end{equation}
This equation shows that at time $x$, the weights with indices $(\tilde{v}(t^i_0),\tilde{v}(t^i_0+1),...\tilde{v}(n))$ will not be picked yet with high probability. 
Denote the event $\{N(x) \geq t^i_0\}$ by $E$. For any $k$ such that $w_k \geq \frac{\ell_n^{1/3}}{i\sqrt{f}}$, we have:
\begin{equation*}
\begin{aligned}
\mathbb{P}(T_k \geq x  , \, \bar{E})
&\leq \mathbb{P}(T_k \geq x) \\
&\leq A\exp\left(\frac{-i\sqrt{f}}{A}\right),
\end{aligned}
\end{equation*}
this equation shows that a large weight has a large probability of being picked before time $x$.\par
Recall that by Conditions \ref{cond_nodes}:
$$
\sum_{k=1}^n w_k^3 = (\mathbb{E}[W^3]+o(1))n.
$$
Hence, the total number of weights larger than $\frac{\ell_n^{1/3}}{i\sqrt{f}}$ is less than $A'i^3f^{3/2}$, where $A' >0$ is a large enough constant. \par
This yields:
\begin{equation}
\begin{aligned}
\mathbb{P}\left(\sup_{k \geq t^i_0}(w_{v(k)}) \geq \frac{\ell_n^{1/3}}{i\sqrt{f}}\right) &\leq \mathbb{P}(E) + \sum_{k=1}^n\mathbb{P}(T_k \geq x  , \, \bar{E})\mathbbm{1}\left(w_k \geq \frac{\ell_n^{1/3}}{i\sqrt{f}}\right)\\
&\leq \exp\left(\frac{-t^i_0}{A}\right) + A'Ai^3f^{3/2}\exp\left(\frac{-i\sqrt{f}}{A}\right) \\
&\leq A''\exp\left(\frac{-i\sqrt{f}}{A''}\right),
\end{aligned}    
\end{equation}
whith $A'' >0$ a large constant and $f$ large enough.
\end{proof}
We now use the same notations as in the proof above. For $0 \leq i \leq \tilde{i}$. Let $B$ be the event that no weight larger than $\frac{\ell_n^{1/3}}{i\sqrt{f}}$ is present after time $t^i_0$. Then for any $t^i_0 \leq x$, with the notation of Section $2$ when $B$ holds we have:
$$
X(x)-X(u) = \sum_{k=1}^n w_k\mathbbm{1}(u \leq T_k \leq x)\mathbbm{1}\left(w_k \leq \frac{\ell_n^{1/3}}{i\sqrt{f}}\right),
$$
And:
$$
 N(x)-N(u)=\sum_{k=1}^n \mathbbm{1}(u \leq T_k \leq x)\mathbbm{1}\left(w_k \leq \frac{\ell_n^{1/3}}{i\sqrt{f}}\right).$$ 
Moreover, clearly:
$$
\mathbb{E}\left[\sum_{k=1}^n w_k\mathbbm{1}(u \leq T_k \leq x)\mathbbm{1}\left(w_k \leq \frac{\ell_n^{1/3}}{i\sqrt{f}}\right)\right] \leq\mathbb{E}[X(x)-X(u)],
$$
and 
$$
\mathbb{E}\left[\sum_{k=1}^n \mathbbm{1}(u \leq T_k \leq x)\mathbbm{1}\left(w_k \leq \frac{\ell_n^{1/3}}{i\sqrt{f}}\right)\right] \leq\mathbb{E}[N(x)-N(u)].
$$
By those remarks, when $B$ holds one can redo the proofs of Theorems \ref{concentration} by only taking nodes with weights smaller than $\frac{\ell_n^{1/3}}{i\sqrt{f}}$. Then use the union bound with Lemma \ref{weight_tail} to obtain the following theorem which is in the spirit of Theorem \ref{concentration}.
\begin{Theorem}
\label{concentration_2}
Suppose that Conditions \ref{cond_nodes} hold. There exists a constant $A > 0$ such that the following holds: \\
If  $(m,l,0,y)$ verify Conditions \ref{cond_3}, and there exists $i \leq \tilde{i}$ such that $l \geq t^i_0$, and $m \leq t^{\tilde{i}}_0$ then:
$$
\mathbb{P}\left[ \sup_{l \leq u \leq w \leq m}  \; \sum_{k=u}^w w_{v(k)} - \mathbb{E}\left[\sum_{k=u}^w w_{v(k)}\right]  \geq y\right] \leq  A\exp\left(\frac{-y^2}{A\left(y\frac{\ell_n^{1/3}}{i\sqrt{f}}+m-l\right)}\right)+A\exp\left(\frac{-i\sqrt{f}}{A}\right).
$$
\end{Theorem}

Moreover we have by Bernstein's inequality:
\begin{equation}
\begin{aligned}
\label{facts2}
\mathbb{P}\left(\exists (h,j), \, j \geq t^i_0, \, X^0_j \geq \frac{2\ell_n^{1/3}}{i\sqrt{f}}\right) &\leq \sum_{k=t^i_0}^n \mathbb{P}\left(d_{v(k)} \geq \frac{2\ell_n^{1/3}}{i\sqrt{f}}\right) \\
&\leq \mathbb{P}(\bar{B}) +  \sum_{k=0}^n \mathbbm{1}\left(w_k \leq \frac{\ell_n^{1/3}}{i\sqrt{f}}\right) \mathbb{P}\left(d_k\geq \frac{2\ell_n^{1/3}}{i\sqrt{f}}\right) \\
&\leq \mathbb{P}(\bar{B}) + A'\exp\left(\frac{-\ell_n^{1/3}}{A'i\sqrt{f}}\right) \\
&\leq A\exp\left(\frac{-i\sqrt{f}}{A}\right)
\end{aligned}
\end{equation}
where $A$ is a large constant. This shows that, similarly to what we did in Section $2$, one can assume that $L^0$ and $L$ have increments of size at most $\frac{2\ell_n^{1/3}}{i\sqrt{f}}$ after time $t^i_0$. Using this fact, one can redo the proofs of Theorems \ref{inf_s_0_3} after time $t^i_0$. Then use the union bound with Lemma \ref{weight_tail} and Equation \eqref{facts2} to obtain the following theorem which is in the spirit of Theorem \ref{inf_s_0_3}.
\begin{Theorem}
\label{concentration_L}
Suppose that Conditions \ref{cond_nodes} hold. There exists a constant $A > 0$ such that the following holds: \\
Let  $(m,l,y)$ be such that $(m,l,0,y)$ verifies Conditions \ref{cond_3}, and there exists $i \leq \tilde{i}$ such that $l \geq t^i_0$, and $m \leq t^{\tilde{i}}_0$. We have:
$$
\mathbb{P}\left(\sup_{l \leq u \leq w \leq m} \; L^0_w-L^0_u-\mathbb{E}[L^0_w-L^0_u] \geq y\right) \leq  A\exp\left(\frac{-y^2}{A\left(y\frac{\ell_n^{1/3}}{i\sqrt{f}}+m-l\right)}\right)+A\exp\left(\frac{-i\sqrt{f}}{A}\right).
$$
\end{Theorem}
We will also need the following lemma. It states that the weights get smaller in probability the further we go in the exploration. For $1 \leq k \leq n$ let $w^{i}_k = w_k$ if $w_k \leq \frac{\ell_n^{1/3}}{i\sqrt{f}}$ and $w^{i}_k=\frac{\ell_n^{1/3}}{i\sqrt{f}}$ otherwise. 
\begin{Lemma}
\label{decreasing}
Let $1 \leq u \leq w \leq n$, then for any $x \geq 0$ and $1 \leq i \leq \tilde{i}$:
$$
\mathbb{P}(w^{i}_{v(u)} \geq x) \leq \mathbb{P}(w^{i}_{v(w)} \geq x).
$$
\end{Lemma}
\begin{proof}
Recall that $\mathcal{V}_{u} = (v(1),v(2),...,v(u))$ for any $n \geq i \geq 1$. It is sufficient to prove the lemma for $w = u+1$. In that case we have:
$$
\mathbb{P}(w^{i}_{v(u)} \geq x | \mathcal{V}_{u-1}) = \frac{\sum_{k \not\in \mathcal{V}_{u-1}} w_{k}\mathbbm{1}(w^{i}_k \geq x)}{\sum_{k' \not\in \mathcal{V}_{u-1}} w_{k'}}.
$$
Let: 
$$U = \sum_{k \not\in \mathcal{V}_{i-1}} w_{k}\mathbbm{1}(w^i_k \geq x),$$
and
$$V = \sum_{k \not\in \mathcal{V}_{i-1}} w_{k}.$$
Since $V \geq U$ we have: \begin{equation*}
\begin{aligned}
\mathbb{P}(w^i_{v(u+1)} \geq x | \mathcal{V}_{u-1}) &=\sum_{k\not\in \mathcal{V}_{u-1}} \mathbb{P}(v(u)=k | \mathcal{V}_{u-1})\mathbb{P}(w^i_{v(u+1)} \geq x | \mathcal{V}_{i-1} , \, v(i)=k) \\
&= \sum_{k\not\in \mathcal{V}_{u-1}}\frac{w_k}{V}\left(\frac{U-w_k\mathbbm{1}(w^i_k \geq x)}{V-w_k}\right) \\
&= \sum_{k\not\in \mathcal{V}_{u-1}  , \, w^i_k \geq x}\frac{w_k}{V}\left(\frac{U-w_k}{V-w_k}\right) + \sum_{k\not\in \mathcal{V}_{u-1}  , \, w^i_k < x}\frac{w_k}{V}\left(\frac{U}{V-w_k}\right) \\
&\leq \sum_{k\not\in \mathcal{V}_{u-1}  , \, w^i_k \geq x}\frac{w_k}{V}\left(\frac{U-x}{V-x}\right) + \sum_{k\not\in \mathcal{V}_{u-1}  , \, w^i_k < x}\frac{w_k}{V}\left(\frac{U}{V-x}\right) \\
&=\frac{U}{V}\left(\frac{U-x}{V-x}\right) + \left(\frac{V-U}{V}\right)\left(\frac{U}{V-x}\right) \\
&= \frac{U}{V}\\
&=\mathbb{P}(w_{v(u)} \geq x | \mathcal{V}_{i-1}).
\end{aligned}
\end{equation*}
\end{proof}
With this lemma in hand we can deal with the case when $m > t^{\tilde{i}}_0$.
\begin{Theorem}
\label{concentration_3}
Suppose that Conditions \ref{cond_nodes} hold. There exists a constant $A > 0$ such that the following holds: \\
For $t^{\tilde{i}}_0 < u \leq w $ and for any $y \geq 0$:
$$
\mathbb{P}\left[ \sum_{k=u}^w (w_{v(k)} - 1)  \geq y\right] \leq  A\exp\left(\frac{-y^2}{A\left(y\frac{\ell_n^{1/3}}{\tilde{i}\sqrt{f}}+w-u\right)}\right)+A\exp\left(\frac{-\tilde{i}\sqrt{f}}{A}\right)
$$
\end{Theorem}
\begin{proof}
Let $\mathcal{A}$ be the even that no weight discovered after time $t^{\tilde{i}}_0$ is larger than $\frac{\ell_n^{1/3}}{\tilde{i}\sqrt{f}}$. Let $(J(i))_{i \geq u}$ be i.i.d with the distribution of $v(u)$. Theorem $1$ from \cite{AYS16} still applies for the $(w^{\tilde{i}}_v)_{v\geq 1}$'s and we get similarly to Theorem \ref{replacement}:
\begin{equation}
\begin{aligned}
\label{die1}
\mathbb{P}\left(\sum_{k=u}^w (w_{v(k)} - 1)  \geq y ,\, \mathcal{A}\right) 
&\leq \mathbb{P}\left(\sum_{k=u}^w (w^{\tilde{i}}_{v(k)} - 1)  \geq y\right) \\
&\leq \mathbb{E}\left[\exp\left(\sum_{k=u}^w (w^{\tilde{i}}_{v(k)} - 1)\right)\exp(-y)\right] \\
&\leq \mathbb{E}\left[\exp\left(\sum_{k=u}^w (w^{\tilde{i}}_{J(k)} - 1)\right)\exp(-y)\right].
\end{aligned}
\end{equation}
By Lemma \ref{decreasing} we can apply an ordered coupling argument (Theorem $7.1$ of \cite{D12}) in order to obtain:
\begin{equation}
\begin{aligned}
\label{die2}
&\mathbb{E}\left[\exp\left(\sum_{k=u}^w (w^{\tilde{i}}_{J(k)} - 1)\right)\exp(-y)\right]
\leq \mathbb{E}\left[\exp\left(\sum_{k=u}^w (w^{\tilde{i}}_{J'(k)} - 1)\right)\exp(-y)\right]
\end{aligned}
\end{equation}
where the $J'(k)$'s are i.i.d random variables with the distribution of $v(t^{\tilde{i}}_0+1)$. Moreover by Lemma \ref{esperance} we have for any $k \geq u$:
\begin{equation*}
\mathbb{E}[w^{\tilde{i}}_{J'(k)}] \leq 1
\end{equation*}
and by Lemma \ref{mean_pow_weights_2} :
\begin{equation*}
\mathbb{E}[{w^{\tilde{i}}_{J'(k)}}^2] \leq C(1+o(1))
\end{equation*}
Hence, by Equation \ref{die1} and the Chernoff bound in Equation \ref{die2} we obtain the following Bernstein's inequality:
\begin{equation*}
\begin{aligned}
\mathbb{P}\left(\sum_{k=u}^w (w_{v(k)} - 1)  \geq y ,\, \mathcal{A}\right) 
&\leq 2\exp\left(\frac{-y^2}{A\left(y\frac{\ell_n^{1/3}}{C\tilde{i}\sqrt{f}}+\sum_{k=u}^w \mathbb{E}\left[({w^{\tilde{i}}_{J'(k)}})^2\right]\right)}\right) \\
&\leq A\exp\left(\frac{-y^2}{A\left(y\frac{\ell_n^{1/3}}{\tilde{i}\sqrt{f}}+w-u\right)}\right).
\end{aligned}
\end{equation*}
Moreover, by Theorem \ref{weight_tail}:
$$
\mathbb{P}(\mathcal{A}) \leq A\exp\left(\frac{-\tilde{i}\sqrt{f}}{A}\right).
$$
We finish the proof by union bound between these last two inequalities.
\end{proof}
By the same method, we obtain the following theorem which deals with the $w_{v(i)}^2$'s.
\begin{Theorem}
\label{concentration4}
Suppose that Conditions \ref{cond_nodes} hold. There exists a constant $A > 0$ such that the following holds: \\
For $ i < \tilde{i}$ and $t^{i}_0 \leq u \leq w  \leq t^{\tilde{i}}_0$ and for any $y \geq 0$:
$$
\mathbb{P}\left[ \sum_{k=u}^w (w_{v(k)}^2 - \mathbb{E}[w_{v(u)}^2])  \geq y\right] \leq  A\exp\left(\frac{-y^2}{A\frac{\ell_n^{2/3}}{\tilde{i^2}f}\left(y+w-u\right)}\right)+A\exp\left(\frac{-\tilde{i}\sqrt{f}}{A}\right).
$$
And for $t^{\tilde{i}}_0 < u \leq w $ and for any $y \geq 0$:
$$
\mathbb{P}\left[ \sum_{k=u}^w (w_{v(k)}^2 - \mathbb{E}[w_{v(t^{\tilde{i}}_0)}^2])  \geq y\right] \leq  A\exp\left(\frac{-y^2}{A\frac{\ell_n^{2/3}}{\tilde{i^2}f}\left(y+w-u\right)}\right)+A\exp\left(\frac{-\tilde{i}\sqrt{f}}{A}\right).
$$
\end{Theorem}
\subsection{The size of connected components discovered after the largest connected component}
We can now prove the main theorem on the concentration of the sizes of the components discovered after $H^*_f$. In order to do that we will once again study the event that $L$ visits $0$ in some intervals.
\begin{Theorem}
\label{localize}
Suppose that Conditions \ref{cond_nodes} are verified. Let $i^* \in \mathbb{N}$ be the time at which the exploration of $H^*_f$ ends.There exists a constant $A > 0$ such that the following is true:\\
The probability that there exists an $\tilde{i} \geq i\geq 1$ and $\bar{k}_i > k \geq 0$, such that $L$ does not visit $0$ between times $t^{i}_k-t+i^*$ and $t^{i}_{k+1}-t+i^*$,  or times  $t^{i}_{\bar{k}_i}-t+i^*$ and $t^{i+1}_{0}-t+i^*$ is at most:
\begin{equation*}
A\exp\left(\frac{-\sqrt{f}}{A}\right)+A\exp\left(\frac{-n^{1/8}}{A}\right).
\end{equation*}
\end{Theorem}

\begin{proof}
By Theorem \ref{lower_bound_1}:
\begin{equation}
\label{condition_largest}
\mathbb{P}\left(\frac{2(1+\epsilon')f\ell_n^{2/3}}{C} \geq i^* \geq \frac{2(1-\epsilon')f\ell_n^{2/3}}{C}\right) \geq 1-A\exp\left(\frac{-\sqrt{f}}{A}\right).
\end{equation}
Define $E^i_k$ as the event that $L$ does not visit $0$ between times $t^{i}_k-t+i^*$ and time $t^{i}_{k+1}-t+i^*$, or $t^{i}_{\bar{k}_i}-t+i^*$ and $t^{i+1}_{0}-t+i^*$ if $k=\bar{k}_i$. \par
Deterministically, for any $0 \leq u \leq w \leq n$:
\begin{equation}
\label{eq100}
\mathbb{P}\left(L'_{w}-L'_{u} \geq 0\right) \leq \mathbb{P}\left(L^0_w-L^{0}_u \geq 0\right),
\end{equation}
so it is sufficient to focus on $L^0$.\par
We start by dealing with $(i,k) =(1,0)$, then the rest of the proof consists in repeating the arguments we will give for $(i,k) =(1,0)$ with an induction. \par 
In order to show that $L$ visits $0$ between $i^*$ and $i^*+\frac{\ell_n^{2/3}}{Cf}$, recall that $t = \frac{2(1-\epsilon')f\ell_n^{2/3}}{C}$ and let $E$ be the event $t + \frac{2\epsilon'f\ell_n^{2/3}}{C} \geq i^* \geq t$. Then:
\begin{equation}
\begin{aligned}
\label{eq88}
\mathbb{P}\left(L^{0}_{i^*+\frac{\ell_n^{2/3}}{Cf}}-L^{0}_{i^*}  \geq 0  \right)  &=\mathbb{P}\left(E  , \, \left \{L^{0}_{i^*+\frac{\ell_n^{2/3}}{Cf}}-L^{0}_{i^*}  \geq 0 \right\} \right) + \mathbb{P}(\bar{E}) \\
&\leq  \mathbb{P}\left(\sup_{t \leq u  \leq t+\frac{2\epsilon'f\ell_n^{2/3}}{C} } L^{0}_{u+\frac{\ell_n^{2/3}}{Cf}}-L^{0}_{u}  \geq 0  \right) + \mathbb{P}(\bar{E}).
\end{aligned}
\end{equation}
Divide the interval between $t$ and $t+\frac{2\epsilon'f\ell_n^{2/3}}{C}$ by introducing intermediate terms of the form:
$t'_j= t + \frac{j\ell_n^{2/3}}{fC}$.
Let $\bar{j}$ be the largest integer such that $t'_{\bar{j}} \leq t+\frac{2\epsilon'f\ell_n^{2/3}}{C}$, and suppose everything is well truncated i.e $t'_{\bar{j}} = t+\frac{2\epsilon'f\ell_n^{2/3}}{C}$. Equation \eqref{eq88} then yields:
\begin{equation}
\begin{aligned}
\label{eq888}
&\mathbb{P}\left(\sup_{t \leq u  \leq t+\frac{2\epsilon'f\ell_n^{2/3}}{C} } L^{0}_{u+\frac{\ell_n^{2/3}}{Cf}}-L^{0}_{u}  \geq 0  \right) 
\leq &\sum_{j=1}^{\bar{j}}\mathbb{P}\left(\sup_{t'_{j-1} \leq u  \leq t'_j } L^{0}_{u+\frac{\ell_n^{2/3}}{Cf}}-L^{0}_{u}  \geq 0  \right).
\end{aligned}
\end{equation}
For $\bar{j} \geq j \geq 1$ let:
$$
y_j = \frac{\ell_n^{1/3}(1-2\epsilon')}{2C}+\frac{\ell_n^{1/3}(j-1)}{2f^2C}.
$$
By Corollary \ref{S_h} and straightforward calculations:
\begin{equation*}
\begin{aligned}
\label{eq189}
\sup_{t'_{j-1} \leq k  \leq t'_j }\mathbb{E}\left[L^0_{k+\frac{\ell_n^{2/3}}{fC}}-L^0_k\right] &\leq \frac{3}{4}\mathbb{E}\left[L^0_{t'_{j}}-L^0_{t'_{j-1}}\right]\\
&\leq \frac{-3y_j}{2}.
\end{aligned}
\end{equation*}
Moreover, for any $\bar{j} \geq j \geq 1$, $(t'_j,t'_{j-1},0,y_j)$ verify Conditions \ref{cond_3}. Hence, by Theorem \ref{concentration_L} and the fact that, by definition, $\bar{j} \leq 2f^2$:
\begin{equation}
\begin{aligned}
\label{eq91}
\sum_{j=1}^{\bar{j}}\mathbb{P}\left(\sup_{t'_{j-1} \leq u  \leq t'_j } L^{0}_{u+\frac{\ell_n^{2/3}}{Cf}}-L^{0}_{k}  \geq 0  \right)
&\leq \sum_{j=1}^{\bar{j}}A\exp\left(\frac{-y_j^2}{A\left(y_j\frac{\ell_n^{1/3}}{\sqrt{f}}+f^{-1}\ell_n^{2/3}\right)}\right)+A\exp\left(\frac{-\sqrt{f}}{A}\right). \\
&\leq A'f^2\exp\left(\frac{-\sqrt{f}}{A'}\right) \\
&\leq  A''\exp\left(\frac{-\sqrt{f}}{A''}\right),
\end{aligned}
\end{equation}
we finish the initialization by injecting Inequalities \eqref{condition_largest} and \eqref{eq91} in \eqref{eq88}. \par
We now move to the heredity property.
Write $$\mathcal{E}_{i,k}:= \cup_{(u,v) \leq (i,k)} E^u_v \cup \bar{E}.$$
Suppose that the following inequality holds for $(i,k)$:
\begin{equation}
\label{recurs}
\mathbb{P}\left( \mathcal{E}_{i,k}\right) \leq A\exp\left(\frac{-\sqrt{f}}{A}\right) + A\sum_{j=0}^{i}(i+1)^2\exp\left(\frac{-i\sqrt{f}}{A}\right)+Ak\exp\left(\frac{-i\sqrt{f}}{A}\right),
\end{equation}
where $A >0$ is a large enough constant that does not depend on $(i,k)$. \par   For now suppose that $(i,k) \leq (\tilde{i},\tilde{k})$.
we want to prove a similar inequality for $(i,k+1)$ if $k+1 < \bar{k}_i$, or $(i+1,0)$ if not.
Suppose we are in the case $k+1 <\bar{k}_i$, the other case is similar. Write $t_{0} = t^{i}_{k}$,    $t_{1} = t^{i}_{k+1}+\frac{2\epsilon'f\ell_n^{2/3}}{C}$.
 By definition of $\mathcal{E}_{(i,k)}$:
\begin{equation}
\begin{aligned}
\label{fes}
\mathbb{P}\left(\mathcal{E}_{(i,k+1)}\right) &\leq \mathbb{P}\left(\sup_{ t_0 \leq u \leq t_1}\,  \left(L^{0}_{u+\frac{\ell_n^{2/3}}{Ci^2f}}-L^{0}_u\right)  \geq 0\right) + \mathbb{P}\left(\mathcal{E}_{(i,k)}\right).
\end{aligned}
\end{equation}
By using a similar division to the one used in Inequality \eqref{eq91} we get again:
\begin{equation*}
\begin{aligned}
\label{eq92}
\mathbb{P}\left(\sup_{ t_0 \leq u \leq  t_1}\,  \left(L^{0}_{u+\frac{\ell_n^{2/3}}{Ci^2f}}-L^{0}_u\right)  \geq 0\right) \leq A\exp\left(\frac{-i\sqrt{f}}{A}\right).
\end{aligned}
\end{equation*}
This finishes the induction in the case where $(i,k) \leq (\tilde{i},\tilde{k})$. \par
Now suppose that $(i,k) > (\tilde{i},\tilde{k})$, we cannot directly use Theorem \ref{concentration_L} because $t^i_k$ might be of order $n$. Thus, we use will use the coupling argument of Theorem \ref{concentration_3}. Similarly to Equation \eqref{eq888} we need to bound: 
$$
\sum_{j=1}^{\bar{j}}\mathbb{P}\left(\sup_{t'_{j-1} \leq u  \leq t'_j } L^{0}_{u+\frac{\ell_n^{2/3}}{Cf}}-L^{0}_{u}  \geq 0  \right),
$$
with $t'_j = t^{\tilde{i}}_k+\frac{j\ell_n^{2/3}}{\tilde{i}^2fC}$ and $\bar{j}$ the largest integer such that $t'_{\bar{j}} \leq t^{\tilde{i}}_k+\frac{2\epsilon'f\ell_n^{2/3}}{C}$.
Let
$$
y=\frac{\ell_n^{1/3}(\tilde{i}^2-1)}{8\tilde{i}^2C^2}.
$$
We have for any $u > t^{\tilde{i}}_{\tilde{k}}$:
\begin{equation}
\begin{aligned}
\label{facct}
\tilde{L}^{0}_{u+\frac{\ell_n^{2/3}}{C\tilde{i}^2f}}-\tilde{L}^{0}_u &=  \sum_{r=u+1}^{u+\frac{\ell_n^{2/3}}{C\tilde{i}^2f}}\sum_{r'>u}\left(1-\exp\left(-w_{v(r)}w_{v(r')}p_f\right)\right)-\frac{\ell_n^{2/3}}{C\tilde{i}^2f}\\
&\leq \left(\sum_{r=u}^{u+\frac{\ell_n^{2/3}}{C\tilde{i}^2f}}w_{v(r)}\right)\left(1+\frac{f}{\ell_n^{1/3}}-\sum_{r'=1}^{u}w_{v(r')}p_f\right)-\frac{\ell_n^{2/3}}{C\tilde{i}^2f} \\
&\leq \left(\sum_{r=u}^{u+\frac{\ell_n^{2/3}}{C\tilde{i}^2f}}w_{v(r)}\right)\left(1+\frac{f}{\ell_n^{1/3}}-\sum_{r'=1}^{t^{\tilde{i}}_{\tilde{k}}}w_{v(r')}p_f\right)-\frac{\ell_n^{2/3}}{C\tilde{i}^2f}
\end{aligned}
\end{equation}
for $u \geq 0$ let $\mathcal{A}_1(u)$ be the event: 
$$
\left\{ \sum_{r=u+1}^{u+\frac{\ell_n^{2/3}}{C\tilde{i}^2f}}w_{v(r)}< \frac{\ell_n^{2/3}}{C\tilde{i}^2f}+y/2\right\} \cap \left\{\left(\frac{\sum_{r'=1}^{t^{\tilde{i}}_{\tilde{k}}}w_{v(r')}}{\ell_n}\right)  > \frac{t^{\tilde{i}}_{\tilde{k}}}{2}\right\}.
$$
Then, if $\mathcal{A}_1(u)$ holds, then Equation \eqref{facct} yields:
$$
\tilde{L}^{0}_{u+\frac{\ell_n^{2/3}}{C\tilde{i}^2f}}-\tilde{L}^{0}_u \leq -\frac{y}{2}.
$$
Let also $\mathcal{A}_2(u)$ be the event:
$$
\left\{ \sum_{r=u}^{u+\frac{\ell_n^{2/3}}{C\tilde{i}^2f}}(w_{v(r)}^2) \leq  8y\frac{\ell_n^{1/3}}{\tilde{i}\sqrt{f}}+2\frac{\ell_n^{2/3}}{C\tilde{i}^2f}\right\}.
$$
Then by Bernstein's inequality for martingales (\cite{F75}):
\begin{equation}
\begin{aligned}
\label{eqrrrr1}
&\mathbb{P}\left(\sup_{t'_{j-1} \leq u  \leq t'_j }\left(L^0_{u+\frac{\ell_n^{2/3}}{C\tilde{i}^2f}}-L^0_u-(\tilde{L}^{0}_{u+\frac{\ell_n^{2/3}}{C\tilde{i}^2f}}-\tilde{L}^{0}_u)\right) \geq \frac{y}{2}, \, \cap_{t'_{j-1} \leq u  \leq t'_j }\mathcal{A}_2(u)\right) \\
\leq &\mathbb{P}\left(\sup_{t'_{j-1} \leq u  \leq t'_j }\left(L^0_{u+\frac{\ell_n^{2/3}}{C\tilde{i}^2f}}-L^0_{t'_{j-1}}-(\tilde{L}^{0}_{u+\frac{\ell_n^{2/3}}{C\tilde{i}^2f}}-\tilde{L}^{0}_{t'_{j-1}})\right) \geq \frac{y}{4}, \, \cap_{t'_{j-1} \leq u  \leq t'_j }\mathcal{A}_2(u)\right) \\
&+ \mathbb{P}\left(\sup_{t'_{j-1} \leq u  \leq t'_j }\left(L^0_{u}-L^0_{t'_{j-1}}-(\tilde{L}^{0}_{u+\frac{\ell_n^{2/3}}{C\tilde{i}^2f}}-\tilde{L}^{0}_{t'_{j-1}})\right) \leq \frac{-y}{4}, \, \cap_{t'_{j-1} \leq u  \leq t'_j }\mathcal{A}_2(u)\right) \\
&\leq \exp\left(\frac{-\tilde{i}\sqrt{f}}{A}\right),
\end{aligned}
\end{equation}
where the last inequality uses the fact that $y^2 = \Theta(\ell_n^{2/3})$. \par
By Theorem \ref{concentration_3}, for any $u >t^{\tilde{i}}_{\tilde{k}}$:
\begin{equation}
\begin{aligned}
\label{eqrr1}
\mathbb{P}\left(  \sum_{r=u}^{u+\frac{\ell_n^{2/3}}{C\tilde{i}^2f}}(w_{v(r)}-1)\geq y/2\right) &\leq A\exp\left(\frac{-y^2}{A\left(y\frac{\ell_n^{1/3}}{\tilde{i}\sqrt{f}}+\frac{\ell_n^{2/3}}{C\tilde{i}^2f}\right)}\right)+A\exp\left(\frac{-\tilde{i}\sqrt{f}}{A}\right)\\
&\leq A'\exp\left(\frac{-\tilde{i}\sqrt{f}}{A'}\right),
\end{aligned}
\end{equation}
with $A'>0$ a large constant.  
By Theorem \ref{concentration4} we also get:
\begin{equation}
\label{eqrr3}
\mathbb{P}\left(   \bar{\mathcal{A}}_2(u)\right)\leq A'\exp\left(\frac{-\tilde{i}\sqrt{f}}{A'}\right).
\end{equation}
By Theorem \ref{concentration_final} and straightforward computations  we obtain:
\begin{equation}
\begin{aligned}
\label{eqrr2}
\mathbb{P}\left(\sum_{r'=1}^{t^{\tilde{i}}_{\tilde{k}}}w_{v(r')}\leq \frac{t^{\tilde{i}}_{\tilde{k}}}{2}\right)
&\leq A'\exp\left(\frac{-\tilde{i}\sqrt{f}}{A'}\right).
\end{aligned}
\end{equation}
By the union bound between inequalities \eqref{eqrrrr1}, \eqref{eqrr1},  \eqref{eqrr3} and \eqref{eqrr2} we obtain
\begin{equation}
\begin{aligned}
\label{ineq999}
    \mathbb{P}\left(\sup_{ t'_{j-1} \leq u \leq t'_j}\,  \left(L^{0}_{u+\frac{\ell_n^{2/3}}{Ci^2f}}-L^{0}_u\right)\geq 0\right) 
    &\leq \exp\left(\frac{-\tilde{i}\sqrt{f}}{A}\right) + \sum_{u=t'_{j-1}}^{t'_j}\mathbb{P}(\bar{\mathcal{A}}_1)+ \sum_{u=t'_{j-1}}^{t'_j}\mathbb{P}(\bar{\mathcal{A}}_2) \\
    &\leq A''(t'_j-t'_{j-1})\exp\left(\frac{-\tilde{i}\sqrt{f}}{A''}\right),
\end{aligned}
\end{equation}
where $ A'' > 0$ is a large constant. 
Since
$t_{\tilde{k}}^{\tilde{i}} > \ell_n^{11/12}$:
\begin{equation*}
\begin{aligned}
\ell_n^{11/12} &\leq t +  \frac{(\tilde{i}^2-1)f\ell_n^{2/3}}{C} + \frac{k\ell_n^{2/3}}{C\tilde{i}^2f} \\
&\leq \frac{3\tilde{i}^2f\ell_n^{2/3}}{C},
\end{aligned}
\end{equation*}
equation \ref{ineq999} yields for $n$ large enough:
\begin{equation*}
\begin{aligned}
    \sum_{j=1}^{\bar{j}}\mathbb{P}\left(\sup_{ t'_{j-1} \leq u \leq t'_j}\,  \left(L^{0}_{u+\frac{\ell_n^{2/3}}{Ci^2f}}-L^{0}_u\right)\geq 0\right) 
    &\leq A''(t'_{\bar{j}}-t'_{0})\exp\left(\frac{-\tilde{i}\sqrt{f}}{A''}\right), \\
    &\leq A\exp\left(\frac{-\tilde{i}\sqrt{f}}{A}\right),
\end{aligned}
\end{equation*}
where $A > 0$ is a large constant.
This finishes the proof of the induction of Equation \eqref{recurs}.
By that same equation  we obtain for $n$ and $f$ large enough:
\hfill
\begin{equation*}
\begin{aligned}
\mathbb{P}\left(\cup_{(u,v) \leq (\tilde{i},n)} E^u_v\cup \bar{E}\right) &\leq A\exp\left(\frac{-\sqrt{f}}{A}\right) + A\sum_{i=1}^{\tilde{i}}(i+1)^2\exp\left(\frac{-i\sqrt{f}}{A}\right) + An\exp\left(\frac{-\tilde{i}\sqrt{f}}{A}\right) \\
&\leq A\exp\left(\frac{-\sqrt{f}}{A}\right) + A\sum_{i=1}^{\infty}(i+1)^2\exp\left(\frac{-i\sqrt{f}}{A}\right) + A'n\exp\left(\frac{-n^{1/8}}{A'}\right) \\
&\leq A''\exp\left(\frac{-\sqrt{f}}{A''}\right)+A''\exp\left(\frac{-n^{1/8}}{A''}\right).
\end{aligned}
\hfill
\end{equation*}
\end{proof}

This theorem shows that, after exploring the largest connected component, we discover small connected components that  become smaller and smaller the further the exploration process goes.
From that, one can get multiple corollaries. A first one is that the total weights of the components also gets smaller and smaller. The proof is the same as that of Theorem \ref{Big_weight} and is omitted. \par
\begin{Corollary}
\label{chrenn}
Suppose that Conditions \ref{cond_nodes} hold. There exists a constant $A > 0$ such that the following holds:\par
For any $\epsilon >0$, the probability that there exists an $i\geq 0$ and $\bar{k}_i \geq k \geq 0$, such that a connected component discovered between times $t^{i}_k-t+i^*$ and $t^{i}_{k+1}-t+i^*$ (or times  $t^{i}_{\bar{k}_i}-t+i^*$ and $t^{i+1}_{0}-t+i^*$) in the exploration process has total weight larger than $(1+\epsilon)(t^{i}_{k+1}-t^{i}_{k})$ (or $(1+\epsilon)(t^{i+1}_{0}-t^{i}_{\bar{k}_i})$), where $i^* \in \mathbb{N}$ is the time when the exploration of $H^*_f$ ends, is at most:
\begin{equation*}
A\exp\left(\frac{-\sqrt{f}}{A}\right)+A\exp\left(\frac{-n^{1/8}}{A}\right).
\end{equation*}
\end{Corollary}

Another fact we can deduce from Theorem \ref{localize} is the following convergence in probability. Its proof is straightforward from Theorems \ref{lower_bound_1} and \ref{localize}.
\begin{Corollary}
\label{conv_cor}
Recall that $f = f(n)$ is such that $f(n) = o(n^{1/3})$. Suppose that $\lim\limits_{n\rightarrow \infty} f(n) = +\infty$. Let $(|C_1|,|C_2|,|C_3|,...)$ denote the sequence of sizes of the connected components of $G(n,\textbf{W},p_{f(n)})$ taken in decreasing order, with the convention $|C_i| = 0$ if there is no $i$-th largest component. We have the following convergence in probability for any $p > 7/3$ as $n \rightarrow \infty$:
$$
\left(\frac{|C_1|}{2f(n)\ell_n^{2/3}},\frac{|C_2|}{\ell_n^{2/3}},\frac{|C_3|}{\ell_n^{2/3}},\frac{|C_4|}{\ell_n^{2/3}},...\right) \xrightarrow{\mathbbm{p}} (C,0,0,..),
$$
in $\ell^p$, the usual $p$ norm.
\end{Corollary}
\begin{proof}
By Theorem  \ref{lower_bound_1}, for any $1 > \epsilon' >0$ there exists a constant $A > 0$ such that for $n$ large enough:
\begin{equation}
\label{salit}
\mathbb{P}\left(\middle|\left(\frac{|C_1|}{2f(n)\ell_n^{2/3}}-C\right)^p \middle|    \geq (3\epsilon')^p \right) \leq A\exp\left(\frac{-f(n)^{1/2}}{A}\right).
\end{equation}
Recall the definition of $\tilde{i}$ and let 
$$
\epsilon(f(n))= \frac{1}{\sqrt{f(n)}^p}+\frac{1}{(Cf(n))^{p-1}}\sum_{i=1}^{\tilde{i}-1}\frac{1}{i^{2p-3}}+\frac{\ell_n^{1/3}}{(\tilde{i}^2f)^{p-1}}.
$$
We know that for any $(x_1,x_2,...x_k)$ which are positive numbers:
$$
\sum_{u=1}^kx_u^p \leq \left(\sum_{u=1}^kx_u\right)^p.
$$
We showed in the end of the proof of Theorem \ref{localize} that $\ell_n^{1/4} = O(\tilde{i}^2f)$, this yields $\lim_{n}\epsilon(f(n))=0$.
Using those remarks alongside Theorems \ref{localize} and  \ref{lower_bound_1}, there exists a constant $A > 0$ such that:
\begin{equation}
\begin{aligned}
\label{salitt}
\mathbb{P}\left(\sum_{k\geq 2}\middle|\left(\frac{|C_k|}{\ell_n^{2/3}}\right)^p\middle| \geq A\epsilon(f(n)) \right) \leq A\exp\left(\frac{-\sqrt{f(n)}}{A}\right)+A\exp\left(\frac{-n^{1/8}}{A}\right).
\end{aligned}
\end{equation}
The corollary follows by the union bound Inequalities \eqref{salit} and \eqref{salitt}
\end{proof}
\textbf{Note} : If we change $\ell_n^{11/12}$ to $\ell_n^{1-\epsilon''}$ for $1/3 > \epsilon'' > 0$ arbitrarily small in the definition of $t^{\tilde{i}}_{\tilde{k}}$ then Theorem \ref{localize} will hold with the term $n^{1/8}$ being replaced by $n^{\frac{-1+3\epsilon''}{6}}$. And this shows that Corollary \ref{conv_cor} holds in fact for any $ p > 2$. Moreover,
with the same technique one can also obtain the same convergence for the sequence of weights of the connected components of $G(\textbf{W},p_{f(n)})$. It is also easy to show that if $f(n)$ is of order $n^{\epsilon}$ for some $\epsilon >0$ then this convergence will hold in expectation for any moment larger than $1$.

\subsection{The excess of the tail}
We showed that after discovering the giant component all the other components have size less than $\ell_n^{2/3}/f$ with high probability. We call excess of a discrete interval between $1$ and $n$, the number of excess edges discovered in that interval of time during the exploration process, regardless of which connected component they belong to. In the following theorem we will first focus on getting bounds on the excess of small intervals, then getting bounds on the excess of the tail will be straightforward by using Theorem \ref{localize}.

\begin{Theorem}
\label{small_exc_2}
Suppose that Conditions \ref{cond_nodes} hold. There exists a constant $A > 0$ such that the following is true: \par
  For $\tilde{i} \geq i \geq 1$, for  $\bar{k}_i \geq k \geq 0$ let $\textnormal{Exc}^i_k$ be the excess of the interval $[t^i_{k},t^i_{k+1})$.  For any $\epsilon >0$:
\begin{equation*}
\begin{aligned}
\mathbb{P}\left(\sup_{k_i > k \geq 0} \, (\textnormal{Exc}^i_k) \geq f^{\epsilon}\right) \leq  &A\exp\left(\frac{-f^{\epsilon}\ln(i\sqrt{f})}{A}\right)+A\exp\left(\frac{-i\sqrt{f}}{A}\right) + A\exp\left(\frac{-\sqrt{f}}{A}\right)\\
&+A\exp\left(\frac{-n^{1/8}}{A}\right).
\end{aligned}
\end{equation*}

\end{Theorem}
\begin{proof}
Let $k < k_i$. If $t^i_{k} \leq \ell_n^{11/12}$, by Theorem \ref{concentration_L}:
\begin{equation}
\label{eq1000}
\mathbb{P}\left(\sup_{t^i_{k-1} \leq u \leq w \leq t^i_{k+1}}( L^{0}_w-L^{0}_u- \mathbb{E}[ L^{0}_w-L^{0}_u ]) \geq \ell_n^{1/3}\right) \leq A\exp\left(\frac{-i\sqrt{f}}{A'}\right).
\end{equation}
By  Corollary \ref{S_h}, for any $t^i_{k-1} \leq u \leq w \leq t^i_{k+1}$: 
\begin{equation*}
\begin{aligned}
\mathbb{E}[L^{0}_{w}-L^{0}_{u}] \leq 0.
\end{aligned}
\end{equation*}
With the above inequality, Equation \ref{eq1000} yields:
\begin{equation}
\label{lem101}
\mathbb{P}\left(\sup_{t^i_{k-1} \leq u \leq w \leq t^i_{k+1}}( L^{0}_w-L^{0}_u) \geq \ell_n^{1/3}\right) \leq A\exp\left(\frac{-i\sqrt{f}}{A}\right).
\end{equation}
And in fact, notice that this inequality also holds for $t^i_k > \ell_n^{11/12}$ by the method used to obtain Inequality \eqref{ineq999}.
Denote the event "no connected component discovered after time $t^i_0$ has size larger $\frac{\ell_n^{2/3}}{i^2fC}$" by $\mathcal{G}$.    When $\mathcal{G}$ holds, $L$ visits $0$ in any interval of size $\frac{\ell_n^{2/3}}{i^2fC}$ after $t^i_0$. In that case: 
$$\sup\limits_{t^i_{k} \leq r \leq t^i_{k+1}}L(r) \leq \sup\limits_{t^i_{k-1} \leq u\leq w \leq t^i_{k+1}} (L^{0}_{w}-L^{0}_{u}).$$ 
This fact and Equation \eqref{lem101} yield:
\begin{equation}
\label{lem102}
\mathbb{P}\left(\sup_{t^i_{k} \leq r\leq t^i_{k+1}}L_r \geq \ell_n^{1/3}\right) \leq A'\exp\left(\frac{-i\sqrt{f}}{A'}\right)+\mathbb{P}(\bar{\mathcal{G}}).
\end{equation}

Let $\mathcal{M}=\left\{\sup_{t^i_{k} \leq r\leq t^{i}_{k+1}}L_r \leq \ell_n^{1/3}\right\}$. By Equation \eqref{lem102} and Theorem \ref{localize} we obtain:
\begin{equation}
\begin{aligned}
\label{sousouu}
\mathbb{P}\left(\bar{\mathcal{M}}\right) &\leq 
A\exp\left(\frac{-\sqrt{f}}{A}\right)+A\exp\left(\frac{-n^{1/8}}{A}\right)+A'\exp\left(\frac{-i\sqrt{f}}{A'}\right).
\end{aligned}
\end{equation}
By the union bound:
\begin{equation}
\begin{aligned}
\label{bez_1u}
\mathbb{P}\left(\textnormal{Exc}^i_k \geq l + \mathbb{E}[\textnormal{Exc}^i_k] \right) \leq \mathbb{P}\left(\textnormal{Exc}^i_k \geq l + \mathbb{E}[\textnormal{Exc}^i_k] , \, \mathcal{M}\right) + \mathbb{P}(\bar{\mathcal{M}}).
\end{aligned}
\end{equation}
Now we use the same method we used in Lemma \ref{exc_big}.
Let $R = \ell_n^{1/3}$ and define $\tilde{t} = t^i_{k+1}-t^i_{k}$.
By Lemma \ref{mean_pow_weights_2}:
\begin{equation}
\begin{aligned}
\label{suu_1'u}
\mathbb{E}\left[p_f\sum\limits_{r= \frac{t^i_{k-1}}{R}}^{ \frac{t^i_{k}}{R}} 2R \left( \sum\limits_{u=rR+1}^{(r+2)R}{w_{v(t^i_{k-1})}}^2\right)\right] 
&\leq A\tilde{t}Rp_f,
\end{aligned}
\end{equation}
Hence, by Equation \eqref{suu_1'u} and Theorem \ref{concentration4}: 
\begin{equation}
\begin{aligned}
\label{eq87}
 &\mathbb{P}\left(\sum\limits_{r=t^i_{k-1}}^{t^i_{k}} \sum\limits_{u=r+1}^{(R+r)}w_{v(u)}w_{v(r)}p_f  \geq 2A\tilde{t}Rp_f + \frac{1}{i\sqrt{f}} \right) \\
&\leq 
\mathbb{P}\left( p_f\sum\limits_{r= \frac{t^i_{k-1}}{R}}^{ \frac{t^i_{k}}{R}} \left(\sum\limits_{u=rR+1}^{(r+2)R}w_{v(u)}\right)^2 \geq 2A\tilde{t}Rp_f + \frac{1}{i\sqrt{f}} \right) \\
&\leq \mathbb{P}\left( \sum\limits_{r= \frac{t^i_{k-1}}{R}}^{ \frac{t^i_{k}}{R}}  \left( \sum\limits_{u=rR+1}^{(r+2)R}w_{v(u)}^2 \right)\geq A\tilde{t} + \frac{1}{2i\sqrt{f}Rp_f}\right)\\
&\leq A''\exp\left(\frac{-i\sqrt{f}}{A''}\right).
\end{aligned}
\end{equation}
By the union bound between Equation  \eqref{suu_1'u} and Equation \eqref{sousou}:
\begin{equation}
\begin{aligned}
\label{suu_3'u}
 &\mathbb{P}\left(\sum\limits_{r=t^i_{k-1}}^{t^i_{k}} \sum\limits_{u=r+1}^{(L_r+i)}w_{v(r)}w_{v(u)}p_f  \geq 2A\tilde{t}Rp_f+\frac{1}{i\sqrt{f}}  \right) \\
&\leq  \mathbb{P}\left(\sum\limits_{r=t^i_{k-1}}^{t^i_{k}} \sum\limits_{u=r+1}^{(R+i)}w_{v(r)}w_{v(u)}p_f  \geq 2A\tilde{t}Rp_f+\frac{1}{i\sqrt{f}}  \right)+\mathbb{P}(\bar{\mathcal{M}})  \\
&\leq A\exp\left(\frac{-\sqrt{f}}{A}\right)+A\exp\left(\frac{-n^{1/8}}{A}\right)+A'\exp\left(\frac{-i\sqrt{f}}{A'}\right).
\end{aligned}
\end{equation}
We know that  for any $\epsilon >0$:
\begin{equation}
\begin{aligned}
\label{imp_1}
\mathbb{P}(\textnormal{Exc}^i_k \geq f^{\epsilon} | \mathcal{F}_n) &\leq 
\mathbb{P}\left(\sum\limits_{r=t^i_{k-1}}^{t^i_k}\sum\limits_{u=r+1}^{(L_r+r-1)}Y(v(r),v(u)) \geq f^{\epsilon}  \middle | \mathcal{F}_n\right)\mathbbm{1}(\mathcal{M}) + \mathbbm{1}({\bar{\mathcal{M}}}).
\end{aligned}
\end{equation}
Since we are dealing with a sum of Bernoulli random variables, this sum is larger than $f^{\epsilon}$ if and only if there are more than $f^{\epsilon}$ Bernoulli variables equal to $1$. Let $S$ be the random set of subsets of size $f^{\epsilon}$ (suppose that $f^{\epsilon}$  is an integer for simplicity) composed of couples $(r,u)$ that appear as indices in the sum in the right-hand side of Equation \eqref{imp_1}, and let $S'$ be the deterministic  set of subsets of size $f^{\epsilon}$ composed of couples $(r,u)$ that appear as indices in the sum in the right-hand side of Equation \eqref{imp_1} when we replace $L_u$ by $R$ for all $t^i_{k-1} \leq r \leq t^i_k$. Then for $f$ large enough:
\begin{equation*}
\begin{aligned}
\mathbb{P}\left(\sum\limits_{r=t^i_{k-1}}^{t^i_k}\sum\limits_{u=r+1}^{(L_u+r-1)}Y(v(r),v(u)) \geq f^{\epsilon}\middle | \mathcal{F}_n\right)\mathbbm{1}({\mathcal{M}})  &= \mathbb{P}\left(\bigcup_{M \in S}\bigcap_{(r,u) \in M} \{Y(v(r),v(u))=1\} \middle | \mathcal{F}_n \right)\mathbbm{1}({\mathcal{M}}) \\
&\leq \sum_{M\in S'}\prod_{(r,u) \in M}\left(1-e^{-w_{v(r)}w_{v(u)}p_f}\right) \\
&\leq \sum_{M\in S'}\prod_{(r,u) \in M}\left(w_{v(r)}w_{v(u)}p_f\right) \\
&\leq \left(\sum\limits_{r=t_{k-1}}^{t_{k}} \sum\limits_{u=r+1}^{(R+r)}w_{v(r)}w_{v(u)}p_f \right)^{f^{\epsilon}+1}.
\end{aligned}
\end{equation*}
By this fact and Equation \eqref{suu_3'u}:
\begin{equation}
\begin{aligned}
\label{almost_f}
\mathbb{P}(\textnormal{Exc}^i_k \geq f^{\epsilon}) &\leq \left(A\tilde{t}Rp_f + \frac{1}{i\sqrt{f}} \right)^{f^{\epsilon}+1} + A'\exp\left(\frac{-i\sqrt{f}}{A'}\right)+A\exp\left(\frac{-n^{1/8}}{A}\right)+A'\exp\left(\frac{-\sqrt{f}}{A'}\right)
\\ &\leq \left(\frac{A''}{i^2f} + \frac{1}{i\sqrt{f}} \right)^{f^{\epsilon}+1} + A'\exp\left(\frac{-i\sqrt{f}}{A'}\right)+A\exp\left(\frac{-n^{1/8}}{A}\right)+A'\exp\left(\frac{-\sqrt{f}}{A'}\right) \\
&\leq \exp\left((f^{\epsilon}+1)\left(\ln\left(\frac{A''}{i^2f}\right)+\ln\left(1+\frac{i\sqrt{f}}{A''}\right)\right)\right)+ A'\exp\left(\frac{-i\sqrt{f}}{A'}\right)\\
&+A\exp\left(\frac{-n^{1/8}}{A}\right)+A'\exp\left(\frac{-\sqrt{f}}{A'}\right)\\
&\leq \exp\left(\frac{-f^{\epsilon}\ln(i\sqrt{f})}{A'''}\right)+ A'\exp\left(\frac{-i\sqrt{f}}{A'}\right)+A\exp\left(\frac{-n^{1/8}}{A}\right)+A'\exp\left(\frac{-\sqrt{f}}{A'}\right).
\end{aligned}
\end{equation}
If $t^i_{k} \geq \ell_n^{11/12}$, then by definition $i = \tilde{i}$. And we obtain similarly:
\begin{equation*}
\begin{aligned}
\mathbb{P}(\textnormal{Exc}^{\tilde{i}}_k \geq f^{\epsilon}) &\leq A\exp\left(\frac{-f^{\epsilon}\ln(\tilde{i}\sqrt{f})}{A}\right)+A\exp\left(\frac{-\tilde{i}\sqrt{f}}{A}\right) + A\exp\left(\frac{-\sqrt{f}}{A}\right)+A\exp\left(\frac{-n^{1/8}}{A}\right).
\end{aligned}
\end{equation*}
This finishes the proof.
\end{proof}
In Theorem \ref{small_exc_2} the term $A\exp\left(\frac{-\sqrt{f}}{A}\right)$ comes from applying Theorem \ref{localize}, and that theorem gives a bound for all the connected components discovered after the giant connected component. Using this remark, we can sum over $i$. And using simple computations, we obtain the concentration of the total surplus of the tail.
\begin{Theorem}
\label{excess_small}
Suppose that Conditions \ref{cond_nodes} hold. There exists $A>0$, such that for any $\epsilon>0$, for $f$ and $n$ large enough, the probability that a connected component discovered after $H^*_f$ has excess more than $f^{\epsilon}$ is at most:
$$
A\exp\left(\frac{-f^{\epsilon}\ln(\sqrt{f})}{A}\right) + A\exp\left(\frac{-\sqrt{f}}{A}\right)+A\exp\left(\frac{-n^{1/8}}{A}\right).
$$
\end{Theorem}
As a Corollary of the work done here we obtain a natural global upper bound on $L$.
\begin{Corollary}
\label{corr36}
Suppose that Conditions \ref{cond_nodes} hold. There exists a constant $A > 0$ large enough, such that:
\begin{equation*}
\begin{aligned}
\mathbb{P}\left(\sup_{t^1_{0} \leq l \leq n}( L_l) \geq \ell_n^{1/3} \right) &\leq A\exp\left(\frac{-\sqrt{f}}{A}\right)+A\exp\left(\frac{-n^{1/8}}{A}\right).
\end{aligned}
\end{equation*}
\end{Corollary}
\begin{proof}
Let $1 \leq i \leq \tilde{i}$, and denote the event "no connected component discovered after time $t^i_0$ has size larger $\frac{\ell_n^{2/3}}{i^2fC}$" by $G_i$. when $G_i$ holds, $L$ visits $0$ in any interval of size $\frac{\ell_n^{2/3}}{i^2fC}$ after $t^i_0$. In that case: 
$$\sup\limits_{t^i_{k} \leq r \leq t^i_{k+1}}L_r \leq \sup\limits_{t^i_{k-1} \leq u\leq w \leq t^i_{k+1}} (L^{0}_{w}-L^{0}_{u}).$$ 
Moreover, by Equation \eqref{lem101}:
\begin{equation*}
\mathbb{P}\left(\sup_{t^i_{k-1} \leq u\leq w   \leq t^i_{k+1}}( L^{0}_w-L^{0}_u) \geq \ell_n^{1/3}\right) \leq A\exp\left(\frac{-i\sqrt{f}}{A}\right),
\end{equation*}
with $A > 0$ a large constant independent of $i$. By summing this equation over $1 \leq k < \bar{k}_i-1$ for every $i$, and then over $1 \leq i \leq \tilde{i}$ we obtain directly:
\begin{equation}
\label{eq9898}
\mathbb{P}\left(\sup_{t^1_{0} \leq r \leq n}( L_r) \geq \ell_n^{1/3}, \, \cap_{i \leq \tilde{i}} G_i \right) \leq A'\exp\left(\frac{-\sqrt{f}}{A'}\right)+A\exp\left(\frac{-n^{1/8}}{A}\right).
\end{equation}
With $A' >0$ a large constant. By Theorem \ref{localize} there exists a large constant $A >0$ such that:
\begin{equation}
\label{eq9897}
\mathbb{P}(\cup_{i \leq \tilde{i}} \bar{G}_i) \leq A\exp\left(\frac{-\sqrt{f}}{A}\right)+A\exp\left(\frac{-n^{1/8}}{A}\right).
\end{equation}
By Equations \eqref{eq9898} and \eqref{eq9897} there exists a large constant $A>0$ such that:
\begin{equation*}
\begin{aligned}
\mathbb{P}\left(\sup_{t^1_{0} \leq r \leq n}( L_r) \geq \ell_n^{1/3} \right) &\leq  \mathbb{P}(\cup_{i \leq \tilde{i}} G_i) 
+ \mathbb{P}\left(\sup_{t^1_{0} \leq r \leq n}( L_r) \geq \ell_n^{1/3}, \, \cap_{i \leq \tilde{i}} G_i \right)  \\ &\leq A\exp\left(\frac{-\sqrt{f}}{A}\right)+A\exp\left(\frac{-n^{1/8}}{A}\right),
\end{aligned}
\end{equation*}
which finishes the proof.
\end{proof}
This upper bound alongside Theorem \ref{sup_l} gives an upper bound for the whole process $L$. However, it can be refined, and it is not hard to show that $L$ gets smaller the further we advance in the exploration.  We elect to stop here and as a last result we use this upper bound on $L$ and the theorems we showed in this paper to give an upper bound on the number of connected components discovered in parts of the exploration of the graph.
\begin{Corollary}
\label{comp_num}
Suppose that Conditions \ref{cond_nodes} hold. Recall that $i^* \in \mathbb{N}$ is the time at which the exploration of $H^*_f$ ends. There exists a constant $A > 0$ such that the following is true:\\
The probability that there exists an $\tilde{i} > i\geq 0$ and $\bar{k}_i > k \geq 0$, such that the number of connected components discovered between times $t^{i}_0-t+i^*$ and time $t^{i}_{\bar{k}}-t+i^*$, is more than $100i^3f^2\ell_n^{1/3}$, is at most:
\begin{equation*}
A\exp\left(\frac{-\sqrt{f}}{A}\right)+A\exp\left(\frac{-n^{1/8}}{A}\right).
\end{equation*}
\end{Corollary}
\begin{proof}
Let $r = \frac{2f\ell_n^{2/3}}{C}$, $t_1 = t^i_0$, and  $t_2 = t^i_{\bar{k}_i}+r$. \\
In order to prove this theorem we need to bound the number of times a new minima of $L'$ is reached in the interval $[t_1,t_2]$. Since $L'$ can only go down by $1$, the number of new minimums created in the interval  $[t_1,t_2]$ is smaller than $$\inf_{t_1 \leq l \leq m \leq t_2} L'_m-L'_l.$$
Choose $x = -50i^3f^2\ell_n^{1/3}$ and. Then $(t_2,t_1,0,x)$ verifies Conditions \ref{cond_3}. Hence, by Theorem \ref{inf_s_1} we have:
\begin{equation}
\begin{aligned}
\label{manal1}
\mathbb{P}\left(\sup_{t_1 \leq u \leq w \leq t_2}\left|L'_w-L'_u-\tilde{L}_w-\tilde{L}_u\right| \geq -x \right)\leq  &A\exp\left(\frac{-x^2}{A(xn^{1/3}+(t_2-t_1))}\right) \\
\leq &A'\exp\left(\frac{-i^3f^2}{A'}\right).
\end{aligned}
\end{equation}
For any $h > 0$, if $L_k < h$ for any $k \leq t_2$ then deterministically $\tilde{L}_m-\tilde{L}_l \geq \tilde{L}^h_m-\tilde{L}^h_l$ for any $1 \leq l \leq m \leq t_2$. \par
Hence, if $\tilde{L}_m-\tilde{L}_l \leq x$ for some $t_1 \leq l \leq m \leq t_2$ then one of the following events happens :
\begin{itemize}
\item There exists $ 0 \leq j \leq t_2$ such that $L_j \geq h$.
\item There exists $ t_1 \leq l \leq m \leq t_2$  such that $\tilde{L}^h_m-\tilde{L}^h_l \leq L'_m-L'_l \leq x$.
\end{itemize} 
Let $h=\frac{10f^2\ell_n^{1/3}}{C}$. 
Then for the first event, by Theorem \ref{sup_l} and Corollary \ref{corr36} : 
\begin{equation}
\label{eq612}
\mathbb{P}\left(\sup_{1 \leq j \leq t_2} L_j \geq h\right) \leq A\exp\left(\frac{-\sqrt{f}}{A}\right)+A\exp\left(\frac{-n^{1/8}}{A}\right).
\end{equation}

For the second event,  Conditions \ref{cond_3} are verified for $(t_2,t_1,h,-x)$. By Corollary \ref{S_h} and a quick computation, for any $t_2 \geq w \geq u \geq t_1$, we have 
$$ -\mathbb{E}[\tilde{L}^h_w-\tilde{L}^h_u] \leq -\mathbb{E}[\tilde{L}^h_{t_2}-\tilde{L}^h_{t_1}] \leq x/2.$$
 We can thus apply Lemma \ref{inf_s_0} to obtain:
\begin{equation}
\begin{aligned}
\label{thendbro}
\mathbb{P}\left(\inf_{t_1 \leq u \leq w \leq t_2}\tilde{L}^h_w-\tilde{L}^h_u \leq x\right) \leq &\mathbb{P}\left(\inf_{t_1 \leq u \leq w \leq t_2}\tilde{L}^h_w-\tilde{L}^h_u-\mathbb{E}[\tilde{L}^h_w-\tilde{L}^h_u] \leq \frac{x}{2}\right)  \\ \leq &A\exp\left(\frac{-x^2}{A(xn^{1/3}+(t_2-t_1))}\right) \\
\leq &A'\exp\left(\frac{-i^3f^2}{A'}\right),
\end{aligned}
\end{equation}
with $A' > 0$ a large constant that does not depend on $i$. \par
Recall that $i^* \in \mathbb{N}$ is the time at which the exploration of $H^*_f$ ends. 
By Theorem \ref{lower_bound_1}:
\begin{equation}
\label{thendbro3}
\mathbb{P}\left(\frac{3f\ell_n^{2/3}}{C} \geq i^* \geq \frac{f\ell_n^{2/3}}{C}\right) \geq 1-A\exp\left(\frac{-\sqrt{f}}{A}\right).
\end{equation}
When this event holds, we have $[t^i_0-t+i^*,t^i_{\bar{k}_i}-t+i^*] \subset [t_1,t_2]$. Hence,  summing Equations \eqref{thendbro} and \eqref{manal1} for $\tilde{i} > i \geq 1$, and using the union bound with Equation \eqref{eq612} and \eqref{thendbro3} finishes the proof. 
\end{proof}
\newpage
\renewcommand\bibname{Bibliography}
\bibliographystyle{abbrvnat} 
\bibliography{bibli.bib}
\end{document}